\documentclass{article}
\usepackage[utf8]{inputenc}
\usepackage{xcolor}
\usepackage{graphicx} 
\graphicspath{ {./img/} }
\usepackage{amsmath,amsthm,amssymb,amsfonts,hyperref,mathtools,enumitem,bm,mathrsfs}
\usepackage[normalem]{ulem}
\usepackage{cite}
\title{The gap between a variational problem and its occupation measure relaxation}
\author{Milan Korda$^{1,2}$ and Rodolfo R\'ios-Zertuche$^{2,3}$}
\date{}

\newcommand{\R}{\mathbb{R}}
\newtheorem{theorem}{Theorem}[section]
\newtheorem{lemma}[theorem]{Lemma}

\theoremstyle{definition}
\newtheorem{definition}[theorem]{Definition}
\newtheorem{example}[theorem]{Example}
\theoremstyle{remark}
\newtheorem{remark}[theorem]{Remark}

\newcommand{\denomE}{41}
\newcommand{\denomxi}{10}
\newcommand{\denomA}{4}
\newcommand{\proj}{\pi}
\newcommand{\projXY}{\proj_{\Omega\times Y}}
\newcommand{\projX}{\proj_{\Omega}}
\newcommand{\projXYZ}{\proj_{\Omega\times Y\times Z}}
\newcommand{\projdXY}{\proj_{\partial\Omega\times Y}}

\usepackage{cancel}
\usepackage{empheq}

\usepackage[most]{tcolorbox}
\newtcolorbox{mymathbox}[1][]{colback=white, sharp corners, #1}

\newcommand{\uu}{u}

\begin{document}

\maketitle

\footnotetext[1]{CNRS; LAAS; 7 avenue du colonel Roche, F-31400 Toulouse; France. }
\footnotetext[2]{Faculty of Electrical Engineering, Czech Technical University in Prague,
Technick\'a 2, CZ-16626 Prague, Czech Republic.}
\footnotetext[3]{Department of Mathematics and Statistics, UiT The Arctic University of Norway.}
\renewcommand{\thefootnote}{\fnsymbol{footnote}} 
\footnotetext{\emph{MSC classes.} 35Q93 (Primary), 49Q15, 26B40, 65M99 (Secondary)}     
\renewcommand{\thefootnote}{\arabic{footnote}}

\begin{abstract}
 Recent works have proposed linear programming relaxations of variational optimization problems subject to nonlinear PDE constraints based on the occupation measure formalism. The main appeal of these methods is the fact that they rely on convex optimization, typically semidefinite programming. In this work we close an open question related to this approach. We prove that the classical and relaxed minima coincide when the dimension of the codomain of the unknown function equals one, both for calculus of variations and for optimal control problems, thereby complementing analogous results that existed for the case when the dimension of the domain equals one. In order to do so, we prove that, when the dimension of the codomain equals one, a relaxed occupation measure can be decomposed as a superposition of functions in Sobolev space $W^{1,\infty}(\Omega)$. 
 We also show by means of a counterexample that, if both the dimensions of the domain and of the codomain are greater than one, there may be a positive gap. The example we construct to show the latter serves also to show that sometimes relaxed occupation measures may represent a more conceptually-satisfactory ``solution'' than their classical counterparts, so that ---even though they may not be equivalent--- algorithms rendering accessible the minimum in the larger space of relaxed occupation measures remain extremely valuable. Finally, we show that in the presence of integral constraints, a positive gap may occur at any dimension of the domain and of the codomain.
\end{abstract}

\tableofcontents

\section{Introduction}
\label{sec:intro}
The occupation measure relaxation replaces a nonconvex variational problem by an infinite-dimensional linear program (LP)  in the space of Radon measures defined on subsets of a Euclidean space. A distinguishing feature of this relaxation is the fact that the LP can be approximated by a hierarchy of tractable convex semidefinite programming problems (referred to as the moment-sum-of-squares or Lasserre hierarchy~\cite{lasserre2009moments}) that provably converge to the solution to the LP under mild assumptions, thereby providing an appealing mesh-free numerical scheme for this class of problems~\cite{korda2018moments}%
 . 
 This work is concerned with the fundamental question of the equivalence of the original variational problem and this relaxation, which justifies the  use of the method. This equivalence is referred to as the absence of a \emph{relaxation gap}.

In its full generality, the variational problem  considered here is subject to constraints in the form of first-order nonlinear partial differential equations and inequalities, and therefore allows one to model a large number of problems from engineering, material sciences, biology, and other domains, and can be viewed as an instance of PDE constrained optimization~\cite{hinze2008optimization}. As such, it also encodes a wealth of PDEs: not only can first-order nonlinear PDEs be recovered by imposing them as constraints, but in fact many other equations arise indirectly as the Euler-Lagrange equations of the variational problem at hand. These include for example elliptic equations \cite[Chapter 4]{le2016partial}, the heat equation \cite[Section 2.1]{armstrong2018quantitative}, the Einstein equations of relativity \cite[Chapter 3]{ashtekartate}, the Schrödinger equation \cite{kobe2006lagrangian,XU2006627}, and the equations of fluid dynamics \cite{brenier1989least,benamou2000computational%
}. Some other important equations can be obtained as Euler-Lagrange equations of a simple generalization of the method considered here, in which the Lagrangian is allowed to involve also the second-order derivatives, for example,  the Korteweg–de Vries equation \cite{HE2004847} and the wave equation \cite{HE2013158}. Apart from these problems, there are many others that can be modelled using the variational framework discussed here, including optimal control~\cite{lasserre2008nonlinear} and optimal transport problems \cite{bernard2007optimal}, as well as the asymptotic dynamics of differential inclusions \cite{bianchirz} as well as finding optimal Poincare constants~\cite{chernyavsky2021convex}. The method of \cite{korda2018moments} can be used to attack all of these problems, whence the motivation for the resolution of the question considered in this paper, which has been at the center of the analysis of this method since its inception.

In this section we present a simplified version of the problem omitting constraints and boundary terms, introduce the convex relaxation and discuss the rich history of the question of equivalence to the original problem. The full version of the problem is treated in Section~\ref{sec:codim1}, with the main results being Theorem \ref{thm:decomposition} (superposition), Theorems \ref{thm:consolidated} and \ref{thm:nogap} (no gap in codimension one); these results are also stated in the context of optimal control in Section \ref{sec:oc}, where the main result is Theorem \ref{cor:oc}. An example with a relaxation gap in codimension greater than one is constructed in Section~\ref{sec:positive gap} with the main result being Theorem~\ref{thm:gap}. Additional examples, showing that there may be a relaxation gap when integral constraints are involved, are presented in Section \ref{sec:integralconstraints}.

\paragraph{A global optimization problem.}

Let $n,m>0$. Let $\Omega\subset\R^n$ be a bounded, connected, open set with piecewise $C^1$ boundary $\partial \Omega$, and  $Y=\R^m$ and $Z=\R^{n\times m}$. Let the Lagrangian density be a locally bounded, measurable function $L\colon \Omega\times Y\times Z\to\R$ that is convex in $z$. 

Let $W^{1,\infty}(\Omega;Y)$ denote the Sobolev space of Lipschitz functions. Observe that for a function $y\in W^{1,\infty}(\Omega;Y)$, the dimension $n$ of the domain of $y$ and the dimension $m$ of its range are also, respectively, the dimension and codimension of the graph of $y$ in $\Omega\times Y$. Therefore, throughout this work we refer to $n$ as the dimension and $m$ as the codimension.

Using these data, consider the problem of determining, globally, the infimum of a possibly nonconvex functional:
\begin{align}\label{eq:pde}
 &\inf\limits_{y \in W^{1,\infty}(\Omega;Y)}  \displaystyle\int_{\Omega} L(x,y(x),D y(x))\, dx.
\end{align}
In \cite{korda2018moments}, it is proposed to attack this problem by first relaxing it to take the infimum over the space of relaxed occupation measures rather than over $W^{1,\infty}(\Omega;Y)$, as this relaxation is amenable --- at least when we have semialgebraic data $\Omega$ and $L$ --- to numerical solution through a hierarchy of finite-dimensional convex semidefinite programs, without resorting to spatio-temporal\footnote{In this context, the ``space'' variable corresponds to $y$ and the (possibly vector-valued) ``time'' variable corresponds to $x$.} discretization. The details of this semidefinite programming hierarchy are not the topic of this work; the reader is referred to~\cite{lasserre2009moments} for basic theory and to~\cite{henrion2020moment} for a number of applications. In this work we focus on the occupation measure relaxation of \eqref{eq:pde}, which we now explain in detail and give the necessary definitions to outline our results.%

\paragraph{Occupation measure relaxation.}
In order to introduce the concept of occupation measures, first observe that each function $y\in C^1(\overline\Omega)$ induces a measure $\mu_y$ on $\Omega\times Y\times Z$ by pushing forward 
the Lebesgue measure on $\Omega$ by the map $x\mapsto (x,y(x),Dy(x))$; in other words, for any measurable function $f\colon \Omega\times Y\times Z\to\R$ we have
\[\int_{\Omega\times Y\times Z}f\,d\mu_y=\int_\Omega f(x,y(x),Dy(x))\,dx.\]
The measure $\mu_y$ is the \emph{occupation measure} associated to the function $y$, and encodes $y$ and its derivative $Dy$.
For all compactly-supported test functions $\phi\in C_c^\infty(\Omega\times Y)$, applying the fundamental theorem of calculus to the function $x_\ell\mapsto \phi(x_1,\dots,x_\ell,\dots,x_n,y(x_1,\dots,x_\ell,\dots,x_n))$, we have
\[\int_\Omega\left[\frac{\partial\phi}{\partial x_\ell}(x,y(x))+\sum_{i=1}^m\frac{\partial\phi}{\partial y_i}(x,y(x))\frac{\partial y_i}{\partial x_\ell}(x)\right]dx=0,\quad \ell=1,\dots,n,
\]
as $\phi$ vanishes on the boundary $\partial\Omega$.
Thus $\mu_y$ satisfies
\[\int_\Omega\left[\frac{\partial\phi}{\partial x_\ell}(x,y)+\sum_{i=1}^m\frac{\partial\phi}{\partial y_i}(x,y)z_{i\ell}\right]d\mu_y(x,y,z)=0,\quad \ell=1,\dots,n,\]
for $\phi\in C_c^\infty(\Omega\times Y)$. This is the property we will use to obtain a %
larger set of measures in which we can still meaningfully consider the problem \eqref{eq:pde}.

Define the space $\mathcal M_0$ of \emph{relaxed occupation measures} to be the set of Radon measures $\mu$ on $\Omega\times Y\times Z$ satisfying, for all $\phi\in C^\infty_c(\Omega\times Y)$,
\begin{equation}\label{eq:closed}
 \int_{\Omega\times Y\times Z}\left[\frac{\partial\phi}{\partial x_\ell}(x,y)+\sum_{i=1}^m\frac{\partial\phi}{\partial y_i}(x,y)z_{i\ell}\right]d\mu(x,y,z)=0,\quad \ell=1,\dots,n,
\end{equation}
as well as
\begin{equation}\label{eq:finitemoments}
 \int_{\Omega\times Y\times Z} \|z\|\,d\mu(x,y,z)<+\infty.
\end{equation}
Then $\mathcal M_0$ contains all the occupation measures $\mu_y$ induced by $C^1$ functions $y$, as we noted above, so we have that the relaxed infimum
\begin{equation}\label{eq:relaxedpde}
 \inf_{\mu\in \mathcal M_0}\int_{\Omega\times Y\times Z}L(x,y,z)\,d\mu(x,y,z)
\end{equation}
is a lower bound of the original problem \eqref{eq:pde}. The advantage of \eqref{eq:relaxedpde} is that it is a linear programming problem, albeit infinite-dimensional, and it is possible to approximate it arbitrarily well using a hierarchy of semidefinite programming problems, at least when $\Omega$ and $L$ are semialgebraic \cite{korda2018moments}. 

 Occupation measures have a rich history within variational calculus ever since it was realized that they are useful in convexifying a large class of problems. %
 They are also closely related to gradient Young measures, which are families of probabilities encoding, for each $x$, a convex combination of vectors whose average equals the weak derivative $Dy(x)$ of a function $y$%
 . The moment-sum-of-squares hierarchy was partially applied to Young measures in several papers (see for example \cite{meziat2005analysis,meziat2008exact,meziat2010coarse}) before the appearance of the paper \cite{korda2018moments} that applied the moment-sum-of-squares hierarchy to the full occupation measure. Relaxations similar to the one used in the latter paper and studied here had also appeared before, for example in \cite{awi2014polyconvex}%
 \footnote{The ``measures reminiscent of Young's measures'' considered in \cite{awi2014polyconvex} are different to the relaxed occupation measures studied in this paper; roughly speaking, the former correspond to measures $\mu$ on $\Omega\times Y\times Z$ such that $(\projX)_{\#}\mu$ is 
the Lebesgue measure on $\Omega$, and, after disintegrating $\mu=\int_{\Omega}\mu_xdx$ with $\mu_x$ a probability on $Y\times Z$ for each $x\in\Omega$, and setting $\bar y(x)=\int y\,d\mu_x(y,z)$ and $\bar z(x)=\int z\,d\mu_x(y,z)$, these functions are required to verify that $\bar z$ is the distributional derivative of $\bar y$, that is, $\int_{\Omega}\sum_{i,j}\bar z_{ij}(x)\psi_{ij}(x)\,dx=-\int_{\Omega}\sum_{i,j}\bar y_j(x)\frac{\partial\psi_{ij}}{\partial x_i}(x)\,dx$ for all $\psi\in C^\infty_c(\Omega;\R^{n\times m})$. This condition, which corresponds to taking $\phi(x,y)=y_j\psi_{ij}(x)$ in \eqref{eq:closed} and adding over $i$, is weaker than \eqref{eq:closed}, so the set of these measures is considerably larger. Since they include representatives of discontinuous set-valued maps, we do not expect a no-gap result to be within reach in that context without imposing strong constraints in $L$.}%
 , but without a connection to moment-sum-of-squares hierarchies and in a specific setting with convex integrands as opposed to the general setting of 
 \cite{korda2018moments}. 

Throughout this history, however, the question of the equivalence of problems \eqref{eq:pde} and \eqref{eq:relaxedpde} has remained open in full generality and is the topic of this paper. 

\begin{example}\label{ex:doublewell} 
To give a simple %
example when a gap between~(\ref{eq:pde}) and (\ref{eq:relaxedpde}) may occur in the presence of additional constraints on $y(\cdot)$, consider $\Omega = [0,1]$, the double-well potential $L(x,y,z) = \min(|z-1|,|z+1|)$ and the constraint $y(x) = 0$ in $\Omega$. This constraint is modeled as a support constraint on $\mu$ in~(\ref{eq:relaxedpde}) in the form $\mathrm{supp}\,\mu \subset \{(x,y,z)\,:\, y = 0\}$. In this case, the only function $y \in W^{1,p}$, $p\in[1,\infty]$, feasible in~(\ref{eq:pde}) is $y = 0$, attaining the value $+1$ whereas the measure $\mu = dx\otimes \delta_0 \otimes (\frac{1}{2}\delta_{-1} + \frac{1}{2}\delta_{+1})$ attains the infimum of~(\ref{eq:relaxedpde}) equal to 0. 
\end{example}

The paper \cite{fantuzzi2022sharpness}, Example 7.1, also shows a relaxation gap in the presence of a double-well potential in $z$ in a situation that, unlike Example \ref{ex:doublewell}, no constraints in $y$ are imposed in the interior of the domain, but the dimension of both $\Omega$ and $Y$ is required to be at least 2. 

Both these examples have the property that $L$ is \emph{not convex} in $z$. We will see that this important property is sufficient for the absence of a relaxation gap if the dimension of either $\Omega$ or $Y$ is equal to one, although it \emph{may not suffice} if both of these dimensions are greater than one. In particular we will see that the infimum of~(\ref{eq:relaxedpde}) \emph{need not be equal} to the infimum in~(\ref{eq:pde}) even when $L$ is replaced by its convexification or quasiconvexification in $z$.

\paragraph{Contributions and previous work.}

It will perhaps come as no surprise that the question of the equivalence of problems \eqref{eq:pde} and \eqref{eq:relaxedpde} depends on the dimensions $n=\dim \Omega$ and $m=\dim Y$, since many related questions have been found to depend on these quantities, such as the regularity of minimal surfaces (see for example \cite{de2014regularity}) and  the possibility of generalization of the Frobenius theorem \cite{alberti2017some,schioppa2016unrectifiable}, among many other examples. Notice that $n$ is the dimension of the graph of a classical minimizer $y$, while $m$ is the codimension of this graph, which motivates our terminology below.

We distinguish three cases according to the dimension and the codimension of the graph of the decision variable $y(\cdot)$ in $\Omega\times Y$:
\begin{itemize}
 \item \emph{Dimension 1, that is, $n=\dim \Omega=1$ and any $m=\dim Y>0$.} In this case, \eqref{eq:pde} and \eqref{eq:relaxedpde} are equivalent. 
 
 The ideas behind this result originated in the seminal work of Young~\cite{young_lectures} (see also \cite{patrick}) but were to the best of our knowledge first proven by Rubio~\cite{rubio1975generalized,rubio1976extremal} and Lewis and Vinter~\cite{vinter1978equivalence,lewis1980relaxation}, and have since been clarified by a number of authors (see for example \cite{bernard2008young}). Computationally, this approach was used in conjunction with semidefinite-programming relaxations in~\cite{lasserre2008nonlinear} for optimal control as well as in~\cite{kordaROA} for region of attraction computation, proving a slight generalization of~\cite{vinter1978equivalence} using a superposition theorem from~\cite{ambrosio2008transport}. We remark that in those papers the equivalence has been proved%
 \footnote{
 To translate the general problem of the type we discuss in Section \ref{sec:oc}, which involves a Lagrangian density $L$ with dependency $L(x,y,y',u)$, into the type of problem discussed in \cite{bernard2008young}, involving a Lagrangian density $\bar L$  with dependency $\bar L(x,y,y')$,  one uses the definition in \eqref{eq:defLbar}. The theory developed in \cite[\S5]{bernard2008young} then shows that, if $n=1$, the measure $\mu$ can be decomposed into a superposition of minimizing curves for $\bar L$, which in turn allows to prove the equivalence of \eqref{eq:pde} and \eqref{eq:relaxedpde} using arguments akin to those of the proof of Theorem \ref{cor:oc}.
 } in situations more general
 than the one stated in \eqref{eq:pde} and \eqref{eq:relaxedpde} that are akin to the one considered in Section \ref{sec:codim1}.

 \item \emph{Codimension 1, that is, $m=1$ and any $n>0$.} In this case, \eqref{eq:pde} and \eqref{eq:relaxedpde} are equivalent as well. 
 
 To prove this in Section \ref{sec:codim1}, we {prove a superposition result that  can perhaps be understood as a version of the coarea formula}. We obtain a decomposition of the measure $\mu$ into a convex combination of functions in Sobolev space $W^{1,\infty}(\Omega)$, which can be approximated arbitrarily well by $C^1$ functions, providing the pursued result. {The superposition result is inspired by the Hardt-Pitts decomposition \cite{hardt1996solving,zworski1988decomposition,tasso},} an old, well-known  result, the existing versions thereof do not directly apply in our setting and are hard to approach for non-expert audience{; see also Section \ref{sec:hardtpitts}}. Here, we provide a self-contained proof of the extension applicable in our setting that relies on theory mostly developed by de Giorgi on the structure of sets with finite perimeter, already made accessible in the books \cite{maggi2012sets,evansgariepy}. This result holds true in a very general setting, with the most important assumption being the convexity of $L$ in the variable $z$; see Theorem \ref{thm:nogap}. 
 
 We have also reformulated the no-gap result in the context of optimal control problems; see Section \ref{sec:oc}.
 
 The idea of reformulating \eqref{eq:pde} as a linear programming problem and using a hierarchy of semidefinite programming problems to approximate it was first proposed in~\cite{korda2018moments}. First partial positive results on the absence of relaxation gap between \eqref{eq:pde} and \eqref{eq:relaxedpde} can be found in~\cite{marx2018moment,chernyavsky2021convex}, with~\cite{marx2018moment} using additional entropy inequalities to ensure concentration of the measure  on a graph of a function for scalar hyperbolic conservation laws while~\cite{chernyavsky2021convex} treating special cases of $L$.

 \item \emph{Higher dimension and codimension, that is, any $m>1$ and any $n>1$.} In this case, we are able to construct an example in which the infimum from \eqref{eq:relaxedpde} is strictly less than the one from \eqref{eq:pde}, thus showing that these two problems are not equivalent. The example constructed in Section \ref{sec:positive gap} consists of a situation in which the measure-valued minimizer corresponds to an irreducible double-covering of $\Omega$, similar to the Riemann surface of the complex square root. The difficulty of the argument is in providing a lower bound for the integral of $L$ on every classical subsolution; this is done applying the Poincar\'e-Wirtinger inequality. In the example we construct, $L$ is of regularity  $C^{1,1}_{\mathrm{loc}}$, that is, it is differentiable with locally Lipschitz gradient, and we indicate how to construct similar examples of arbitrary regularity $C^k$, $k\geq 1$. 
 
 Under further assumptions of joint convexity in $(y,z)$ and an additional additive structure of $L$, 
 the absence of a relaxation gap can be restored~\cite{fantuzzi2022sharpness}. The additive structure assumption can also be removed \cite{convexcov}.
\end{itemize}
We have additionally found that integral constraints of the form 
\[\int_{\Omega}H(x,y(x),Dy(x))dx\leq 0\quad \textrm{or}\quad \int_{\Omega}H(x,y(x),Dy(x))dx= 0\] may give rise to positive gaps in any dimension; we give some examples in Section \ref{sec:integralconstraints}.

\paragraph{Further discussion.}
While it is tempting to understand measure-valued solutions as less satisfactory objects than their classical counterparts due to the possible existence of gaps between the original problem \ref{eq:pde} and its measure-valued relaxation \ref{eq:relaxedpde}, there are cases in which measure-valued solutions may make more sense than the ``true solutions'' of a minimization problem, depending on taste and desired applications. This in particular means that in many cases, even as there may be a gap between the classical problem \eqref{eq:pde} and its relaxation \eqref{eq:relaxedpde}, the algorithms proposed in \cite{korda2018moments} will still prove useful and valuable. 

A good example is given by the multi-valued minimizer of the Lagrangian $L$ constructed in Section \ref{sec:positive gap} below. In this case  the measure-valued minimizer correctly encodes both values, and its support elegantly occupies exactly the zeros of $L$. No weakly-differentiable function is able to capture the multi-valued aspect of the problem, and in fact no global classical solution exists. While it is possible to construct discontinuous minimizing functions, these are likely to be deemed defective or incomplete when compared to the information conveyed by the measure-valued minimizer. Thus in this case the latter is likely superior for most applications, and in this sense problem \eqref{eq:relaxedpde} may  be preferred over \eqref{eq:pde}.

{Remark that occupation measures constitute only one of several existing relaxations. An important relaxation is the one given by gradient Young measures\footnote{{A gradient Young measure is a family $(\nu_x)_{x\in\Omega}$ of Borel probability measures $\nu_x$ on $Z$ generated by a family of gradients. According to the Kinderlehrer--Pedregal theorem \cite{kinderlehrerpedregal}, such measures are characterized by two conditions:~(1) their underlying barycenter is a gradient $\nabla y(x)=\int z\,d\nu_x(z)$ for some weakly-differentiable function $y(x)$, and (2) they satisfy Jensen's inequality for all quasiconvex functions. In this relaxation, one minimizes $\int_\Omega \int_ZL(x,y(x),z)\,d\nu_x(z)\,dx$.}}.
This relaxation was designed in order to ensure the existence of minimizers for the variational problem. However, it does not address the problem of non-convexity. Indeed, if $L(x,y,z)$ is a non-convex function of $y$, the resulting gradient Young measure relaxation of the variational problem may be non-convex. The occupation measure relaxation, on the other hand, is always convex by design and its minimum is attained under mild assumptions. The price to pay is the possibility of a relaxation gap which is the topic studied in this work.}

\paragraph{Notations.}
For a set $A\subset\R^n$, we denote its closure by $\overline A$. 
For a measurable set $A\subset\R^k$, denote by $|A|$ its Lebesgue measure, and by $\chi_A$ the indicator function of $A$, which is equal to 1 on $A$ and to 0 elsewhere. Given a measure $\mu$ on a set $A$ and a map $\phi\colon A\to B$, the pushforward measure $\phi_\#\mu$ is defined by $\phi_\#\mu(X)=\mu(\phi^{-1}(X))$ for all measurable sets $X\subset B$. For a finite-dimensional linear space $V$, denote by $V^*$ the space of linear functionals $V\to\R$. Denote by $C^\infty(X)$ the set of infinitely-differentiable functions on $X$, real valued, and by $C^\infty_c(X)$ the subset consisting of compactly-supported functions. If $X$ is an open set, the functions in $C^\infty_c(X)$ must vanish in a neighborhood of the boundary $\partial X$.

For two vectors $u=(u_1,\dots,u_n)$ and $v=(v_1,\dots,v_n)$ in $\R^n$, we denote the Euclidean inner product by $\langle u,v\rangle=\sum_{i=1}^nu_iv_i$, and its induced norm by $\|u\|=\sqrt{\langle u,u\rangle}$.

For a closed set $B\subset\R^n$, the notation $C^k(B)$ denotes the space of functions $f\colon B\to\R$ such that there is an open set $U$ containing $B$ such that $f$ can be extended to a $k$-times continuously differentiable function on $U$. 

Recall a function $\varphi\colon\Omega\to\R$ is \emph{weakly differentiable} if there is an integrable function $D\varphi\colon\Omega\to\R^n$, referred to as the weak derivative of $\varphi$, such that
\begin{equation}\label{eq:defweakdiff}
 \int_\Omega \varphi \,D\phi\,dx=-\int_\Omega \phi \, D\varphi\,dx
\end{equation}
for all $\phi\in C_c^\infty(\Omega)$.
The Sobolev space $W^{k,p}(U)$, for $U\subset\R^n$ open, contains all $k$ times weakly-differentiable functions $U\to\R$ with weak derivatives in $L^p(U)$.

Given a function $f\colon A\times B\to\R$, defined on the product of two convex subsets $A$ and $B$ of Euclidean spaces, we say that $f$ is \emph{convex in $A$} if for all $a,a'\in A$ and all $b\in B$ we have
\[f(\lambda a+(1-\lambda)a',b)\leq \lambda f(a,b)+(1-\lambda)f(a',b),\quad \lambda\in [0,1].\]

For projections on product spaces $A\times B$, we will use the notation
\begin{align*}
    &\proj_{A}\colon A\times B\to A,\\
    &\proj_A(a,b)=a,\qquad a\in A,\; b\in B.
\end{align*}

\section{No gap in codimension one}
\label{sec:codim1}
In this section we study the relaxation gap in codimension one in a rather general setting including constraints in the form of nonlinear first-order partial differential equations and inequalities as well as boundary conditions. We do so first for the problem of calculus of variations and then generalize it to optimal control, with the backbone of both results being the superposition principle proved in Theorem~\ref{thm:decomposition}.
\subsection{Formulation for variational calculus problems}
\label{sec:lagrangian}

Let $\Omega$ be a bounded, connected, open subset of $\R^n$ with piecewise $C^1$ boundary $\partial \Omega$ and denote the variables on $\Omega$ by $x=(x_1,\dots,x_n)$. Let  $\sigma$ denote the Hausdorff boundary measure on the piecewise $C^1$ set $\partial \Omega$. %
We also set $Y=\R$ with variable $y$ and $Z=\R^n$ with variables $z=(z_1,\dots,z_n)$. For simplicity, we will sometimes denote $x_{n+1}=y$.

Recall that a function is \emph{locally bounded} if it is bounded on every compact subset of its domain.

 We consider two optimization problems, formulated with the following objects and assumptions:  
\begin{enumerate}[label=CV\arabic*.,ref=CV\arabic*]
    \item \label{U:first}\label{U:lsc} $L\colon\Omega\times Y\times Z\to\R$ and $L_\partial\colon\partial\Omega\times Y\to\R$ are measurable and locally bounded,%
    \item\label{U:bulkconditions} $F,G%
    \colon \Omega\times Y\times Z\to\R$ are measurable functions,
    \item\label{U:boundaryconditions} $F_\partial,G_\partial%
    \colon\partial\Omega\times Y\to\R$ are measurable functions on the boundary,
    \item \label{U:convexity} $L$ is convex in $z$,
    \item \label{U:synthetic} 
     $F^{-1}(0)\cap G^{-1}((-\infty,0])\cap((x,y)\times Z)$ is convex for every $(x,y)\in\Omega\times Y$.
     \item\label{U:last}\label{U:closedconditions}  $F^{-1}(0)\cap G^{-1}((-\infty,0])$ and $F_\partial^{-1}(0)\cap G_\partial^{-1}((-\infty,0])$ are closed.
\end{enumerate}
As an example, here are some simple assumptions that imply \ref{U:first}--\ref{U:last}:
\begin{itemize}
    \item $L,F,G,L_\partial, F_\partial,G_\partial$ are continuous,
    \item $L$ and $G$ are convex in $z$, and
    \item  $F$ satisfies either  of the following two assumptions:
  \begin{enumerate}[label=A\arabic*.,ref=A\arabic*]
   \item \label{a:convex} $F$ is nonnegative and convex in $z$, or
   \item \label{a:affine} $F$ is affine in $z$.
\end{enumerate}
\end{itemize}
The first problem that interests us is the classical one:

\begin{mymathbox}[ams align]\label{opt:classical}
 M_\mathrm{c}= &\!\inf\limits_{y \in W^{1,\infty}(\Omega;Y)}  &\quad &\displaystyle\int_{\Omega} L(x,y(x),D y(x))\, dx + \int_{\partial\Omega} L_\partial(x,y(x))\, d\sigma(x) \\ 
 &\textrm{subject to}& & F(x,y(x),Dy(x)) = 0,\quad  G(x,y(x),Dy(x)) \le 0,\quad \text{a.e. }x\in \Omega,\nonumber \\
 &&& F_\partial(x,y(x)) = 0,\quad \;\hspace{0.9cm} G_\partial(x,y(x)) \le 0, \quad \hspace{0.9cm}\text{a.e. } x\in \partial\Omega, \nonumber
\end{mymathbox}
Observe that ``a.e. $x\in \Omega$'' and ``a.e. $x\in\partial\Omega$'' mean, respectively, almost-everywhere with respect to 
{the} Lebesgue measure on $\Omega$ and with respect to the $(n-1)$-dimensional Hausdorff measure on the boundary $\partial\Omega$. 

We remark that in the case with codimension $m=1$, the assumption \ref{U:convexity} of convexity in $z$ is equivalent to quasi-convexity of $L$, an important property that is in turn equivalent to the lower-semicontinuity of the integral of $L$ in \eqref{opt:classical}, an important property that is used in the direct method to prove the existence of minimizers; see \cite{dacorogna} for details. For this reason, the convexity assumption is a natural one in this context.

The second problem is the occupation-measure relaxation. 

\begin{definition}[Relaxed occupation measures]\label{def:M}
 Let $\mathcal M$ be the set of pairs $(\mu,\mu_\partial)$ consisting of compactly-supported, positive, Radon measures on $\overline\Omega\times Y\times Z$ respectively $\partial\Omega\times Y$ satisfying
 \begin{equation}\label{eq:measureomega}
 \mu(\Omega\times Y\times Z)=|\Omega|,
\end{equation}
and
\begin{equation}\label{eq:boundarymeasure}
  \int_{\Omega\times Y\times Z}\frac{\partial\phi}{\partial x}(x,y)+\frac{\partial \phi}{\partial y}(x,y)z\,d\mu(x,y,z)
  =\int_{\partial\Omega\times Y} \phi(x,y)\mathbf n(x)\,d\mu_\partial(x,y),\quad \phi\in C^\infty(\overline\Omega\times Y).
 \end{equation}
Here $\mathbf n$ denotes the exterior unit vector normal to the boundary $\partial\Omega$.
Note that here $\frac{\partial \phi}{\partial x}$, $\mathbf{n}$ and $z$ are in $\R^n$ and hence for each $\phi$ the above equation is in fact a system of $n$ equations.
 In each pair $(\mu,\mu_\partial)\in \mathcal M$, the measure $\mu$ is referred to as a \emph{relaxed occupation measure} and the measure $\mu_\partial$ as a \emph{relaxed boundary measure}.
 \end{definition}

 Observe that every $(\mu,\mu_\partial)\in \mathcal M$ satisfies
\begin{equation}\label{eq:finitemoment s}
\int_{\Omega\times Y\times Z}\|z\|\,d\mu(x,y,z)<+\infty,
\end{equation}
since $\mu$ is finite and compactly-supported.

The relaxation of problem \eqref{opt:classical} considered in this work is
\begin{mymathbox}[ams align] \label{opt:relaxed}
 M_\mathrm{r}=&\!\inf\limits_{(\mu,\mu_\partial)\in \mathcal M}  &\quad&\displaystyle\int_{\Omega\times Y\times Z} L(x,y,z)\, d\mu(x,y,z) + \int_{\partial\Omega\times Y} L_\partial(x,y)\, d\mu_\partial \\ 
 &\textrm{subject to} %
&& \operatorname{supp}\mu\subset \{(x,y,z)\in\Omega\times Y\times Z:F(x,y,z) = 0,\;\; G(x,y,z) \le 0\},\nonumber\\ 
&&& \operatorname{supp}\mu_\partial\subset \{(x,y)\in\Omega\times Y:F_\partial(x,y) = 0,\;\; G_\partial(x,y) \le 0\}.\nonumber%
\end{mymathbox}
Naturally we have $M_{\mathrm{c}}\geq M_{\mathrm{r}}$ (see the proof of Theorem \ref{thm:nogap}) and the primary goal of this section is to prove that $M_{\mathrm{c}} = M_{\mathrm{r}}$ if \ref{U:first}-\ref{U:last} hold. The main theoretical result of this work that will enable us to establish this is the following result.%

\begin{theorem}\label{thm:decomposition}
 Let $m = \mathrm{dim}\,Y = 1$ and let $\mu$ be a compactly supported, positive, finite, Radon measure on $\overline\Omega\times Y\times Z$ and,
 for all $\phi\in C_c^\infty(\Omega\times Y)$,
\begin{equation}\label{eq:boundarycondition}\int_{\Omega\times Y\times Z}\frac{\partial\phi}{\partial x}(x,y)+\frac{\partial \phi}{\partial y}(x,y)z\,d\mu(x,y,z)=0.
\end{equation}
 Then there are a compactly-supported, finite, positive, Radon measure $\nu$ on $\R$ and a family of  %
 functions $(\varphi_r\colon\Omega\to Y)_{r\in\R}\subset W^{1,\infty}(\Omega)$ such that, for all functions $\phi\in L^1(\mu)$ that are affine in $z$ we have
 \begin{equation}\label{eq:superposition}
  \int_{\Omega\times Y\times Z} \phi\,d\mu=\int_\R\int_\Omega\phi(x,\varphi_r(x),D\varphi_r(x))\,dx\,d\nu(r).
 \end{equation}
 Additionally, if $r\geq r'$ then $\varphi_r(x)\leq \varphi_{r'}(x)$ for all $x\in \Omega$.
\end{theorem} 

Note that \eqref{eq:boundarycondition} is a special case of \eqref{eq:boundarymeasure} when the set of test functions is restricted to $C_c^\infty(\Omega\times Y)$. 

The proof of Theorem \ref{thm:decomposition} presented in Section \ref{sec:proof} follows the arguments given in  \cite{zworski1988decomposition}, although the setting of~\cite{zworski1988decomposition} is different than the one considered here.  Theorem~\ref{thm:decomposition} enables us to prove the following result, which leads immediately to establishing $M_\mathrm{r} = M_\mathrm{c}$:

\begin{theorem}\label{thm:consolidated}
 Assume that $m = \mathrm{dim}\,Y = 1$ and that the functions $L,F,G,L_\partial, F_\partial,G_\partial$ satisfy \ref{U:first}--\ref{U:last}.

 Let $(\mu,\mu_\partial)\in\mathcal M$.
Suppose that the supports of $\mu$ and $\mu_\partial$ satisfy
 \begin{gather}
  \label{eq:Fcond}
  \operatorname{supp}\mu\subset \{(x,y,z)\mid F(x,y,z) = 0, \;G(x,y,z) \le 0\}\\
  \label{eq:Gcond}
  \operatorname{supp}\mu_\partial \subset \{(x,y)\mid F_\partial(x,y) = 0,\; G_\partial(x,y) \le 0\}.
 \end{gather}
 Then we have the following two conclusions:
 \begin{enumerate}[label=\roman*.,ref=(\roman*)]
     \item \label{it:barphi} 
      There is a function $\bar\varphi\in W^{1,\infty}(\bar\Omega)$ such that 
     \begin{gather}\label{eq:barphiL1}
     \int_\Omega L(x,\bar\varphi(x),D\bar\varphi(x))\,dx+\int_{\partial\Omega} L_\partial(x,\bar\varphi(x))\,d\sigma(x)\leq \int_{\Omega\times Y\times Z} L\,d\mu+\int_{\partial\Omega\times Y} L_\partial\,d\mu_\partial,\\
     \label{eq:barphiF}
     F(x,\bar\varphi(x),D\bar\varphi(x))=0,\quad G(x,\bar\varphi(x),D\bar\varphi(x))\le 0\quad\textrm{
      a.e. 
      $x\in\Omega$},\\
     \label{eq:barphiG}
     F_\partial(x,\bar\varphi(x))=0,\quad G_\partial(x,\bar\varphi(x))\leq 0\quad\text{a.e. }
     \textrm{$x\in\partial\Omega$,} %
     \end{gather}
     where $\sigma$ is the $(n-1)$-dimensional Hausdorff measure on $\partial\Omega$.
     \item \label{it:gi} 
     Assume additionally that $L$, $F$, and $G$ are continuous.
     There exists a sequence of functions $(g_i\colon\overline \Omega \to Y)\subset C^\infty(\Omega)\cap W^{1,\infty}(\overline\Omega)$, such that
     \begin{equation}\label{eq:Lineq}
 \lim_{i\to+\infty}\int_\Omega L(x,g_i(x),Dg_i(x))dx+ \int_{\partial\Omega} L_\partial(x,g_i(x))d\sigma(x) \leq \int_{\Omega\times Y\times Z} L\,d\mu+\int_{\partial\Omega\times Y} L_\partial\,d\mu_\partial,
 \end{equation}
 and %
 \begin{gather}
  \label{eq:limF}
  \lim_{i\to+\infty}
  F(x,g_i(x),D g_i(x))=0,\quad  \lim_{i\to+\infty}
  G(x,g_i(x),D g_i(x))\le 0 \quad %
  \text{a.e. }x\in\Omega,\\
  \label{eq:limG}
  F_\partial(x,g_i(x))=0,\quad G_\partial(x,g_i(x))\le 0\quad \text{a.e. }
   x\in\partial\Omega, \;i=1,2,\dots %
 \end{gather}
 \end{enumerate}
\end{theorem}

The proof of Theorem \ref{thm:consolidated} is presented in Section \ref{sec:proof2}. This theorem immediately leads to a result on the absence of a relaxation gap between \eqref{opt:classical} and \eqref{opt:relaxed}. %

\begin{theorem}
 \label{thm:nogap}
 Assume that $m = \mathrm{dim}\,Y = 1$ and that the functions $L,F,G,L_\partial, F_\partial,G_\partial$ satisfy \ref{U:first}--\ref{U:last}.
If $M_\mathrm{c}<+\infty$, then \[M_\mathrm{c}=M_\mathrm{r}.\] 
\end{theorem}
\begin{proof}
 Since every function $y\in W^{1,\infty}(\Omega;Y)$ induces measures $(\mu,\mu_\partial)$ by 
 \begin{gather*}
 \int_{\Omega\times Y\times Z}\phi(x,y,z)d\mu(x,y,z)=\int_{\Omega}\phi(x,y(x),Dy(x))\,dx,\quad \phi\in C^0(\Omega\times Y\times Z),\\
 \int_{\partial \Omega\times Y}\phi_\partial(x,y) d\mu_\partial(x,y)=\int_{\partial\Omega}\phi_\partial(x,y(x))\,d\sigma(x),\quad \phi_\partial\in C^0(\partial\Omega\times Y),
 \end{gather*}
 and the pair $(\mu,\mu_\partial)$ satisfies all the constraints of $M_\mathrm{r}$,
 we have $M_\mathrm{r}\leq M_\mathrm{c}$. In order to prove the opposite direction, assume that $(\mu,\mu_\partial)$ is feasible in \eqref{opt:relaxed}. Such $(\mu,\mu_\partial)$ satisfies the assumptions of Theorem~\ref{thm:consolidated} and hence there exists a function $\bar\varphi \in W^{1,\infty}(\overline\Omega)$ satisfying~\eqref{eq:barphiL1}--\eqref{eq:barphiG}. This implies that $\bar\varphi$ is feasible in \eqref{opt:classical} and achieves an objective value no worse than the objective value achieved by $(\mu,\mu_\partial)$ in \eqref{opt:relaxed}.
\end{proof}
\begin{definition}[Centroid and centroid-concentrated measure]
\label{def:centroid}
 Let $\mu$ be a positive Radon measure on $\Omega\times Y\times Z$.
 Denote the marginal measure $(\projXY)_\#\mu$ by $\mu_{\Omega\times Y}$.
 Disintegrate $\mu$ through the projection map $\projXY{}$
to obtain a family of probability measures $(\mu_{xy})_{(x,y)\in\Omega\times Y}$, with $\mu_{xy}$ being a measure on $Z$, such that
\[\mu=\int_{\Omega\times Y}\mu_{xy}\,d \mu_{\Omega\times Y}(x,y).\]
In other words, we have, for measureable $f\colon\Omega\times Y\times Z\to\R$,
\[\int f(x,y,z) d\mu=\int_{\Omega\times Y} \int_Z f(x,y,z)d\mu_{xy}(z)\,d \mu_{\Omega\times Y}(x,y).\]
By \eqref{eq:finitemoment s}, the quantity
\begin{equation}\label{eq:defZ}
 \mathcal Z(x,y)=\int z\,d\mu_{xy}(z)
\end{equation}
is well defined and finite for $(\projXY)_\#\mu$-almost every $(x,y)$; it is referred to as the \emph{centroid} of $\mu$ at $(x,y)$ and can also be thought of as the conditional expectation of the $z$ variable given $(x,y)$. Let $\bar \mu$ be the measure whose projection coincides with that of $\mu$, that is, $(\projXY)_\#\bar\mu=(\projXY)_\#\mu=\mu_{\Omega\times Y}$, and which is concentrated on $\mathcal Z(x,y)$, that is,
\[\bar\mu=\int_{\Omega\times Y}\delta_{\mathcal Z(x,y)}d\mu_{\Omega\times Y}(x,y);\]
this means that, for measurable $f\colon\Omega\times Y\times Z\to\R$, we have
\[\int_{\Omega\times Y\times Z}f(x,y,z)\,d\bar\mu(x,y,z)=\int_{\Omega\times Y}f(x,y,\mathcal Z(x,y))\,d\mu_{\Omega\times Y}(x,y).\]
The measure $\bar\mu$ is \emph{the version of $\mu$ concentrated at its centroid in the z variable}. 
\end{definition}
\begin{remark}
In the absence of the convexity assumptions \ref{U:convexity} and \ref{U:synthetic}, $M_{\mathrm{r}}$ remains the same if we replace $L$ with its convexification $\tilde L$ in $z$, given, for $(x,y,z)\in \Omega\times Y\times Z$, by
 \begin{align*}\tilde L(x,y,z)=\inf \{&\lambda L(x,y,z')+(1-\lambda)L(x,y,z''):\\
 &z=\lambda z'+(1-\lambda)z'',\;\lambda\in[0,1],\;z',z''\in Z,\\
 &F(x,y,z')=0=F(x,y,z''),\;G(x,y,z')\leq 0,\;G(x,y,z'')\leq 0\}.
 \end{align*}
  Indeed, denoting the latter minimum by $\tilde M_\mathrm{r}$, observe that we always have $M_\mathrm{r}\geq \tilde M_\mathrm{r}$ because $L\geq \tilde L$; let us show the opposite inequality. The measure $\bar\mu$ constructed in Definition \ref{def:centroid}, which concentrates the mass of $\mu$ on its centroid $\mathcal Z(x,y)$ in each fiber $(x,y)\times Z$, satisfies
 \[\int L\,d\mu\geq \int \tilde L\,d\mu\geq\int \tilde L\,d\bar\mu.\]

 A new measure $\tilde\mu$ can be constructed that redistributes, on each fiber $(x,y) \times Z$, the mass of $\bar\mu$ on the points where $\tilde L=L$ while maintaining the same centroid; indeed, on each fiber $(x,y)\times Z$ we can pick (for example, using Choquet's theorem) a probability measure $\nu_{(x,y)}$ supported on the extreme points of the facet of $\tilde L$ containing the centroid $\mathcal Z(x,y)$, in such a way that the centroid of $\nu_{(x,y)}$ will again be $\mathcal Z(x,y)$; it can be argued using the Kuratowski--Ryll-Nardzewski Selection Theorem (see \cite[Th. 18.13]{aliprantisborder} or \cite[Th.\ 8.1.3]{aubinfrankowska}) %
 that this choice can be done in such a way as to produce a measurable selection on the set-valued map associating to each $(x,y)\in\Omega\times Y$ the set of probabilities on the extreme points of the facet containing the centroid; to finish the construction, let $\tilde\mu=\int_{\Omega\times Y}\nu_{(x,y)}d(\projXY)_\#\mu(x,y)$.
 Then we have
 \[\int \tilde L\,d\mu=\int \tilde L\,d\tilde\mu=\int L\,d\tilde\mu.\]
 Now, $(\tilde\mu,\mu_\partial)\in\mathcal M$ because condition \eqref{eq:boundarycondition} does not change by the construction of $\tilde\mu$ because integrals of functions linear in $z$ are not affected. Thus we have $M_\mathrm{r}\leq \tilde M_r$, which is what we wanted to show.
\end{remark}

\subsection{Formulation for optimal control}
\label{sec:oc}

In this section we extend the no-gap result of Theorem \ref{thm:nogap} to the context of optimal control. Let $\Omega\subset\R^n$ be a bounded, connected, open set with piecewise $C^1$ boundary $\partial\Omega$ and with boundary measure $\sigma$. Let also $Y=\R$, and $Z=\R^{n}$. Let $U$ and $U_\partial$ be compact topological spaces. 

Let $\projXYZ{}\colon \Omega\times Y\times Z\times U\to \Omega\times Y\times Z$ and $\projdXY{}\colon\partial\Omega\times Y\times U_\partial\to \partial\Omega\times Y$ 
be the projections 
$\projXYZ{}(x,y,z,u)=(x,y,z)$ and $\projdXY{}(x,y,u)=(x,y)$. 

In analogy with \ref{U:first}--\ref{U:last}, we will assume:
\begin{enumerate}[label=OC\arabic*.,ref=OC\arabic*]
    \item \label{OC:first} $L\colon \Omega\times Y\times Z\times U\to\R$ and $L_\partial \colon\partial\Omega\times Y\times U_\partial\to\R$ are measurable and locally bounded functions, %
    \item $F,G\colon \Omega\times Y\times Z\times U\to\R$ are measurable functions,
    \item $F_\partial,G_\partial\colon\partial \Omega\times Y\times U_\partial\to \R$ are measurable functions on the boundary,
     \item\label{OC:3} the function $\bar L\colon \Omega\times Y\times Z\to\R$ defined by
     \begin{equation}\label{eq:defLbar}
      \bar L(x,y,z)=\inf\{L(x,y,z,u):u\in U,\;F(x,y,z,u)=0,\; G(x,y,z,u)\leq 0\}
     \end{equation}
     is measurable, locally bounded, and convex in $z$,
     \item\label{OC:2} $\projXYZ{}(F^{-1}(0)\cap G^{-1}((-\infty,0]))\cap((x,y)\times Z)$ is convex for every $(x,y)\in\Omega\times Y$.
     \item \label{OC:last}  $F^{-1}(0)\cap G^{-1}((-\infty,0])$ and $F_\partial^{-1}(0)\cap G^{-1}_\partial((-\infty,0])$ are closed.
\end{enumerate}
Assumption \ref{OC:2} amounts to the set of permissible points being convex on each fiber $Z$, once we project with $\projXYZ{}$. For a concrete application satisfying these assumptions, refer to Example \ref{ex:affine}.

We want to consider the following two optimization problems: first, the classical multivariable optimal control problem

\begin{alignat}{2} \label{opt:classical_cont}
 M_\mathrm{c}^\mathrm{oc}= &\!\inf\limits_{\substack{y \in W^{1,\infty}(\Omega;Y)\\u \in L^{\infty}(\Omega;U)\\u_\partial \in L^\infty(\partial\Omega;U_\partial)}}  &\quad &\displaystyle\int_{\Omega} L(x,y(x),D y(x),u(x))\, dx + \int_{\partial\Omega} L_\partial(x,y(x),u_\partial(x))\, d\sigma(x) \\ 
 &\textrm{subject to}& & F(x,y(x),Dy(x),u(x)) = 0,\quad  G(x,y(x),Dy(x),u(x)) \le 0,\quad \text{a.e. }x\in \Omega,\nonumber \\
 &&& F_\partial(x,y(x),u_\partial(x)) = 0,\quad \;\hspace{0.725cm} G_\partial(x,y(x),u_\partial(x)) \le 0, \quad \hspace{0.725cm} \text{a.e. }x\in \partial\Omega, \nonumber
\end{alignat}
and its relaxation
\begin{alignat}{2} \label{opt:relaxed_cont}
 M_\mathrm{r}^\mathrm{oc}=
 &\!\inf\limits_{(\mu,\mu_\partial)\in \mathcal M^\mathrm{oc}}  &\quad&\displaystyle\int_{\Omega\times Y\times Z\times U} L(x,y,z,u)\, d\mu(x,y,z,u) + \int_{\partial\Omega\times Y\times U_\partial} L_\partial(x,y,u)\, d\mu_\partial(x,y,u) \\ 
 &\textrm{subject to}%
&& \operatorname{supp}\mu\subset \{(x,y,z,u)\in\Omega\times Y\times Z\times U:F(x,y,z,u) = 0,\;\; G(x,y,z,u) \le 0\},\nonumber\\ 
&&& \operatorname{supp}\mu_\partial\subset \{(x,y,u)\in\partial\Omega\times Y\times U_\partial:F_\partial(x,y,u) = 0,\;\; G_\partial(x,y,u) \le 0\},\nonumber%
\end{alignat}
where $\mathcal M^\mathrm{oc}$ denotes the set of pairs $(\mu,\mu_\partial)$ consisting of compactly-supported positive Borel measures on $\overline\Omega\times Y\times Z\times U$ respectively $\partial\Omega\times Y\times U_\partial$ satisfying
\begin{gather}
\label{eq:mass_cont}
  \mu(\overline\Omega\times Y\times Z\times U)=|\Omega|,
\end{gather}
 and
\begin{equation}
  \label{eq:boundarymeasure_cont}
  \int_{\overline\Omega\times Y\times Z\times U}\frac{\partial\phi}{\partial x}(x,y)+\frac{\partial\phi}{\partial y}(x,y)z\,d\mu(x,y,z,u)
  =\int_{\partial\Omega\times Y\times U_\partial}\hspace{-8mm}\phi(x,y)\mathbf n(x)\,d\mu_\partial(x,y,u),\;\; \phi\in C^\infty(\Omega\times Y),
\end{equation}
which are the analogies of \eqref{eq:measureomega} and \eqref{eq:boundarymeasure}. Note that none of these conditions \eqref{eq:mass_cont}--\eqref{eq:boundarymeasure_cont} substantially involves the control set $U$, and they correspond to the hypotheses of Theorem \ref{thm:consolidated} and Theorem \ref{thm:nogap}. 

\begin{theorem}\label{cor:oc}
 If $M^\mathrm{oc}_\mathrm{c}$ is finite and \ref{OC:first}--\ref{OC:last} hold, then $M^\mathrm{oc}_\mathrm{c}=M^\mathrm{oc}_\mathrm{r}$.
\end{theorem}
\begin{proof}
 We always have $M^\mathrm{oc}_\mathrm{r}\leq M^\mathrm{oc}_\mathrm{c}$ because every  $(y_0,u_0,v_0)\in W^{1,\infty}(\Omega;Y)\times L^{\infty}(\Omega;U)\times L^\infty(\partial\Omega;U_\partial)$ induces pairs of measures $(\mu,\mu_\partial)\in \mathcal M^\mathrm{oc}$ by
\[\int_{\overline\Omega\times Y\times Z\times U}\phi(x,y,z,u)\,d\mu(x,y,z,u)=\int_{\Omega}\phi(x,y_0(x),Dy_0(x),u_0(x))\,dx,\quad \phi\in C^0(\Omega\times Y\times Z\times U),\]
and 
\[ \int_{\partial\Omega\times Y\times U_\partial}\phi(x,y,u)\,d\mu_\partial(x,y,u)=\int_{\partial\Omega}\phi(x,y_0(x),v_0(x))\,d\sigma(x),\quad \phi\in C^0(\Omega\times Y\times U_\partial),\]
and they satisfy \eqref{eq:mass_cont}--\eqref{eq:boundarymeasure_cont}.

 Define $\bar L$ as in \eqref{eq:defLbar} and similarly,
 \[\bar L_\partial(x,y)=\inf \{L_\partial(x,y,u):u\in U,\; F_\partial(x,y,u)=0,\;G_\partial(x,y,u)\leq 0\},\quad (x,y)\in \partial\Omega\times Y.\]
 Then because of the local boundedness of $L_\partial$ and the compactness of $U_\partial$, $\bar L_\partial\colon \partial\Omega\times Y\to \R$ is locally bounded and measurable. We will use the functions $\bar L$ and $\bar L_\partial $ to reduce the optimal control problem to the variational calculus problem from Section \ref{sec:lagrangian}.
 
 The sets 
 \[\projXYZ{}(F^{-1}(0)\cap G^{-1}((-\infty,0]))\quad\textrm{and} \quad\projdXY{}(F_\partial^{-1}(0)\cap G_\partial^{-1}((-\infty,0]))\]
 are closed. We explain why this is true for the former, the latter being similar. For every compact set $K\subset \Omega\times Y\times Z$, the set $(K\times U)\cap (F^{-1}(0)\cap G^{-1}((-\infty,0])$ is compact, so its image under the continuous map $\projXYZ{}$ is compact, and it equals $K\cap\projXYZ{}(F^{-1}(0)\cap G^{-1}((-\infty,0])) $. Thus $\projXYZ{}(F^{-1}(0)\cap G^{-1}((-\infty,0]))$ is a set whose intersection with every compact set is compact, so it must be closed.

In order to reduce the optimal control problem to the variational calculus one considered in Section \ref{sec:lagrangian}, we will need functions that encode the admissibility conditions. Let
\begin{align*}
 \bar F(x,y,z)&=1-\chi_{\projXYZ{}(F^{-1}(0)\cap G^{-1}((-\infty,0]))}(x,y,z),\quad (x,y,z)\in \Omega\times Y\times Z,\\
 \bar F_\partial(x,y)&=1-\chi_{\projdXY{}(F_\partial^{-1}(0)\cap G_\partial^{-1}((-\infty,0]))}(x,y),\quad (x,y)\in\partial\Omega\times Y,
 \end{align*}
 as well as $\bar G=0=\bar G_\partial$.

 Consider problems \eqref{opt:classical} and \eqref{opt:relaxed} with $L,F,G,L_\partial,F_\partial,G_\partial$ replaced by $\bar L,\bar F,\bar G,\bar L_\partial,\bar F_\partial,\bar G_\partial$; since assumptions \ref{OC:first}--\ref{OC:last} imply the corresponding assumptions \ref{U:first}--\ref{U:last}, and since $M_\mathrm{c}^\mathrm{oc}<+\infty$ on the optimal control side implies $M_\mathrm{c}<+\infty$ on the variational side, we have, by Theorem \ref{thm:nogap}, $M_{\mathrm{c}}=M_{\mathrm{r}}$ on the variational side.
 Denote by  
 \[I_1\coloneqq\int_{\partial\Omega}\bar L_\partial(x,\bar\varphi(x))\,d\sigma(x)\]
 and by
 \[I_2\coloneqq\int_{\partial\Omega\times Y\times U_\partial }\bar L_\partial(x,y)\,d\mu_\partial(x,y,u)\leq \int_{\partial\Omega\times Y\times U_\partial}L_\partial(x,y,u)\,d\mu_\partial(x,y,u).\]
 We have (omitting for brevity the conditions on $\bar\varphi$ and $\mu$ as in  \eqref{opt:classical} and \eqref{opt:classical_cont} for lines involving $\bar L$, and  as in \eqref{opt:classical_cont} and \eqref{opt:relaxed_cont} for lines involving $L$),
 \begin{align*}
     M^\mathrm{oc}_\mathrm{c}
     &\leq \inf_{\bar\varphi\in W^{1,\infty}(\Omega;Y)}\inf_{u\in L^\infty(\Omega; U)}\int_\Omega L(x,\bar\varphi(x),D\bar\varphi(x),u(x))\,dx+I_1\\
      &=\inf_{\bar\varphi\in W^{1,\infty}(\Omega;Y)}
      \int_\Omega \bar L(x,\bar\varphi(x),D\bar\varphi(x))\,dx+I_1\\
      &=M_\mathrm{c}\\
      &=M_\mathrm{r}\\
      &=\inf_{(\mu,\mu_\partial)\in \mathcal M}\int_{\Omega\times Y\times Z}\bar L\,d\mu+I_2\\
      &=\inf_{(\mu,\mu_\partial)\in \mathcal M^\mathrm{oc}}\int_{\Omega\times Y\times Z}\bar L\,d(\projXYZ{})_\#\mu+I_2\\
      &=\inf_{(\mu,\mu_\partial)\in \mathcal M^\mathrm{oc}}\int_{\Omega\times Y\times Z}\bar L\circ\projXYZ{}\,d\mu+I_2\\
      &\leq \inf_{(\mu,\mu_\partial)\in \mathcal M^\mathrm{oc}}\int_{\Omega\times Y\times Z\times U} L\,d\mu+I_2\\
      &= M^\mathrm{oc}_\mathrm{r}\\
      &\leq \inf_{\substack{y\in W^{1,\infty}(\Omega;Y)\\u\in L^\infty(\Omega;U)}}\int_\Omega L(x,y(x),Dy(x),u(x))\,dx+I_2\\
      &=M^\mathrm{oc}_\mathrm{c}.\qedhere
 \end{align*}
\end{proof}

\begin{example}[Affine control of the derivatives]
\label{ex:affine}
 Consider an optimal control problem in which a relation of the form
 \[Dy(x)=v(x,y,u)\]
 must be enforced. Assume that $v\colon \Omega\times Y\times U\to Z$ is such that $u\mapsto v(x,y,u)$ is affine and invertible for each pair $(x,y)$. Then we may encode the relation above by letting 
 \[F(x,y,z,u)=z-v(x,y,u).\]
 The effective Lagrangian $\bar L$ is then simply
 \[\bar L(x,y,z)=L(x,y,z,(v(x,y,\cdot))^{-1}(z)).\]
 If $L$ is continuous and convex in $z$ and $v$ is continuous, then $\bar L$ is continuous and convex in $z$ as well.
 With $F$ defined as above, and assuming for simplicity that $F_\partial=G_\partial=G=0$, then \ref{OC:first}--\ref{OC:last} are true.
\end{example}

\subsection{Proof of Theorem \ref{thm:decomposition} }\label{sec:proofdecomposition}

Now we come to the proof of Theorem~\ref{thm:decomposition}. We start by illustrating the main steps of the proof on a simple example.

\subsubsection{Overview of the proof of Theorem \ref{thm:decomposition}}
\label{sec:overview}

To fix ideas, let us show how the proof of Theorem \ref{thm:decomposition} works in the very simple case when $\Omega=[0,1]\subset\R$, $Y = \R$, $Z = \R$, and $\mu$ is induced by a $C^1$ curve $\gamma\colon\Omega\to Y$, so that it is given by
\[\int_{\Omega\times Y\times Z}f(x,y,z)\,d\mu(x,y,z)=\int_0^1f(x,\gamma(x),\gamma'(x))\,dx,\quad f\in C^0(\Omega\times Y\times Z).\]
In this case, Lemma \ref{lem:radon-nikodym} will confirm that the projection of $\mu$ onto $\Omega$ is a multiple of 
{the} Lebesgue measure (it is just $dx|_{[0,1]}$). We will then use a trick involving the computation of the circulation ${\mathcal C_\mu}(X)$ of vector fields $X$ and its relation to a linear functional $S\colon C^0(\Omega\times Y)\to \R$ that will be related by the fundamental identity (Lemma \ref{lem:dividentity})
\begin{equation}\label{eq:fundamental}
{\mathcal C_\mu}(X)=-S(\operatorname{div}X)
\end{equation}
and will give us, by the Radon-Nikodym theorem (see Lemma \ref{lem:radon-nikodym}), a function $\rho\colon\Omega\times Y\to\R$ that heuristically has the property that
\[\textrm{``}(\projXY)_\#\mu=-\frac{\partial\rho}{\partial y}.\textrm{''}\]
Thus in our example (see Figure \ref{fig:10}), 
\[\rho(x,y)=\begin{cases}
 -1,&y\geq \gamma(x),\\
 0,&y< \gamma(x).
 \end{cases}
 \]
 \begin{figure}
\includegraphics[width=13.5cm]{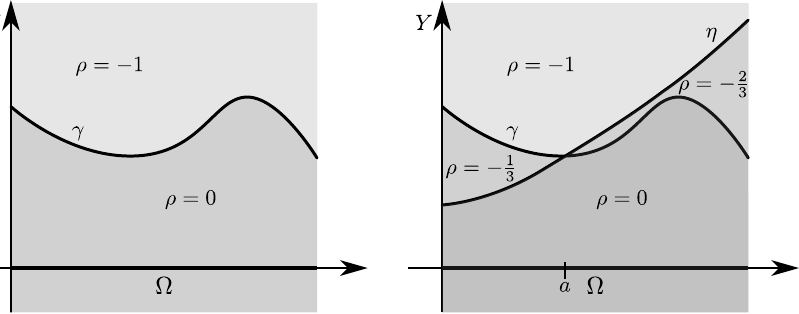}
\centering
\caption{The left-hand side diagram illustrates the values of $\rho$ when $\mu$ is induced by a single curve $\gamma$. In the right-hand side, we illustrate the case in which $\mu$ is the convex combination $\mu=\frac23\mu_1+\frac13\mu_2$ and $\mu_1$ and $\mu_2$ are measures induced by two curves, $\gamma$ and $\eta$, respectively. }
\label{fig:10}
\end{figure}

 After checking that $\rho$ is bounded (Lemma \ref{lem:boundedrange}), we will use the function $\rho$ to define the functions $\varphi_r$ (commonly known as \emph{sheets}) that will give the decomposition of $\mu$. This is done in Lemma \ref{lem:main}. Lemma \ref{lem:function} shows that $\varphi_r$ roughly corresponds to the boundary of a level set of $\rho$, and that it is ``almost continuous,'' and Lemma \ref{lem:weakderivative} shows that it is weakly differentiable; these two lemmas are used to prove Lemma \ref{lem:main}. The proof of Theorem \ref{thm:decomposition}, presented at the end of Section \ref{sec:proof}, relies on the fundamental identity \eqref{eq:fundamental}, together with the technical details from Lemma \ref{lem:main}. 
 
 In our example, the decomposition of Theorem \ref{thm:decomposition} gives the measure $\nu$ equal to 
{the} Lebesgue measure on ${\mathcal R}=[-1,0]$, and 
 \[\varphi_{r}(x)=\inf_{\substack{y\in Y\\\rho(x,y)\leq r}}y=\gamma(x),\quad r\in [-1,0),\]
 so that, indeed,
 \begin{multline*}
  \int_{\Omega\times Y\times Z} f\,d\mu
  =\int_{{\mathcal R}}\int_\Omega f(x,\varphi_r(x),D\varphi_r(x))\,dx\,d\nu(r)\\
  =\int_{-1}^0\int_0^1 f(x,\gamma(x),\gamma'(x))\,dx\,d\nu(r)=\int_0^1 f(x,\gamma(x),\gamma'(x))\,dx.
 \end{multline*}
 
 Another example, illustrated as well in Figure \ref{fig:10}, is the case in which $\mu=\tfrac23\mu_1+\tfrac13\mu_2$, and $\mu_1$ and $\mu_2$ are the measures induced by curves $\gamma$ and $\eta$, and say that $\gamma\geq \eta$ on $[0,a]$ and $\gamma<\eta$ on $(a,1]$, for some $0<a<1$. In this case, 
 \[\rho(x,y)=\begin{cases}
  0,&y<\gamma(x)\;\textrm{and}\;y<\eta(x),\\
  -\tfrac13,&\eta(x)\leq y<\gamma(x),\\
  -\tfrac23, &\gamma(x)\leq y<\eta(x),\\
  -1,&y\geq \gamma(x)\;\textrm{and}\;y\geq\eta(x).
 \end{cases}
 \]
 Similarly,
 \[\varphi_r(x)=\begin{cases}
    \gamma(x),&\textrm{($-1<r<-\tfrac13$ and $0<x<a$) or ($-\tfrac23<r<0$ and $a<x<1$)},\\
    \eta(x),&\textrm{($-\tfrac13<r<0$ and $0<x<a$) or ($-1<r<-\tfrac23$ and $a<x<1$)}.
 \end{cases}
 \]

\subsubsection{Proof of Theorem \ref{thm:decomposition}}
\label{sec:proof}

We collect some lemmas needed in the proof of the theorem, which is presented at the end of the section.
Throughout this section, we assume that $\mu$ is a non-zero measure satisfying the hypotheses of Theorem \ref{thm:decomposition}

\begin{lemma}\label{lem:lebesgueproj}
 If $\projX{}\colon \Omega\times Y\times Z\to\Omega$ is the projection, then there is $c>0$ such that
 \[(\projX{})_{\#}\mu=c\,dx.\]
 In other words, the pushforward $(\projX{})_{\#}\mu$ is a positive multiple of the Lebesgue measure on $\Omega$.
\end{lemma}
\begin{figure}
\includegraphics[width=4.5cm]{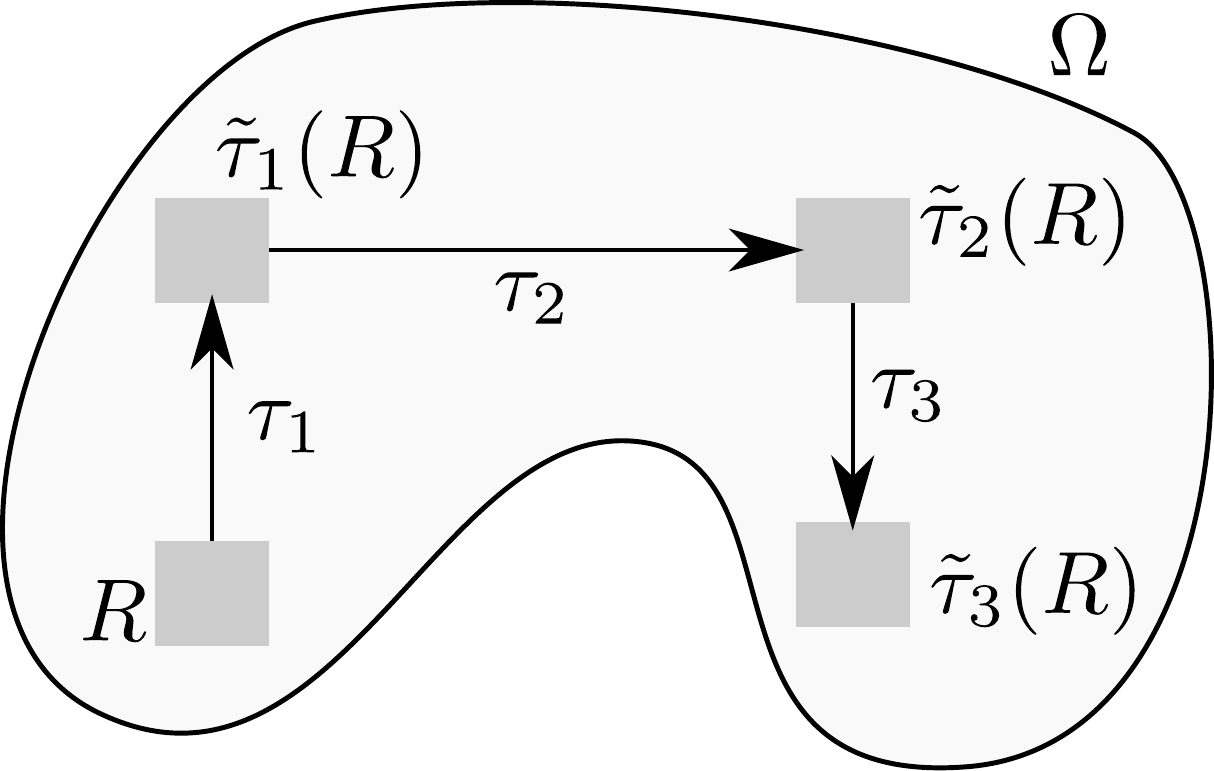}
\centering
\caption{Translating a rectangle in the proof of Lemma \ref{lem:lebesgueproj}.}
\label{fig:lebesgueproj}
\end{figure}
\begin{proof}
 Let $R\subset \Omega$ be a small parallelepiped, and let $\tau$ be a translation such that $\tau(R)\subset \Omega$. We will show that $(\projX{})_{\#}\mu(R)=(\projX{})_{\#}\mu(\tau(R))$, and since this will be true for all $R$ and all $\tau$, $(\projX{})_{\#}\mu$ must be a positive multiple of 
{the} Lebesgue measure on $\Omega$ \cite[Thm. 2.20]{rudin}. 
 Write $\tau$ as a finite composition of translations $\tau_i$ in the directions of the axes $x_1,\dots,x_n$,
 \[\tau=\tau_k\circ\tau_{k-1}\circ\dots\circ\tau_1.\]
 Denote $\tilde \tau_i=\tau_i\circ\tau_{i-1}\circ\dots\circ\tau_1$ and set $\tilde\tau_0$ equal to the identity. We assume $\tau_1,\dots,\tau_k$ have been chosen also in such a way that the convex hull of $\tilde\tau_{i-1}(R)\cup\tilde\tau_{i}(R)$ is contained in $\Omega$ for each $i$. Refer to Figure \ref{fig:lebesgueproj}.
 For each $i=1,\dots,k$, let $j_i$ be such that $\tau_i$ is a translation in direction $x_{j_i}$. 

 Take a sequence of smooth, compactly-supported functions $u_{i}\colon \Omega\to \R$, such that, for all $j=1,\dots,k$, the support of $u_{i}\circ\tilde\tau_j$ is properly contained in $\Omega$, and  $u_{i}\tilde\tau_j\to \chi_{R}\circ\tilde\tau_j$ in $L^1(\mu)$, that is, 
 \begin{equation}\label{eq:L1convergence}
  \lim_{i\to+\infty}\int_{\Omega\times Y\times Z} |u_{i}\circ\tilde\tau_j(x)-\chi_{\tilde\tau_j(R)}(x)|\,d\mu(x,y,z)=0,\qquad j=1,\dots,k.
 \end{equation}
 Define also, for $a\in\mathbb N$, $1\leq i\leq k$,
 \begin{multline*}
  \bar\phi_{ia}(x_1,\dots,x_n)=\\
  \int_{-\infty}^{x_{j_i}}u_{a}\circ\tilde\tau_i(x_1,\dots,x_{j_i-1},s,x_{j_i+1},\dots,x_n)-u_{a}\circ\tilde\tau_{i-1}(x_1,\dots,x_{j_i-1},s,x_{j_i+1},\dots,x_n)\,ds.
 \end{multline*}
 Observe that the $C^\infty$ function $\bar\phi_{ia}$ vanishes on the boundary of $\Omega$, so that equation \eqref{eq:boundarycondition} applies to $\bar\phi_{ia}$, which means that for the $j_i$-th entry we have, since $\partial\bar\phi_{ia}/\partial y=0$,
 \begin{equation*}
   0=\int_{\Omega\times Y\times Z}\frac{\partial\bar\phi_{ia}}{\partial x_{j_i}}(x)d\mu(x,y,z).
   \end{equation*}
  Applying the Fundamental Theorem of Calculus we get $\frac{\partial\bar\phi_{ia}}{\partial x_{j_i}}=u_a\circ\tilde\tau_i-u_a\circ\tilde\tau_{i-1}$, and then \eqref{eq:L1convergence} gives, as $a\to+\infty$, 
   \begin{multline*}
   0=\int_{\Omega\times Y\times Z}(u_{a}\circ\tilde\tau_i-u_{a}\circ\tilde\tau_{i-1})\,d\mu\\
   \to \int_{\Omega\times Y\times Z}(\chi_{\tilde\tau_i(R)}-\chi_{\tilde\tau_{i-1}(R)})\,d\mu
   =(\projX)_{\#}\mu(\tilde\tau_i(R))-(\projX)_{\#}\mu(\tilde\tau_{i-1}(R)).
 \end{multline*}
 By induction we get
 \[(\projX{})_{\#}\mu(R)=(\projX{})_{\#}\mu(\tilde\tau_{0}(R))=(\projX{})_{\#}\mu(\tilde\tau_{k}(R))=(\projX{})_{\#}\mu(\tau(R)).\qedhere\]
 \end{proof}

For a vector field $X\colon \Omega\times Y\to\R^{n+1}$, we can define
\begin{equation}\label{eq:defmuX}
{\mathcal C_\mu}(X)\coloneqq\int \langle X(x,y), 
 ( z_1,\dots,z_n,-1)\rangle\, d\mu(x,y,z).
\end{equation}
When $\mu$ is induced by a smooth function $\varphi\colon \Omega\to Y$, ${\mathcal C_\mu}(X)$ is the circulation of $X$ through the graph of $\varphi$, since $(z_1,\dots,z_n,-1)=(\frac{\partial\varphi}{\partial x_1},\dots,\frac{\partial\varphi}{\partial x_n},-1)$ is normal to the graph of $\varphi$. 

\begin{lemma}\label{lem:divergencecond}
Let $X\colon\Omega\times Y\to\R^{n+1}$ be a smooth, compactly-supported vector field  that vanishes on a neighborhood of $\partial \Omega\times Y$ and satisfies
$\operatorname{div} X=0.$
Then
\[{\mathcal C_\mu}(X)=0.\]
\end{lemma}
\begin{proof} 
Let 
\begin{equation}\label{eq:Xtilde}
\tilde X_i(x,y)=\int_{-\infty}^yX_i(x,s)ds.
\end{equation}
Then $%
\tilde X_i\in C^\infty%
(\Omega\times Y)$ and vanishes on $\partial\Omega\times Y$, so by the $i$-th entry of \eqref{eq:boundarycondition}, with $\phi=%
\tilde X_i$,%
\[
\int_{\Omega\times Y\times Z}\left(\frac{\partial\tilde X_i}{\partial x_i}(x,y) +X_i(x,y)z_i\right) d\mu(x,y,z)%
=0.
\]%
Rearranging, and plugging this into the definition of ${\mathcal C_\mu}(X)$, it follows that
\begin{align*}
 {\mathcal C_\mu}(X)&=\sum_{i=1}^n\int_{\Omega\times Y\times Z}X_i(x,y)z_i\,d\mu(x,y,z)-\int_{\Omega\times Y\times Z}X_{n+1}(x,y)d\mu(x,y,z)\\
 &=- %
 \int_{ \Omega\times Y\times Z}\sum_{i=1}^{n}\frac{\partial\tilde X_i}{\partial x_i}(x,y)%
 +X_{n+1}(x,y)\,d\mu(x,y,z)
\end{align*}
Now, using (\ref{eq:Xtilde}) and 
\[X_{n+1}(x,y)=\int_{-\infty}^y\frac{\partial X_{n+1}}{\partial x_{n+1}}(x,s)\,ds,\] 
we get
\begin{equation*}
 \sum_{i=1}^n\frac{\partial\tilde X_i}{\partial x_i}+X_{n+1}
= \int_{-\infty}^{y} \sum_{i=1}^{n+1}\frac{\partial X_i}{\partial x_i}(x,s)\,ds=\int_{-\infty}^{y}\operatorname{div} X\,ds, %
\end{equation*}
which vanishes by the assumption that $\operatorname{div} X = 0$. 
\qedhere
\end{proof}
We define, for measurable, compactly supported, and bounded functions $u\colon\Omega\times Y\to\R$,
\begin{align}
\notag S(u)&=\mu\left(0,\dots,0,\int_{y}^\infty u(x,s)ds\right)\\
 \label{eq:defS} &=-\int \left(\int_{y}^\infty u(x,s)ds\right)\,d\mu(x,y,z).
\end{align}
\begin{lemma}\label{lem:radon-nikodym} 
 The functional $S$ corresponds to integration with respect to an absolutely continuous nonpositive measure; in other words, there is a measurable function $\rho\colon\Omega\times Y\to(-\infty,0]$ such that 
\begin{align*}
 S(u)&=\int_{\Omega\times Y} u(x,y)\rho(x,y)\,dx\,dy.
\end{align*}
\end{lemma}
\begin{proof}
 This follows from the Radon-Nikodym theorem. To apply the theorem we need to check that, if $A\subset \Omega\times Y$ is a set of measure zero and $\chi_A$ is its indicator function, then $S(\chi_A)=0$. 
 Indeed, if $A\cap \{(x,y):y\in Y\}$ has zero measure for Lebesgue-almost all $x\in\Omega$, then $\int_{-\infty}^y\chi_A(x,s)ds=0$ for almost all $x\in \Omega$, and by Lemma \ref{lem:lebesgueproj} and the Fubini theorem, the integral in the definition \eqref{eq:defS} of $S(\chi_A)$ vanishes. To see that the function $\rho$ can be taken to be nonpositive, observe that whenever $u$ is nonnegative, its primitive also satisfies $\int_{-\infty}^yu(x,s)ds\geq 0$, so $S(u)\leq 0$.
\end{proof}

\begin{lemma}\label{lem:decreasingconstant}
 When restricted to a line $\{(x,y):y\in Y\}$, $x\in \Omega$, the function $y\mapsto \rho(x,y)$ is non-increasing for almost every $x\in \Omega$. If $N>0$ is such that $\operatorname{supp}\mu\subset \Omega\times (-N,N)\times Z$, then $y\mapsto\rho(x,y)$ vanishes throughout $(-\infty,{-}N]$ and is constant on $[N,+\infty)$, for almost every $x\in\Omega$.
\end{lemma}

Observe that, strictly speaking, $\rho$ is only defined Lebesgue-almost everywhere on $\Omega\times Y$, so the statement of the lemma should be interpreted as ascertaining the existence of a representative, in the 
equivalence class of measurable functions coinciding with $\rho$ Lebesgue-almost everywhere, having the desired properties.

\begin{proof}
Let $R$ be an $(n+1)$-dimensional box in $\Omega\times Y$ whose edges are parallel to the axes%
 , and let $\tau_t(x,y)=(x,y+t)$ be the translation in the $y$ direction. Then, by Lemma \ref{lem:radon-nikodym} and definitions \eqref{eq:defmuX} and \eqref{eq:defS},
 \begin{align*}
 \int_{R} \rho(x,y-t)\,dx\,dy&=\int_{\tau_t(R)} \rho(x,y)\,dx\,dy\\
 &=\int_{\Omega\times Y} \chi_{\tau_t(R)}(x,y)\rho(x,y)\,dx\,dy\\
 &=S(\chi_{\tau_t(R)})\\
 &=\mu\left(0,\dots,0,\int_{y}^\infty\chi_{\tau_t(R)}(x,s)ds\right)\\
 &=\mu\left(0,\dots,0,\int_{y-t}^{\infty}\chi_{R}(x,s)ds\right)\\
 &=-\int_{\Omega\times Y\times Z}\left(\int_{y-t}^{\infty}\chi_R(x,s)ds\right)d\mu(x,y,z).
 \end{align*}
 Since $\mu$ is a positive measure, the last term is nonincreasing in $t$. Since this is true for all $t$ and all $R$, this proves that $\rho$ is nonincreasing in the $y$ direction. This proves the first part of the lemma.
 
 To prove the second statement of the lemma, consider the case in which $R=R_\Omega\times[a,b]$ for some box $R_\Omega\subset \Omega$ and some $a<b$. Then, if $x\in R_\Omega$,
 \[\int_y^\infty\chi_R(x,s)\,ds=\begin{cases}
  b-a, &y\leq a,\\
  0,&y\geq b.
 \end{cases}
 \]
 Thus if $a<b\leq -N$, we have
 \[\int_R\rho(x,y)\,dx\,dy=-\int_{\Omega\times Y\times Z}\left(\int_y^\infty\chi_R(x,y)\,ds\right)d\mu(x,y,z)=\int_{\Omega\times Y\times Z}0\,d\mu(x,y,z).\]
 On the other hand, if $N\leq a<b$, then
  \[\int_R\rho(x,y)\,dx\,dy=-\int_{\Omega\times Y\times Z}\left(\int_y^\infty\chi_R(x,y)\,ds\right)d\mu(x,y,z)=-\int_{\Omega\times Y\times Z}(b-a)\,d\mu(x,y,z).\]
 This is invariant under translations of the interval $[a,b]$.
 This proves the second statement of the lemma.
\end{proof}

\begin{lemma}\label{lem:boundedrange}
The function $\rho$ in Lemma \ref{lem:radon-nikodym} is essentially bounded.
\end{lemma}
\begin{proof}
 Aiming for a contradiction, assume that the function $\rho\leq 0$ is not essentially bounded. 
 Then the sets $B_j=\{(x,y)\in\Omega\times Y:\rho(x,y)\leq -j\}$, $j\in\mathbb N$, have positive measure. By Lemma \ref{lem:decreasingconstant}, if we take $N>0$ to be such that $\operatorname{supp}\mu\subset\Omega\times(-N,N)\times Z$, then $y\mapsto \rho(x,y)$ is everywhere non-increasing and is constant on $[N,+\infty)$ for all $x\in\Omega$. Thus the sets $B_j\cap(\Omega\times [N,N+1])$ must have positive measure.
 Pick a subset $A_j\subset \{(x,y)\in\Omega\times [N,N+1]:\rho(x,y)\leq -j\}\subset B_j$ of finite measure $|A_j|<\infty$ and of the form $A_j=A_j^\Omega\times [N,N+1]$, with $A_j^\Omega\subset\Omega$. Observe that this means that $A_j$ does not intersect the compact set $\operatorname{supp}\mu$. Pick an open set $U_j\subset\Omega\times Y$, of the same product form, $U_j=U^\Omega_j\times[N,N+1]$, such that $A_j\subset U_j$ with $|U_j\setminus A_j|\leq |A_j|/j$, which is possible due to the outer regularity of {the} Lebesgue measure. Note that the function $f_j=\chi_{U_j}/|A_j|$ verifies $\int_{\Omega\times Y} \rho f_j\leq \int_{A_j} (-j) f_j+\int_{U_j\setminus A_j}0\,f_j= -j$. Take $\phi_j\in C^\infty_c(\Omega\times Y)$  to be any $C^\infty$ nonnegative function approximating $f_j$ well enough and satisfying
 \begin{gather}
 \label{eq:contradiction}
 \int_{\Omega\times Y}\rho\phi_j dx\,dy\leq -j/2, \\
 \notag |\projX{}(\operatorname{supp}\phi_j)|\leq 2|A^\Omega_j|, \\
 \notag \sup_{x\in\Omega}\int_{-\infty}^\infty\phi_j(x,s)ds\leq 2\frac{1}{|A_j|}=\frac2{|A^\Omega_j|}.%
 \end{gather}
 Then we have by Lemma \ref{lem:radon-nikodym}, \eqref{eq:defS}, the fact that $\mu$ and $\phi_j$ are non-negative, the bounds above, and Lemma \ref{lem:lebesgueproj},
 \begin{align*}
    -\frac j2&\geq\int_{\Omega\times Y}\phi_j(x,y)\rho(x,y)dx\,dy\\
    &=S(\phi_j)\\
    &=\mu\left(0,\dots,0,\int_{-\infty}^y\phi_j(x,s)ds\right)\\
    &=-\int_{\Omega\times Y\times Z}\int_{-\infty}^y\phi_j(x,s)ds\, d\mu(x,y,z)\\
    &\geq - \int_{\Omega\times Y\times Z}\int_{-\infty}^\infty\phi_j(x,s)ds\, d\mu(x,y,z)\\
    &\geq -\int_{\Omega\times Y\times Z}
    \frac2{|A^\Omega_j|}%
    \chi_{\projX{}(\operatorname{supp}\phi_j)}(x)d\mu(x,y,z)\\
    &=-\frac2{|A^\Omega_j|}(\projX{})_\#\mu(\projX{}(\operatorname{supp}\phi_j))\\
    &=-\frac2{|A^\Omega_j|}c|\projX{}(\operatorname{supp}\phi_j)|\\
    &\geq -\frac2{|A^\Omega_j|}c(2|A^\Omega_j|)=-4c,
 \end{align*}
 where $\projX{}$ and $c$ are as in the statement of Lemma \ref{lem:lebesgueproj}.
 This uniform bound gives the contradiction we were aiming for. We conclude that the essential range of $\rho$ is a bounded interval in $(-\infty,0]$. 
 \end{proof}
 
 We will henceforth take $\rho$ to be bounded (we may choose such a representative in its class of essentially bounded functions) and denote the range of $\rho$ by
 \[{\mathcal R}=\rho(\Omega\times Y)\subset(-\infty,0].\]
 We will also denote by $\nu$ the restriction of 
{the} Lebesgue measure to ${\mathcal R}$, sometimes denoted $\nu=\mathcal L^1\llcorner\mathcal R$.

\begin{lemma}\label{lem:dividentity}
For all smooth vector fields $X$ compactly supported in $\Omega\times Y$ and vanishing in a neighborhood of $\partial \Omega\times Y$,
\[{\mathcal C_\mu}(X)=-S(\operatorname{div} X).\]
\end{lemma}
\begin{proof}
Indeed,
\begin{align*}
 \operatorname{div}&(X+\Big(\underbrace{0,\dots,0}_n,\int_y^{\infty} \operatorname{div}X(x,s)\,ds\Big))\\
 &=\sum_{i=1}^n\frac{\partial }{\partial x_i}\left(
  X_i+0
 \right)+\frac{\partial}{\partial y}\left(X_{n+1}+\int_y^\infty\operatorname{div} X\,ds\right)\\
 &=\operatorname{div} X-\operatorname{div} X\\
 &=0.
\end{align*}
By Lemma \ref{lem:divergencecond},
\begin{align*}0&={\mathcal C_\mu}(X+\left(0,\dots,0,\int_y^\infty\operatorname{div}X\,ds\right))\\
&={\mathcal C_\mu}(X)+\mu\left(0,\dots,0,\int_y^\infty\operatorname{div}X\,ds\right)\\
&={\mathcal C_\mu}(X)+S(\operatorname{div}X).\qedhere
\end{align*}
\end{proof}

\begin{lemma}\label{lem:main}
The functions $\varphi_r:\Omega\to \R$ defined by,
\[\varphi_r(x)=\inf_{\substack{y\in Y\\\rho(x,y)\leq r}}y,\qquad r\in {\mathcal R}=\rho(\Omega\times Y)\subset(-\infty,0], \; x\in \Omega,\]
are weakly differentiable.
These functions satisfy, for all $X\in C^\infty_c(\Omega\times Y;\R^{n+1})$,
\begin{equation}\label{eq:divXvarphi}
    \int_{{\mathcal R}}\int_{\{(x,y)\in\Omega\times Y:y\geq \varphi_r(x)\}}\hspace{-20mm}\operatorname{div}X\,dx\,dy\,d\nu(r)=\int_{{\mathcal R}}\int_\Omega\langle X(x,\varphi_r(x)),(D\varphi_r(x),-1)\rangle\,dx\,d\nu(r),
\end{equation}
where $\nu$ is 
{the} Lebesgue measure restricted to ${\mathcal R}$.
\end{lemma}

Observe that, since by Lemma \ref{lem:boundedrange} $\rho$ is essentially bounded, we may take a representative in the class of $\rho$ that is bounded, and then $\varphi_r(x)$ is finite for each $x\in \Omega$.

\begin{proof}

 Consider the set $C^1_c(\Omega\times Y;\R^{n+1})$ of compactly-supported vector fields that are continuously differentiable. Observe that since these vector fields are compactly supported, they vanish on the boundary $\partial\Omega\times Y$.
 Consider also the set $B$ of vector fields $X\in C^1_c(\Omega\times Y;\R^{n+1})$ satisfying $\sup_{\Omega\times Y}\|X(x,y)\|\leq 1$.

 Since we have a uniform bound %
 for $X\in B$, by Lemma \ref{lem:radon-nikodym}, \eqref{eq:defS}, Lemma \ref{lem:dividentity}, \eqref{eq:defmuX}, the Cauchy-Schwarz inequality, and the finiteness of $\mu$ together with \eqref{eq:finitemoment s},  
 \begin{multline*}
\left|
 \int \rho\operatorname{div} X\,dx\,dy\right|
  =| S(\operatorname{div}X)|=|{\mathcal C_\mu}(X)|=\\
  \left|\int \langle X(x,y),(z,-1)\rangle\,d\mu(x,y,z)\right|
  \leq \int 1+\|z\|\,d\mu(x,y,z)<+\infty,
 \end{multline*}
 we conclude that $\rho$ is a \emph{function of bounded variation} (see \cite[Def. 5.1]{evansgariepy}) in $\Omega\times Y$.
It follows from the coarea formula \cite[Thm. 5.9]{evansgariepy} that there is a set of full measure $A\subset {\mathcal R}$, $|{\mathcal R}\setminus A|=0$, such that if $r\in A$ then $\rho^{-1}(-\infty,r]\subset \Omega\times Y$ is a \emph{set of locally finite perimeter} (\cite[Def. 5.1]{evansgariepy}, \cite[Ch. 12]{maggi2012sets}), meaning that 
 \[\sup_{X\in B}\int_{\rho^{-1}(-\infty,r]}\operatorname{div}X(x,y)\,dx\,dy<+\infty, \quad r\in A.\]
By \cite[Prop. 12.1]{maggi2012sets}, there exists an $\R^{n+1}$-valued measure $\bm\nu_r$ with bounded total variation $|{\bm\nu}_r|$ (defined in \cite[Rmk. 4.12]{maggi2012sets}), $|{\bm\nu}_r|(\overline{\Omega\times Y})<+\infty$,  such that, for $X\in C^1_c(\Omega\times Y;\R^{n+1})$,
 \[\int_{\rho^{-1}(-\infty,r]}\operatorname{div}X\,dx\,dy=\int_{\Omega\times Y}X\cdot d{\bm\nu}_r  = \sum_{i=1}^{n+1} \int_{\Omega\times Y} X_i d{\bm\nu}_{r,i}\]
 De Giorgi's Structure Theorem (\cite[Th. 5.15 and 5.16]{evansgariepy} or \cite[Th. 15.9]{maggi2012sets}) then implies that ${\bm\nu}_r$ is supported on the boundary $\partial \rho^{-1}(-\infty,r]$, that this boundary is of Hausdorff dimension $n$, and that the unit normal $\eta_r$ to the boundary of $\rho^{-1}(-\infty,r]$ is well defined for almost every point $(x,y)$ on the boundary with respect to Hausdorff measure $H^n$ of dimension $n$ by
 \begin{equation}\label{eq:defeta}
  \eta_r(x,y)=\lim_{b\searrow 0}\frac{{\bm\nu}_r(D((x,y),b))}{|{\bm\nu}_r|(D((x,y),b))}
 \end{equation}
 where $D((x,y),b)$ denotes the ball centered at $(x,y)$ of radius $b>0$ and $|{\bm\nu}_r|$ denotes the total variation of ${\bm\nu}_r$. Refer to Figure \ref{fig:2}.
 
 \begin{figure}
 \includegraphics[width=6.5cm]{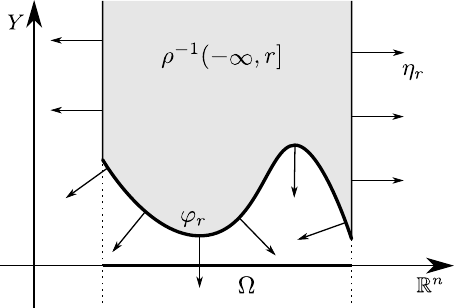}
 \centering
 \caption{The sets $\rho^{-1}(-\infty,r]$ and their exterior unit normal $\eta_r$, as in the situation of Lemma \ref{lem:main}.}
  \label{fig:2} 
 \end{figure}
 
 Also, the Gauss--Green formula holds: for {a compactly supported vector field} $X\in C^1(\overline{\Omega\times Y},\R^{n+1})$,
 \begin{equation}\label{eq:normal}
  \int_{\rho^{-1}(-\infty,r]}\operatorname{div}X\,dx\,dy=\int_{\partial \rho^{-1}(-\infty,r]} \langle X,\eta_r\rangle d H^n.
 \end{equation}
  Indeed, this is equivalent to \cite[eq. (15.11)]{maggi2012sets}, summing over all the entries in that vector-valued equation; cf. \cite[Rmk. 12.2]{maggi2012sets}.
 
 From Lemma \ref{lem:function} below and Remark \ref{rmk:function}, it follows that $H^n$-almost all the boundary $(\partial \rho^{-1}(-\infty,r])\cap (\Omega\times Y)$ corresponds to the graph of $\varphi_r$, i.e.,
 \[H^n((\partial \rho^{-1}(-\infty,r])\cap (\Omega\times Y)\setminus 
 \{(x,\varphi_r(x)):x\in\Omega\})=0.\]
 
 Let $\zeta_r\colon\Omega\to \R^n$ be the vector field whose $i$-th entry is given by
 \begin{equation}\label{eq:defzeta}
 [\zeta_r(x)]_i=-\frac{[\eta_r(x,\varphi_r(x))]_i}{[\eta_r(x,\varphi_r(x))]_{n+1}},\quad 1\leq i\leq n,
 \end{equation}
 if the denominator is $\neq 0$, and $[\zeta_r(x)]_i=\operatorname{sign}([\eta_r(x)]_i)\infty$ otherwise. It follows from Lemma \ref{lem:weakderivative} below that the denominator in \eqref{eq:defzeta} is almost-everywhere nonzero, and that $\zeta_r$ is the weak derivative of $\varphi_r$. 
 
 Equality \eqref{eq:divXvarphi} 
 follows from  \begin{align*}
      \int_{{\mathcal R}}\int_{\{y\geq \varphi_r(x)\}}\operatorname{div}X\,dx\,dy\,d\nu(r)
       &=\int_{{\mathcal R}}\int_{\partial\rho^{-1}(-\infty,r]}\langle X,\eta_r\rangle dH^n\,d\nu(r)\\
       &=\int_{{\mathcal R}}\int_{\Omega}\langle X(x,\varphi_r(x)),(\zeta_r(x),-1)\rangle dx\,d\nu(r),
  \end{align*}
  which is \eqref{eq:normal} and Lemma \ref{lem:weakderivative}\ref{it:weakderivative2}, together with $\zeta_r$ being the weak derivative  $D\varphi_r$.
\end{proof}

 \begin{lemma}\label{lem:samederivatives2}
  If $r\leq r'\leq 0$ and  $x_0\in\Omega$ is such that $\varphi_r(x_0)=\varphi_{r'}(x_0)$ and $\eta_r$ and $\eta_{r'}$ are defined at $x_0$, then $\eta_r(x,\varphi_r(x_0))=\eta_{r'}(x,\varphi_{r'}(x_0))$.
 \end{lemma}
 \begin{proof}
  This follows immediately from \cite[Th. 5.13]{evansgariepy}.
 \end{proof}
 
 \begin{lemma}\label{lem:function}
 For $r\in {\mathcal R}$, let $V_r$ be the set of points $(x,y)\in\Omega\times Y$ such that $\eta_r(x,y)$ is defined and its last coordinate is $[\eta_r(x,y)]_{n+1}=0$. For almost every $r\in {\mathcal R}$, we have 
 \[H^n(V_r)=0.\]
 \end{lemma}
 \begin{remark}\label{rmk:function}
  The statement of Lemma \ref{lem:function} means that the graph of the function $\varphi_r$ defined in the statement of Lemma \ref{lem:main} coincides with $(\partial \rho^{-1}(-\infty,r])\cap\Omega\times Y$ almost everywhere.
 \end{remark}
  \begin{figure}
 \includegraphics[width=6.5cm]{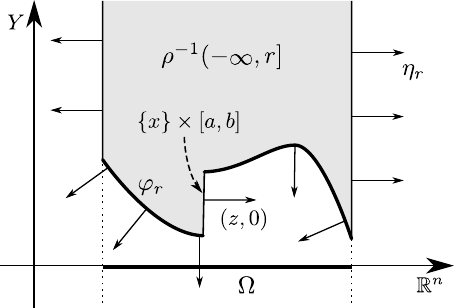}
 \centering
 \caption{When there is a vertical segment $\{x\}\times[a,b]$ 
 in the boundary $\partial \rho^{-1}(-\infty,r]$, the normal vector is horizontal, that is, of the form $(z,0)$, $z\in\R^n$. The proof of Lemma \ref{lem:function} shows that the $n$-dimensional volume of the union of these segments is zero.}
  \label{fig:3}
 \end{figure}
 \begin{proof}
 From the boundedness of $\rho$ (Lemma \ref{lem:boundedrange}) and the fact that $\lim_{y\to-\infty}\rho(x,y)=0$ and $\lim_{y\to+\infty}\rho(x,y)=\inf {\mathcal R}$ (see Lemma \ref{lem:decreasingconstant}), it follows that, for each $r\in {\mathcal R}$ and $x\in \Omega$, there is some $y\in\R$ such that $(x,y)\in \partial\rho^{-1}(-\infty,r]$. %
 Since $\rho$ is non-increasing (Lemma \ref{lem:decreasingconstant}), given $x\in\Omega$ the set of values $y\in Y$ with $(x,y)\in \partial\rho^{-1}((-\infty,r])$ must be an interval.

 {Let $I=[\inf\mathcal R,\sup\mathcal R]$ be the closed convex hull of $\mathcal R=\rho(\Omega\times Y)$; it is a compact interval (cf.comments after Lemma \ref{lem:boundedrange}).}
 By equation \eqref{eq:normal}, {the Cavalieri principle}\footnote{\label{fn:cavalieri}{If $\mathbb X\subset \R^n$ and the functions $f\colon \mathbb X\to \R_+$ and $g\colon \mathbb X\to \R$ are integrable, as is their product $fg$, then by the Fubini theorem, $\int_0^\infty\int_{f^{-1}[t,+\infty)}g(x)\,dx\,dt
 =\int_0^\infty\int_{\mathbb X}\chi_{f^{-1}[t,+\infty)}(x)g(x)\,dx\,dt
 =\int_{\mathbb X}\int_0^\infty\chi_{f^{-1}[t,+\infty)}(x)g(x)\,dt\,dx
 =\int_{\mathbb X}f(x)g(x)\,dx$. We apply this with $f=-\rho$, $g=\operatorname{div}X$, $\mathbb X=\Omega\times Y$.}}, Lemmas \ref{lem:radon-nikodym} and \ref{lem:dividentity} as well as definitions \eqref{eq:defmuX} and \eqref{eq:defS},
 \begin{align}
  \notag \int_{I}\int_{\partial \rho^{-1}(-\infty,r]} \langle X,\eta_r\rangle d H^n\,dr
  \notag &=\int_{I}\left(\int_{\rho^{-1}(-\infty,r]} \operatorname{div}X\,dx\,dy\right)dr \\
  \notag &=-\int\rho\operatorname{div}X\,dx\,dy\\
  \notag &=-S(\operatorname{div} X)\\
  \notag &={\mathcal C_\mu}(X)\\
  &=\int_{\Omega\times Y\times Z} \langle X(x,y),(z,-1)\rangle\,d\mu(x,y,z)
  \label{eq:normalid}
 \end{align}
 for all $X\in C^1_c(\Omega\times Y;\R^{n+1})$. By the density 
 of $C^1_c(\Omega\times Y;\R^n)$ in $L^1((\projXY)_\#\mu;\R^n)$ \cite[Th.~8.14]{folland1999real} and dominated convergence, this holds as well for vector fields $X$ in the latter space.
 
 Let $X=({\vec0_{\R^n}},\chi_A)$, with $A=\bigcup_rV_r$. Then, by \eqref{eq:normalid},
 \begin{equation}\label{eq:Anegligible}
 \begin{aligned}
 \mu(A)=&\int_{\Omega\times Y\times Z}\chi_A(x,y)d\mu(x,y,z)\\
 &=-\int_{\Omega\times Y\times Z}\langle X(x,y),(z,-1)\rangle\,d\mu(x,y,z)\\
 &= - \int_{I}\int_{\partial\rho^{-1}(-\infty,r]}\langle X,\eta_r\rangle\,dH^n\,dr\\
 &=-\int_{I}\int_{\partial\rho^{-1}(-\infty,r]}\chi_A[\eta_r]_{n+1}\,dH^n\,dr=0,
 \end{aligned}
 \end{equation}
 since $[\eta_r]_{n+1}(x,y)=0$ whenever $(x,y)\in V_r$; by Lemma \ref{lem:samederivatives2} this happens whenever $(x,y)\in A$ because $\eta_r(x,y)$ is independent of $r$ (among those values $r\in {\mathcal R}$ such that $y=\varphi_r(x)$, which is to say, $(x,y)\in\partial\rho^{-1}(-\infty,r]$).
 Then we have, again using \eqref{eq:normalid} and $\|\eta_r\|=1$,
 \begin{multline*}
     {\int_{\mathcal R}H^n(V_r)\,dr\leq}
     \int_{I}H^n(V_r)\,dr
     =\int_{I}\int_{\partial\rho^{-1}(-\infty,r]}\chi_A\,dH^n\,dr\\
     =\int_{I}\int_{\partial\rho^{-1}(-\infty,r]}\langle \chi_A\eta_r,\eta_r\rangle\,dH^n\,dr
     =-\int_{\Omega\times Y\times Z}\chi_A\langle \eta_r,(z,-1)\rangle\,d\mu(x,y,z)=0,
 \end{multline*}
 {where the last term vanishes because $A$ is a negligible set with respect to $\mu$, by \eqref{eq:Anegligible}.}
 This is what we wanted to show.
\end{proof}

 \begin{lemma}\label{lem:weakderivative}
 \begin{enumerate}[label=\roman*.,ref=(\roman*)]
 The following statements are true for almost every $r\in {\mathcal R}$:
  \item \label{it:weakderivative1}The vector field $\zeta_r$ defined in \eqref{eq:defzeta} is minus the weak derivative of $\varphi_r$, that is, 
  \[\int_\Omega\varphi_r(x)\nabla\phi(x)\,dx=\int_\Omega\phi(x)\zeta_r(x)\,dx\]
  for all $\phi\in C^\infty_c(\Omega)$.
  \item \label{it:weakderivative2} We also have, for $X\in C^1_c(\Omega\times Y;\R^{n+1})$, %
  \[\int_{\partial\rho^{-1}(-\infty,r]}\langle X,\eta_r\rangle dH^n
      =\int_{\Omega}\langle X(x,\varphi_r(x)),(\zeta_r(x),-1)\rangle dx.\]
  \end{enumerate}
 \end{lemma}
 \begin{proof}
 For almost every $r\in {\mathcal R}$, $\eta_r$ is well defined on almost all the boundary $\partial\rho^{-1}(-\infty,r]$, and by Lemma \ref{lem:function}, the set $V_r$ of points $(x,y)\in\partial\rho^{-1}(-\infty,r]$ with $[\eta_r(x,y)]_{n+1}=0$ has $n$-dimensional Hausdorff measure zero, $H^n(V_r)=0$. Let $r\in {\mathcal R}$ be a number that enjoys those properties.
 
 Denote by $m_1^r$ the Hausdorff measure $H^n$ on the boundary $(\Omega\times Y)\cap\partial \rho^{-1}(-\infty,r]$. The set $(\Omega\times Y)\cap\partial \rho^{-1}(-\infty,r]$ is depicted in Figure \ref{fig:2} this set is depicted as a thickened curve; since $\Omega$ is open, it does not include any points in $(\partial\Omega)\times Y$.
 Observe that $m_1^r$ also \emph{is equal to its restriction to the graph of $\varphi_r$,}
 \begin{equation}\label{eq:restriction}
  m^r_1=m^r_1|_{(\Omega\times Y)\cap\partial \rho^{-1}(-\infty,r]\setminus V_r}=m^r_1|_{\operatorname{graph}(\varphi_r)}
 \end{equation}
 because $H^n(V_r)=0$, and $V_r$ contains the part of $\partial\rho^{-1}(-\infty,r]$ corresponding to discontinuities of $\varphi_r$; see Figure \ref{fig:3}. 
 
 Denote by $m_2^r$ the pushforward of 
{the} Lebesgue measure on $\Omega$ by the map $\Psi(x)=(x,\varphi_r(x))$, that is, $m_2^r$ is a measure on $\Omega \times Y$ given by
 \[m_2^r=\Psi_{\#}dx|_{\Omega}.\] 
 
 The measure $m_2^r$ is absolutely continuous with respect to $m_1^r$. Indeed, if $A\subset\Omega\times Y$ is a measurable set of zero $m_1^r$ measure, this means that for every $\varepsilon>0$, $A$ can be covered with countably many balls $U_1,U_2,\dots$ such that $\sum_{i=1}^\infty(\operatorname{diam} U_i)^n\leq \varepsilon$, according to the definition of $n$-dimensional Hausdorff measure. The image $\projX{}(A)$ through the projection $\projX{}\colon\Omega\times Y\to\Omega$ can then be covered by the projections of the balls $\projX{}(U_i)$, which will still satisfy (for some $C>0$ dependent only on $n$)
 \[m^r_2(A)\leq m^r_2({\textstyle\bigcup_{i=1}^\infty U_i})\leq |\textstyle\bigcup_{i=1}^\infty \projX{}(U_i)|\leq C\sum_{i=1}^\infty(\operatorname{diam}\projX{}(U_i))^n=C\sum_{i=1}^\infty(\operatorname{diam}U_i)^n\leq \varepsilon,\] 
 and hence $A$ is a set of measure zero with respect to $m_2^r$. This proves the absolute continuity of $m_2^r$ with respect to $m_1^r$.
 
 By the Radon-Nikodym theorem, there is a measurable function $G_r(x,y)$ such that $m^r_2=G_rm^r_1$.
 The measure $m_2^r$ is also supported in $(\partial \rho^{-1}(-\infty,r])\cap(\Omega\times Y)$. For any measurable set $B\subset \Omega$, we have 
 \[\int_{\Omega\times Y}\chi_{B}(x)\,G_r(x,y)\,dm_1^r(x,y) =\int_{B\times Y}G_r(x,y)\,dm_1^r(x,y)=m_2^r(B\times Y)=|B|.\]
  This means that $(\projX)_\#(G_rm^r_1)$ is 
{the} Lebesgue measure on $\Omega$. From \eqref{eq:restriction} it follows that also $(\projX)_\#(G_rm^r_1|_{\operatorname{graph}(\varphi_r)})$ is 
{the} Lebesgue measure on $\Omega$, and this means that, for almost every $x\in \Omega$, $G_r(x,\varphi_r(x))>0$. Set
  \[J_r(x,y)=\begin{cases}1/G_r(x,y),&G_r(x,y)>0,\\0,&\text{else,}\end{cases}\]
  for $(x,y)\in \partial\rho^{-1}(-\infty,r]$; we then have $m_1^r=J_rm_2^r$.
 
 With $J_r$ defined like this, we have %
 for all measurable functions  $f\colon\Omega\times Y\to\R$, 
 \begin{multline}\label{eq:jacobian}
  \int_{\partial\rho^{-1}(-\infty,r]}f(x,y)\eta_r(x,y)\,dH^n=\int_{\partial\rho^{-1}(-\infty,r]}f(x,y)\eta_r(x,y)\,dm^r_1\\
  =\int_{\partial\rho^{-1}(-\infty,r]}f(x,y)\eta_r(x,y)J_r(x,y)\,dm_2^r\\=
  \int_\Omega f(x,\varphi_r(x))\eta_r(x,\varphi_r(x))J_r(x,\varphi_r(x))dx.
 \end{multline}

 From \eqref{eq:normal} it follows that
 \[\int_{\{y\geq \varphi_r(x)\}}\operatorname{div}X\,dx\,dy=\int_{\partial \rho^{-1}(-\infty,r]} \langle X,\eta_r\rangle d H^n,\quad X\in C^\infty_c(\overline{\Omega\times Y};\R^{n+1}),\]
 or equivalently, we have the vector-valued identity (proved from the above by taking $X=\phi \,e_i$ for each vector $e_i$ in the standard basis)
 \begin{equation}\label{eq:previouslydisplayedidentity}
 \int_{\{y\geq \varphi_r(x)\}}\nabla \phi(x,y)\,dx\,dy=\int_{\partial \rho^{-1}(-\infty,r]} \phi\, \eta_r\, d H^n,\quad \phi\in C^\infty_c(\overline{\Omega\times Y};\R).
 \end{equation}
 Take $N>0$ large enough that $\operatorname{supp}\mu\subset\overline\Omega\times[-N,N]$, and take a function $\psi\in C^\infty_c(\overline{\Omega\times Y})$, $0\leq\psi\leq 1$, such that $\psi(x,y)=1$ for all $|y|\leq N$. 
  Then, if we let $\mathbf n$ denote the unit normal to the boundary of $\Omega\times [0,+\infty)\subset\Omega\times Y$,  using \eqref{eq:jacobian}, and computing the derivatives below as in $\Omega\times Y$, so that for example $\nabla \phi(x)=\nabla_{x,y} \phi(x)=(\frac{\partial\phi}{\partial x_1},\dots,\frac{\partial\phi}{\partial x_n},0)$ to account for $\frac{\partial\phi}{\partial y}=0$, we have
 \begin{align*}
     \int_\Omega&\varphi_r(x)\nabla\phi(x)\,dx= \int_\Omega\int_0^{\varphi_r(x)}\nabla \phi(x)\,dy\,dx\\
     &=\int_\Omega\int_0^{\varphi_r(x)}\nabla (\psi\phi)(x,y)\,dy\,dx\\
     \notag &=\int_{\{(x,y)\in\Omega\times Y:0\leq y\leq \varphi_r(x)\}}\nabla(\psi\phi)\,dx\,dy
     -\int_{\{(x,y)\in\Omega\times Y:0\geq y\geq \varphi_r(x)\}}\nabla(\psi\phi)\,dx\,dy \\
     \notag &=\int_{\{(x,y)\in\Omega\times Y:y\geq 0\}}\nabla(\psi\phi)\,dx\,dy-\int_{\{(x,y)\in\Omega\times Y:y\geq \varphi_r(x)\}}\nabla(\psi\phi)\,dx\,dy \\
     \notag &=\int_{\partial\{(x,y)\in\Omega\times Y:y\geq 0\}} \psi\phi\mathbf{n}\,dH^n-\int_{\partial\rho^{-1}(-\infty,r]} \psi\phi\eta_r\,dH^n\\
     \notag &=\int_{\Omega\times \{0\}} \phi\mathbf{n}\,dH^n-\int_{\{(x,\varphi_r(x)):x\in\Omega\}} \phi\eta_r\,dH^n\\
     &=\int_\Omega \phi(x)e_{n+1}dx+\int_\Omega\phi(x)\eta_r(x)J_r(x,\varphi_r(x))%
     dx
 \end{align*}
  for all $\phi\in C^\infty_c(\Omega;\R)$; here, we have first used the fact that $\varphi_r(x)=\int_0^{\varphi_r(x)}dy$. Then we introduced $\psi$, and we separated the positive and negative parts of $\varphi_r$. We have expressed them in a slightly different form (see Figure \ref{fig:4}), and then passed to the boundary using \eqref{eq:previouslydisplayedidentity} and its equivalent for the domain $\{y\geq 0\}\subset \overline{\Omega\times Y}$. In the next-to-last line we used that $\psi=1$ in the  domain of integration and we have kept only the parts that do not cancel out from the fact that $\phi$ is compactly supported in $\Omega$, namely, the sets $\Omega\times \{0\}\subset\partial\{y\geq 0\}$, whose normal is $e_{n+1}$,  and  $\varphi_r(\Omega)$, whose normal is $\eta_r$;  $J_r(x)$ comes in %
  once we apply \eqref{eq:jacobian}.
  In other words, we have
  \begin{equation}\label{eq:vectident}
   \int_\Omega\varphi_r(x)\begin{pmatrix}\nabla_x \phi(x)\\0\end{pmatrix}dx
   =\int_\Omega\phi(x)e_{n+1}dx+\int_\Omega\phi(x)\eta_r(x)J_r(x,\varphi_r(x))dx,%
  \end{equation}
  where the 0 entry in the left-hand side appears because $\phi$ is independent of $y$.
  The last entry of \eqref{eq:vectident} gives
  \[0=\int_\Omega\phi(x)1\,dx+\int_\Omega \phi(x)[\eta_r(x,\varphi_r(x))]_{n+1}J_r(x,\varphi_r(x))\,dx.\]
  Since this is true for all $\phi\in C^\infty_c(\Omega)$, we conclude that, for almost every $x\in\Omega$, \[J_r(x,\varphi_r(x))=-\frac1{[\eta_r(x,\varphi_r(x))]_{n+1}},\]
  and \eqref{eq:jacobian} becomes (cf.~\eqref{eq:defzeta})
  \begin{equation}\label{eq:jacobian2}
  \int_{\partial\rho^{-1}(-\infty,r]}f(x,y)\eta_r(x,y)\,dH^n(x,y)=\int_{\Omega}f(x,\varphi_r(x))\
  \begin{pmatrix}
   \zeta_r(x)\\-1
  \end{pmatrix}dx.
  \end{equation}
  Applying \eqref{eq:jacobian2} to $f=X_i$ to obtain, in the $i$-th entry,
  \begin{equation*}%
  \int_{\partial\rho^{-1}(-\infty,r]}X_i(x,y)[\eta_r(x,y)]_i\,dH^n(x,y)=\begin{cases}\int_{\Omega}X_i(x,\varphi_r(x))
   [\zeta_r(x)]_i\,dx,& 1\leq i\leq n-1,\\
   -\int_{\Omega}X_n(x,\varphi_r(x))
   \,dx,&i=n,
   \end{cases}
  \end{equation*}  
  and adding these over all $i=1,\dots, n$ proves the identity in item \ref{it:weakderivative2}.

    \begin{figure}
 \includegraphics[width=\textwidth]{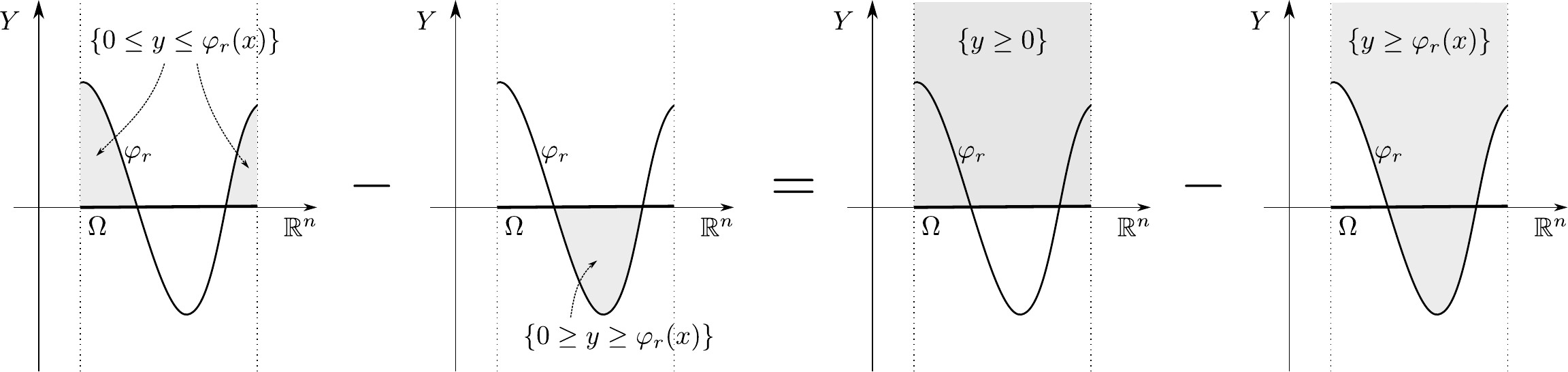}
 \centering
 \caption{Illustrating a step in the proof of Lemma \ref{lem:weakderivative}, we see that the difference of integrals on the shaded areas in the first two diagrams on the left is equal to the difference of integrals on the two  areas on the right. %
 }
  \label{fig:4}
 \end{figure}
  
  We also have, taking only the first $n$ entries in \eqref{eq:vectident},
  \[\int_\Omega\varphi_r(x)\nabla\phi(x)\,dx=\int_\Omega\phi(x)\zeta_r(x)\,dx.\]
  This is precisely the definition \eqref{eq:defweakdiff} of weak differentiability.
\end{proof}

\begin{proof}[Proof of Theorem \ref{thm:decomposition}]  
Let $\phi\colon\Omega\times Y\times Z\to\R$ be, for now, a \emph{smooth} function that is linear in $z$ and is compactly-supported in $\Omega\times Y$, vanishing in a neighborhood of $\partial \Omega\times Y$. 

By Lemma \ref{lem:boundedrange}, the function $\rho$ from Lemma \ref{lem:radon-nikodym} is bounded; its range is the bounded set ${\mathcal R\subset I=[\inf\mathcal R,\sup\mathcal R]}\subset(-\infty,0]$. The functions $\varphi_r$  in Lemma \ref{lem:main} are defined only for $r\in  {\mathcal R}$. Let $\nu$ denote the restriction of 
{the} Lebesgue measure to ${\mathcal R}$.

 Since $\phi$ is linear in $z$, for each $(x,y)\in\Omega\times Y$ the functional $z\mapsto \phi(x,y,z)$ corresponds to a vector $\tilde X(x,y)\in\R^n$ satisfying
 \[\phi(x,y,z)=\langle\tilde X(x,y),z\rangle,\quad (x,y,z)\in\Omega\times Y\times Z,\]
 and we let $X(x,y)=(\tilde X(x,y),0)\in\R^{n+1}$.
 
 Then by Lemmas \ref{lem:dividentity} and \ref{lem:main}{, and the Cavalieri principle (cf. footnote \ref{fn:cavalieri}),} we have
 \begin{align*}
  \int_{\Omega\times Y\times Z}\phi\,d\mu&=\int_{\Omega\times Y\times Z}\langle X,(z,-1)\rangle\,d\mu\\
  &={\mathcal C_\mu}(X)\\
  &=-S(\operatorname{div} X)\\
  &=-\int_{\Omega\times Y} \rho\operatorname{div} X\,dx\,dy\\
  &=\int_{I}\int_{\rho^{-1}(-\infty,r]} \operatorname{div} X\,dx\,dy\,dr\\
  &=\int_{I}\int_{\{(x,y)\in\Omega\times Y:y\geq \varphi_r(x)\}}\operatorname{div} X\,dx\,dy\,dr\\
  &=\int_{I} \int_{\Omega} \langle X(x,\varphi_r(x)),(D\varphi_r(x),-1)\rangle \,dx\,dr\\
  &=\int_{I}\int_\Omega \phi(x,\varphi_r(x),D\varphi_r(x)) %
  \,dx\,dr
 \end{align*}
 Thus the first statement of the theorem is true in the case of smooth $\phi$ linear in $z$. Defining $\mathcal Z$, $\bar\mu$, and $\mu_{xy}$ as in \eqref{eq:defZ}
 we have for all continuous functions $\phi\colon\Omega\times Y\times Z\to \R$ that are linear in $z$, 
\begin{equation}\label{eq:linearinz}
 \int_{\Omega\times Y\times Z}\phi\,d\bar\mu=\int_{\Omega\times Y\times Z}\phi\,d\mu=\int_{I}\int_\Omega\phi(x,\varphi_r(x),D\varphi_r(x))\,dx\,dr.
\end{equation}
This means that for $(\projXY)_\#\mu$-almost every $(x,y)$ we have, by Lemma \ref{lem:samederivatives2}, that, if $r$ is such that $\varphi_r(x)=y$ then
\begin{equation}\label{eq:singlepoint}
 D\varphi_r(x)=\mathcal Z(x,y).
\end{equation}

Observe that, by \eqref{eq:singlepoint}, Jensen's inequality, and \eqref{eq:finitemoments}, we have
\begin{multline*}
 \int_{{\mathcal R}}\int_\Omega\|D\varphi_r\|\,dx\,d\nu(r)
 {\leq \int_I\int_\Omega \|D\varphi_r\|dx\,dr}\\
 =\int_{\Omega\times Y\times Z}\|\mathcal Z(x,y)\|\,d\mu(x,y,z)
  \leq \int_{\Omega\times Y\times Z}\|z\|d\mu<+\infty,
\end{multline*}
whence it follows that for $\nu$-almost every $r$, we have $\varphi_r\in W^{1,1}(\Omega)$. The argument used to establish \eqref{eq:singlepoint} shows that in fact $D\varphi_r(x)$ is, for almost every $x$ and $\nu$-almost every $r$, in the convex hull of $\operatorname{supp}\mu\cap\{(x,\varphi_r(x),z):z\in Z\}$. Since $\operatorname{supp} \mu$ is compact, this implies that $\varphi_r$ is in $W^{1,\infty}(\Omega)$ for $\nu$-almost every $r$. 

For $\psi\in C_c^\infty(\Omega\times Y)$, we have, by \eqref{eq:defS}, Lemma \ref{lem:radon-nikodym}, Fubini, the Cavalieri principle (cf. footnote \eqref{fn:cavalieri}), and the fact that $y\geq \varphi_r(x)$ is equivalent to $\rho(x,y)\leq r$,
\begin{multline}\label{eq:scalarint}
    \int_{\Omega\times Y\times Z}\psi(x,y)\,d\mu(x,y,z)
    =-\int_{\Omega\times Y\times Z}\int_y^\infty \frac{\partial\psi}{\partial y}(x,y)\,d\mu(x,y,z)
    =S\left(\frac{\partial\psi}{\partial y}\right)\\
    =\int_{\Omega\times Y}\frac{\partial\psi}{\partial y}(x,y)\,\rho(x,y)\,dx\,dy
    =-\int_{\Omega\times Y}\frac{\partial\psi}{\partial y}(x,y)\int_I\chi_{\{y\geq \varphi_r(x)\}}d\nu(r)\,dx\,dy\\
    =-\int_I\int_{\Omega}\int_{\varphi_r(x)}^\infty\frac{\partial\psi}{\partial y}(x,y)\,dy\,dx\,d\nu(r)
    =\int_I\int_{\Omega}\psi(x,y)\,dx\,d\nu(r).
\end{multline}
This shows that (\ref{eq:superposition}) holds for smooth $\phi$ constant in $z$, whence adding up we get the statement for smooth $\phi$ affine in $z$. This implies that the statement holds also for continuous $\phi$  affine in $z$ by the density of $C^\infty$ functions affine in $z$ in the space of continuous functions affine in $z$ in the uniform norm on compact sets.

 The case of $\phi \in L_1(\mu)$ is proven by the following argument. Let $\mu' = \int_{{\mathcal R}} (\mathrm{id},\varphi_r,D\varphi_r)_{\#}\lambda_{\Omega}\, dr$, where $\lambda_\Omega$ is the Lebesgue measure on $\Omega$; that is, $\mu'$ is the superposition of the occupation measures generated by the functions $\varphi_r$. Then equation (\ref{eq:superposition}) reads
 \[
 \int_{\Omega\times Y\times Z} \phi\,d\mu = \int_{\Omega\times Y\times Z} \phi\,d\mu'
 \]
 for all $\phi\in L_1(\mu)$ affine in $z$. Since the result holds for all continuous $\phi$ affine in $z$ and both $\mu$ and $\mu'$ are Radon measures, it follows immediately by a classical density result that the statement holds for all $\phi \in L_1(\mu)$ independent of $z$. This implies that the $(x,y)$ marginals $\mu_{\Omega\times Y}$ and $\mu'_{\Omega\times Y}$ coincide. 
 It remains to prove the statement with $\phi\in L_1(\mu)$ linear in $z$; it suffices to consider $\phi$ of the form $\phi(x,y,z)=f(x,y)z_k$ for $f\in L^1(\mu_{\Omega\times Y})$ and $1\leq k\leq n$, because a general $\phi$ will be a linear combination of these. We already have, for continuous $f$, the identity
 \[\int_{\Omega\times Y\times Z}f(x,y)z_k\,d\mu(x,y,z)=\int_{\Omega\times Y\times Z}f(x,y)z_k\,d\mu'(x,y,z).\]
 By the same classical density result cited above applied to the signed Radon measures $z_k d\mu$ and $z_k d\mu'$, this identity holds for $f\in L^1(\mu_{\Omega\times Y})$ too. This shows that the result is true for all $\phi\in L_1(\mu)$.

The last statement of the theorem follows directly from Lemma \ref{lem:main}.
 \end{proof}
 
 \subsection{Connection with the original Hardt-Pitts decomposition}
 \label{sec:hardtpitts}
The context in which superpositions of the type described in Theorem \ref{thm:decomposition}  were first developed is that of Geometric Measure Theory, in which the main objects of interest are \emph{currents}, which are the continuous functionals on the space of smooth differential forms on an open set or a manifold. Just like distributions (continuous functionals on $C^\infty_c(U)$) can be of order higher than 1, involving integrals of derivatives of the test function, currents can also involve derivatives of the differential forms they are fed. This is why it is interesting to distinguish \emph{normal currents}, which roughly correspond to currents that can be expressed as integrals over a finite measure, evaluating the differential form on a set of vector fields, and satisfying some additional integrability conditions (see for example \cite{morgan}). Thus for example, the version of the Hardt-Pitts decomposition described in \cite{zworski1988decomposition} shows that a normal current of dimension $n$ in $\R^{n+1}$ and associated to a finite measure whose density is a positive $C^\infty$ function, and smooth vector fields satisfying an integrability condition, can be expressed as a superposition of so-called \emph{rectifiable currents} of dimension $n$. These are currents that can be written as a sum of countably many integrals over Lipschitz hypersurfaces. Our result does not require the smoothness assumptions of \cite{zworski1988decomposition}.

Let us explain how to associate a normal current $T_\mu$ to the measure $\mu$ that Theorem \ref{thm:decomposition} decomposes: for a differential form $\omega$ of order $n$ on $\Omega\times Y$, we let
\[T_\mu(\omega)=\int_{\Omega\times Y\times Z}
\omega_{(x,y)}(\tiny\begin{pmatrix}
 1\\0\\\vdots\\0\\z_1
\end{pmatrix},
\begin{pmatrix}
 0\\1\\\vdots\\0\\z_2
\end{pmatrix},
\dots,
\begin{pmatrix}
 0\\0\\\vdots\\1\\z_n
\end{pmatrix})\,d\mu(x,y,z).
\]
Similarly, to each of the sheets $\varphi_r$ %
we can associate a rectifiable current $R_r$ on $\Omega\times Y$ given by
\[R_r(\omega)=\int_{\Omega}
\omega_{(x,\varphi_r(x))}(\tiny\begin{pmatrix}
 1\\0\\\vdots\\0\\\frac{\partial\varphi_r}{\partial x_1}(x)
\end{pmatrix},
\begin{pmatrix}
 0\\1\\\vdots\\0\\\frac{\partial\varphi_r}{\partial x_2}(x)
\end{pmatrix},
\dots,
\begin{pmatrix}
 0\\0\\\vdots\\1\\\frac{\partial\varphi_r}{\partial x_n}(x)
\end{pmatrix})\,dx.
\]
Then the main conclusion of Theorem \ref{thm:decomposition} can be written as
\[T_\mu=\int_\R R_rd\nu(r),\]
an expression that roughly corresponds to \cite[eqs. (2), (8)]{zworski1988decomposition}.

\subsection{{Proof of Theorem \ref{thm:consolidated}}}
\label{sec:proof2}

In order to prove Theorem \ref{thm:consolidated}, we need a lemma.

\begin{lemma}\label{lem:jensenlemma}
\begin{enumerate}[label=\roman*.,ref=(\roman*)]
    \item \label{it:insideomega}Let $\mu$ be as in the statement of Theorem \ref{thm:consolidated}. Let $\nu$ and $(\varphi_r)$ be as in the conclusion of Theorem \ref{thm:decomposition}. Assume that $L\colon\Omega\times Y\times Z\to\R$ is measurable, and convex in $z$. Then
    \[\int_{\R}\int_\Omega L(x,\varphi_r(x),D\varphi_r(x))\,d\nu\leq \int_{\Omega\times Y\times Z}L\,d\mu.\]
    \item \label{it:boundary} Assume, additionally to the previous item, that $\mu_\partial$ is as in the statement of Theorem \ref{thm:consolidated}. Then the restriction of $\varphi_r$ to $\partial\Omega$ is a well-defined Lipschitz function, and we have, for all measurable functions $\phi\colon\partial\Omega\times Y\to\R$,
    \[\int_{\partial\Omega\times Y}\phi(x,y) d\mu_\partial(x,y)=\int_{\R}\int_{\partial\Omega}\phi(x,\varphi_r(x))d\sigma(x)\,d\nu(r),\]
    where $\sigma$ denotes Hausdorff measure on the boundary $\partial \Omega$.
     In other words, the decomposition of $\mu$ implies a decomposition of $\mu_\partial$. 
\end{enumerate}
\end{lemma}
\begin{proof}
Let us prove item \ref{it:insideomega}. Let $\mathcal Z$, $\mu_{\Omega\times Y}$, $\mu_{xy}$, and $\bar \mu$ be as in Definition \ref{def:centroid}.
It follows from Jensen's inequality that
\begin{equation}\label{eq:jensenarg}
    L(x,y,\mathcal Z(x,y))\leq \int_{Z} L(x,y,z)\,d\mu_{xy}(z)
\end{equation}
and hence also
\begin{multline}\label{eq:Jensen_mu}
\int_{\Omega\times Y\times Z} L\,d\bar\mu=\int_{\Omega\times Y} L(x,y,\mathcal Z(x,y))\,d\mu_{\Omega\times Y}(x,y)\\\leq \int_{\Omega\times Y}\int_Z L(x,y,z)d\mu_{xy}(z)d\mu_{\Omega\times Y}(x,y)=\int_{\Omega\times Y\times Z} L\,d\mu.
\end{multline}
Since the integrals of the functions $\phi$ in the statement of Theorem \ref{thm:decomposition} with respect to $\mu$ and $\bar\mu$ coincide,
the decomposition given by the theorem is the same for either of these measures; let $\nu$ be the corresponding measure, and $(\varphi_r)$ be the corresponding family of  functions. 
Thus by \eqref{eq:Jensen_mu}, the definition of $\bar\mu$, the fact that the $L(x,y,\mathcal Z(x,y))$ does not depend on $z$, \eqref{eq:scalarint}, and \eqref{eq:singlepoint},
\begin{align*}
 \int L\,d\mu&\geq \int L\,d\bar\mu\\
 &=\int L(x,y,\mathcal Z(x,y))\,d\mu_{\Omega\times Y}(x,y)\\
 &=\int L(x,y,\mathcal Z(x,y))\,d\mu(x,y,z)\\
 &=\int L(x,\varphi_r(x),\mathcal Z(x,\varphi_r(x)))\,dx\,d\nu(r)\\
 &=
 \int_\R\int_\Omega L(x,\varphi_r(x),D\varphi_r(x))\,dx\,d\nu(r)
\end{align*}
This proves item \ref{it:insideomega}.

 To prove item \ref{it:boundary}, note that the set $\Omega$ has a boundary measure  $\sigma$ supported on $\partial \Omega$ such that, if $X\in C^1(\overline\Omega;\R^n)$, then the Gauss theorem holds, that is
\begin{equation}\label{eq:gauss}
 \int_{\partial \Omega}\langle X(x),\mathbf n(x)\rangle\,d\sigma(x)=\int_\Omega\operatorname{div} X(x)\,dx,\quad \operatorname{div} X=\sum_i\frac{\partial X_i}{\partial x_i},
\end{equation}
where $\mathbf n\colon\partial\Omega\to\R^n$ is the exterior unit vector normal to $\Omega$.
Equivalently, for all $u\in C^1(\overline\Omega;\R)$ and all $\phi\in C^\infty(\overline\Omega\times Y;\R)$, taking $X(x)=e_j\phi(x,u(x))$ for each $j=1,\dots n$ in \eqref{eq:gauss}, we get,
\begin{equation}\label{eq:gauss2}
 \int_{\partial \Omega}
 \phi(x,u(x))\mathbf n(x)\,d\sigma(x)=\int_\Omega\frac{\partial\phi}{\partial x}(x,u(x))
 +\frac{\partial\phi}{\partial y}(x,u(x))Du(x)\,dx.
\end{equation}
By the density of smooth functions among the weakly-differentiable ones, and continuity of the integral, \eqref{eq:gauss2} holds also for bounded weakly differentiable functions $u$.

Remark that, since $\mu$ is compactly supported, for $\nu$-almost every $r$ the function $\varphi_r$ is bounded, as is its weak derivative $D\varphi_r$. Thus $\varphi_r\in W^{1,\infty}(\overline\Omega)$, and $\varphi_r$ is Lipschitz, as is its restriction to the boundary $\partial\Omega$. %

 We have, from  \eqref{eq:boundarymeasure}, \eqref{eq:linearinz}, and \eqref{eq:gauss2} with $u=\varphi_r$, for $f\in C^\infty(\overline \Omega\times Y;\R)$, 
 \begin{align*}
     \int_{\partial\Omega\times Y}&f(x,y)\mathbf n(x)\,d\mu_\partial(x,y)\\
     &=\int_{\Omega\times Y\times Z}\frac{\partial f}{\partial x}(x,y)+\frac{\partial f}{\partial y}(x,y)z\,d\mu(x,y,z)\\
     &=\int_\R\int_{\Omega}\frac{\partial f}{\partial x}(x,\varphi_r(x))+\frac{\partial f}{\partial y}(x,\varphi_r(x))D\varphi_r(x)\,dx\,d\nu(r)\\
     &=\int_\R\int_{\partial \Omega}f(x,\varphi_r(x))\mathbf n(x)\,d\sigma(x)\,d\nu(r).
 \end{align*}
 Let $\phi\in C^\infty(\overline\Omega\times Y;\R)$.
 Letting $f=\phi\mathbf n_i$, where $\mathbf n=(\mathbf n_1,\dots,\mathbf n_n)$, we get
 \[\int_{\partial\Omega\times Y\times Z}\phi(x,y)\mathbf n_i(x)\mathbf n(x)\,d\mu_\partial(x,y)=
 \int_\R\int_{\partial \Omega}\phi(x,\varphi_r(x))\mathbf n_i(x)\mathbf n(x)\,d\sigma(x)\,d\nu(r).\]
 Summing over the $i$-th entries, we get integrals of $\phi(x,y)\langle\mathbf n(x),\mathbf n(x)\rangle=\phi(x,y)$, thereby proving item \ref{it:boundary}.
\end{proof}

\begin{proof}[Proof of Theorem \ref{thm:consolidated}]
Define $\mathcal C=F^{-1}(0)\cap G^{-1}((-\infty,0])$. Clearly $(x,y,z) \in \mathcal C$ if and only if $F(x,y,z) = 0$ and $G(x,y,z) \le 0$. %
Assumption \ref{U:synthetic} means that $\mathcal C\cap ((x,y)\times Z)$ is convex for each $(x,y)\in \Omega\times Y$. Assumption \ref{U:closedconditions} means that $\mathcal C$ and $\mathcal C_\partial$ are closed sets.

Define also $\mathcal C_\partial=F_\partial^{-1}(0)\cap G_\partial^{-1}((-\infty,0])$ and observe that $(x,y)\in\mathcal C_\partial$ if, and only if, $F_\partial(x,y) = 0$ and $G_\partial(x,y) \le 0$. 

The total mass $\nu(\R)$ of the measure $\nu$ is  1 because
\[\nu(\R)|\Omega|=\int_\R \int_\Omega dx\,d\nu(r)%
=\int_{\Omega\times Y\times Z}d\mu=\mu(\Omega\times Y\times Z),\]
and we assumed $\mu(\Omega\times Y\times Z)=|\Omega|$.
Hence we also have, using Lemma \ref{lem:jensenlemma},
\begin{multline}\label{eq:averagearg}
 \inf_{r\in\operatorname{supp}\nu}\int_\Omega L(x,\varphi_r(x),D\varphi_r(x))\,dx
 +\int_{\partial\Omega} L_\partial(x,\varphi_r(x))\,d\sigma(x)\\
 \leq\frac1{\nu({\R})}
 \int_\R \int_\Omega L(x,\varphi_r(x),D\varphi_r(x))\,dx\,d\nu(r)+\frac1{\nu({\R})}
 \int_\R \int_{\partial\Omega} L_\partial(x,\varphi_r(x))\,d\sigma(x)\,d\nu(r)\\
 \leq\frac1{\nu({\R})} \int L\,d\mu+\frac1{\nu({\R})}\int L_\partial\,d\mu_\partial=\int L\,d\mu+\int L_\partial\,d\mu_\partial.
\end{multline}
This means that the set $I_1$ of values of $r$ such that $\varphi_r$ satisfies \eqref{eq:barphiL1} has positive measure $\nu(I_1)>0$.

For $\nu$-almost every $r$ and almost every $x\in\Omega$, the point $(x,\varphi_r(x))$ is in the support of $(\projXY)_\#\mu$, for if we take $\phi\in C^0(\Omega\times Y)$ then, by \eqref{eq:scalarint},
\[\int_{\Omega\times Y}\phi\,d(\projXY)_\#\mu=\int_{\Omega\times Y}\phi\,d\mu=\int_\R\int_\Omega\phi(x,\varphi_r(x))\,dx\,d\nu(r).\]
From the argument leading to \eqref{eq:singlepoint}, it follows that for $\nu$-almost every $r$ and almost every $x\in\Omega$ we have $(x,\varphi_r(x),D\varphi_r(x))=(x,\varphi_r(x),\mathcal Z(x,\varphi_r(x)))$. This point is in $\mathcal C$ because $\mathcal Z(x,\varphi_r(x))$ is in the convex hull of $\operatorname{supp}\mu\cap ((x,\varphi_r(x))\times Z)$, and the latter is contained in the convex set $\mathcal C\cap ((x,\varphi_r(x))\times Z)$. Let $I_2$ be the set of values of $r$ such that $(x,\varphi_r(x),D\varphi_r(x))\in\mathcal C$ for almost every $x\in\Omega$; we have shown that $\nu(I_2)=1$.

Also, \eqref{eq:Gcond} and the  decomposition of $\mu_\partial $ from Lemma \ref{lem:jensenlemma}\ref{it:boundary} imply that the set $I_3$ of values of $r$ such that, for $\sigma$-almost every $x\in \partial\Omega$ we have
 $ (x,\varphi_r(x))\in\mathcal C_\partial$,
satisfies $\nu(I_3)=1$.

We thus have that $\nu(I_1\cap I_2\cap I_3)>0$. Pick $r_0\in I_1\cap I_2\cap I_3$, and set $\bar\varphi=\varphi_{r_0}$. Then $\bar\varphi$ satisfies \eqref{eq:barphiL1}--\eqref{eq:barphiG}.

 To prove item \ref{it:gi}, note that $C^\infty(\Omega)\cap W^{1,\infty}(\overline\Omega)$ is dense in $W^{1,\infty}(\overline\Omega)$, so we may take the functions $g_i$ to be equal to $\bar\varphi$ on the boundary $\partial\Omega$ and smooth in $\Omega$;
 for example, we can take a mollifier $\psi\colon \R^n\to \R_{\geq0}$, $\psi\in C^\infty(\R^n)$ supported in the unit ball and verifying $\int_{\R^n}\psi=1$, and take $h\in C^\infty(\Omega)\cap W^{1,\infty}(\overline\Omega)$ such that $0<h(x)<\operatorname{dist}(x,\partial\Omega)/2$, and define $h(x)=0$ for $x\in\partial\Omega$. Then
 \[g_i(x)=\begin{cases}
 \frac{i^n}{h(x)^n}\int_{\R^n}\psi\left(i\frac{x-y}{h(x)}\right)\,\bar\varphi(y)\,dy,&x\in\Omega,\\
 \bar\varphi(x),&x\in\partial \Omega.
 \end{cases}
 \]
 This makes $g_i$ into a convolution of $\bar\varphi$ with a smooth kernel that approximates the Dirac delta as $i\to+\infty$ that is supported inside of $\Omega$ ($h$ guarantees this). From this definition and  properties \eqref{eq:barphiL1}--\eqref{eq:barphiG} of $\bar\varphi$, together with the continuity of $F$ and $G$, it follows that \eqref{eq:Lineq}-- \eqref{eq:limG} also hold. We may differentiate $\psi$ infinitely many times inside the integral sign, by the dominated convergence theorem, so $g_i\in C^\infty(\Omega)$.  
 
 Let us prove that $g_i$ is Lipschitz on $\overline \Omega$. Since $\bar\varphi\in W^{1,\infty}(\overline\Omega)$, it is Lipschitz, and we will denote its Lipschitz constant by $\ell$. For $x_1,x_2\in\overline\Omega$, we have three cases. First, if $x_1,x_2$ are both in $\partial\Omega$, then 
 \[|g_i(x_1)-g_i(x_2)|=|\bar\varphi(x_1)-\bar\varphi(x_2)|\leq \ell|x_1-x_2|.\]
 Next, if $x_1,x_2\in \Omega$ and $H$ is the Lipschitz constant of $h$, then
 \begin{align*}
     &|g_i(x_1)-g_i(x_2)|\\
     &=\left|\frac{i^n}{h(x_1)^n}\int_{\R^n}\psi\left(i\frac{x_1-y}{h(x_1)}\right)\,\bar\varphi(y)\,dy-\frac{i^n}{h(x_2)^n}\int_{\R^n}\psi\left(i\frac{x_1-y}{h(x_2)}\right)\,\bar\varphi(y)\,dy\right|\\
     &=\left|\frac{i^n}{h(x_1)^n}\int_{\R^n}\psi\left(i\frac{y}{h(x_1)}\right)\,\bar\varphi(x_1-y)\,dy-\frac{i^n}{h(x_2)^n}\int_{\R^n}\psi\left(i\frac{y}{h(x_2)}\right)\,\bar\varphi(x_2-y)\,dy\right|\\
     &\leq 
     \left|\frac{i^n}{h(x_1)^n}\int_{\R^n}\psi\left(i\frac{y}{h(x_1)}\right)\,\bar\varphi(x_1-y)\,dy-\frac{i^n}{h(x_1)^n}\int_{\R^n}\psi\left(i\frac{y}{h(x_1)}\right)\,\bar\varphi(x_2-y)\,dy\right|\\
     &\quad+\left|\frac{i^n}{h(x_1)^n}\int_{\R^n}\psi\left(i\frac{y}{h(x_1)}\right)\,\bar\varphi(x_2-y)\,dy-\frac{i^n}{h(x_2)^n}\int_{\R^n}\psi\left(i\frac{y}{h(x_2)}\right)\,\bar\varphi(x_2-y)\,dy\right|\\
     &\leq \ell\|x_1-x_2\|\frac{i^n}{h(x_1)^n}\int_{\R^n}\psi\left(i\frac{y}{h(x_1)}\right)dy \\
     &\quad +\left|\frac{i^n}{h(x_1)^n}\int_{\R^n}\psi\left(i\frac{y}{h(x_1)}\right)\,\bar\varphi(x_2-y)\,dy-\frac{i^n}{h(x_1)^n}\int_{\R^n}\psi\left(i\frac{u}{h(x_1)}\right)\,\bar\varphi\left(x_2-u\frac{h(x_2)}{h(x_1)}\right)\,du\right|\\
     &\leq \ell\|x_1-x_2\|+\ell\sup_{\|y\|\leq h(x_1)/i}\|(x_2-y)-(x_2-y\tfrac{h(x_2)}{h(x_1)})\|\frac{i^n}{h(x_1)^n}\int_{\R^n}\psi\left(i\frac{y}{h(x_1)}\right)dy  \\
     &\leq \ell\|x_1-x_2\|+\ell|h(x_1)-h(x_2)|/i\\
     &\leq (\ell+\ell H/i)\|x_1-x_2\|
 \end{align*}
 where we used the change of variables $u=yh(x_1)/h(x_2)$.
 Similarly, if, say, $x_1\in\partial\Omega$ and $x_2\in\Omega$, we have (and this is our last case),
 \begin{align*}
  &|g_i(x_1)-g_i(x_2)|\\
     &=\left|\bar\varphi(x_1)-\frac{i^n}{h(x_2)^n}\int_{\R^n}\psi\left(i\frac{x_1-y}{h(x_2)}\right)\,\bar\varphi(y)\,dy\right|\\
     &=\left|\frac{i^n}{h(x_2)^n}\int_{\R^n}\psi\left(i\frac{y}{h(x_2)}\right)\,\bar\varphi(x_1)\,dy-\frac{i^n}{h(x_2)^n}\int_{\R^n}\psi\left(i\frac{y}{h(x_2)}\right)\,\bar\varphi(x_2-y)\,dy\right|\\
     &\leq \ell(\|x_1-x_2\|+h(x_2)/i)\frac{i^n}{h(x_2)^n}\int_{\R^n}\psi\left(i\frac y{h(x_2)}\right)dy\\
     &\leq \ell(\|x_1-x_2\|+|h(x_2)-h(x_1)|/i)\\
     &\leq(\ell+\ell H/i)\|x_1-x_2\|
 \end{align*}
 since $h(x_1)=0$ in this case.
 Thus indeed $g_i\in W^{1,\infty}(\overline\Omega)$.
 This concludes the proof of item \ref{it:gi}. %
 \end{proof}

\section{Positive gap in codimensions greater than one} 
\label{sec:positive gap}
In this section we construct an explicit example of a Lagrangian $L$ that exhibits a positive gap between the classical and relaxed solution in codimension two (i.e., $ m =\mathrm{dim}(Y) = 2$). The Lagrangian constructed is \emph{strictly convex} in $z$ and of class $C^{1,1}_{\mathrm{loc}}$. The construction extends to codimensions greater than two and can be modified to provide a higher degree of differentiability of $L$.

Let $\Omega=B(0,1)$ be the unit ball in $\R^2$, $Y=\R^2$, $Z=\R^{2\times 2}$. Denote by $W^{1,2}(\Omega;\R^2)$ the Sobolev space weakly differentiable functions on $\Omega$, with values in $\R^2$, and whose derivative is in $L^2(\Omega;Z)$. Let $\mathcal M$ denote the set of pairs $(\mu,\mu_\partial)$ of relaxed occupation measures and their boundary measures, as in  Definition \ref{def:M}.%

We say a function is of class $C^{1,1}_{\mathrm{loc}}$ if it is continuously differentiable and its derivative is Lipschitz continuous on each compact set.

\begin{theorem}\label{thm:gap}
 There is a function $L\colon\Omega\times Y\times Z\to\R$ of class $C^{1,1}_{\mathrm{loc}}$, strictly convex in $z$, and such that 
 \begin{equation}\label{eq:gap}
 \inf_{h\in W^{1,2}(\Omega)}\int_{\Omega}L(x,h(x),Dh(x))\,dx>\min_{(\mu,\mu_\partial)\in\mathcal M} \int_{\Omega\times Y\times Z} L\,d\mu.
 \end{equation}
\end{theorem}

\begin{remark}
 In our construction below, it will be clear that while $L$ is convex in $z$, it is not convex in $\Omega$ or in $Y$. Also, by replacing the exponent $3$ by larger integers $p\geq 4$ in \eqref{eq:example} below, examples of arbitrarily high regularity $C^{p-2}$ can be obtained.
\end{remark}

\paragraph{Construction of $L$.}
Define a set-valued map $f\colon\Omega\rightrightarrows Y \subset \mathbb{R}^2$ by
\begin{equation}\label{eq:example}
 f(x)=\{r^3(\cos \tfrac\theta2,\sin \tfrac\theta2):x=r(\cos\theta,\sin\theta), r\geq 0,\theta\in\R\}, \quad x\in\R^2,
\end{equation}
so that $f$ is essentially a modified version of the complex square root {$\sqrt {re^{i\theta}}=\pm\sqrt {r}e^{i\theta/2}=\pm\sqrt r(\cos\frac\theta2+i\sin\frac\theta2)$}, where we have replaced $\sqrt r$ by $r^3$. If $x\neq0$, $f(x)$ consists of exactly two points in $\mathbb{R}^2$.

Let, for $k=0,1$,
\begin{align*}
 \uu_k(r(\cos \theta,\sin\theta))&=(-1)^kr^3\left(\cos\frac\theta2,\sin\frac\theta2\right)\\
 &=r^3\left(\cos\frac{\theta+2\pi k}2,\sin\frac{\theta+2\pi k}2\right),\quad r\in[0,1),\theta\in[0,2\pi).
\end{align*}
Thus $f(x)=\{\uu_0(x),\uu_1(x)\}$ and $\uu_0(x)=-\uu_1(x)$. See Figure \ref{fig:6}.

    \begin{figure}
 \includegraphics[width=0.6\textwidth]{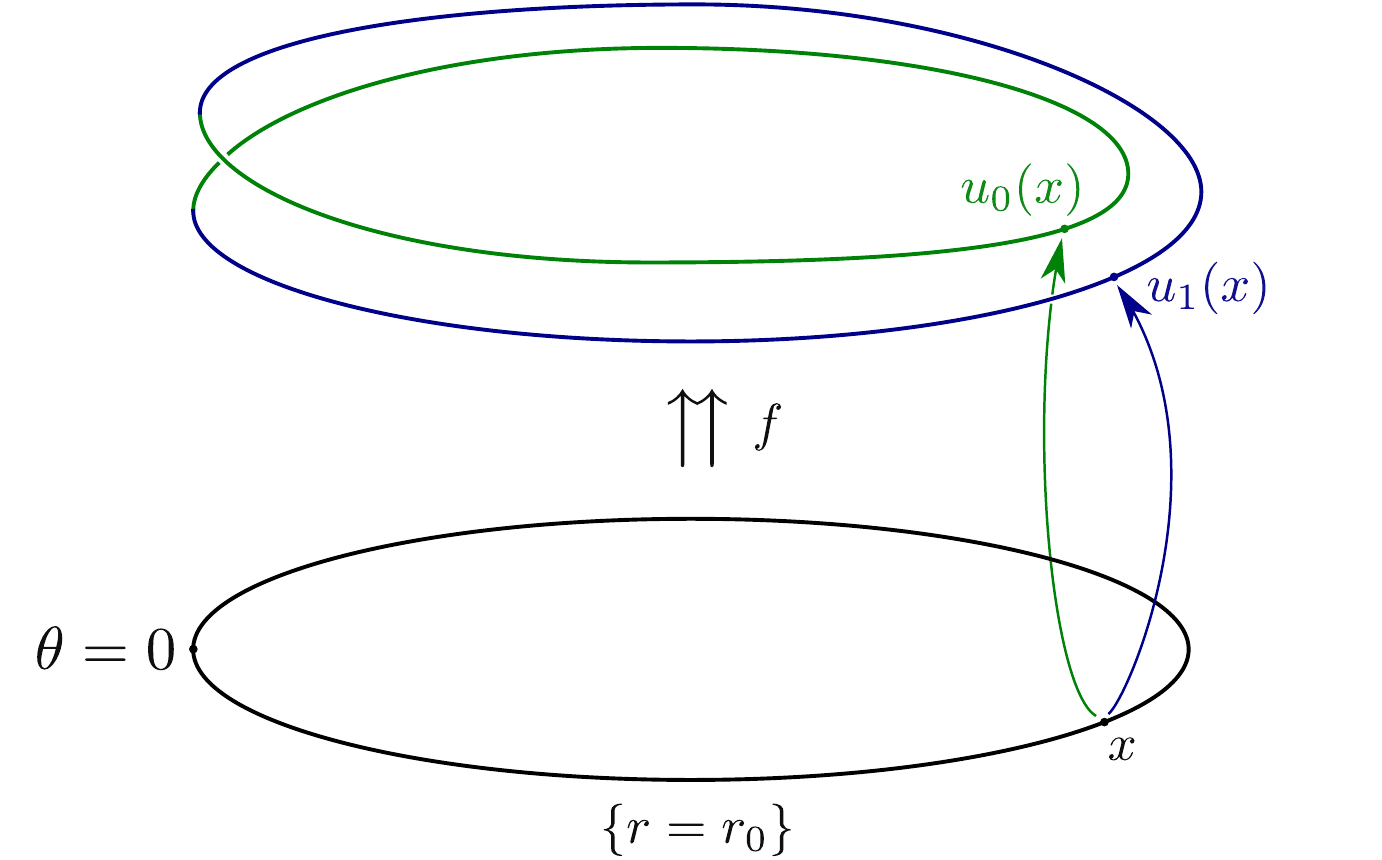}
 \centering
 \caption{This very rough scheme captures only the topological aspect of the situation to illustrate the fact that the image under the set-valued map $f$ of each circle $\{r=r_0\}$, $0<r_0<1$, is a twice-winding, non-self-intersecting cycle; on top of each point $x$ on the circle there are two points, $\uu_0(x)$ and $\uu_1(x)$. We have also marked the point corresponding to angle $\theta=0$ that is mapped to the interface between the parameterizations $\uu_0$ and $\uu_1$ of the image of $f$. %
 More accurate depictions of the situation can be found in Figures \ref{fig:5}--\ref{fig:9}.}
  \label{fig:6}
 \end{figure}

Let
\begin{equation}\label{eq:Delta}
\Delta=\{(x,y)\in \Omega\times Y:x\neq0,\;|\langle y,\uu_0(x)\rangle|> \|x\|^6/10\}.
\end{equation}
Note that $\|\uu_i(x)\|^2=\|x\|^6$, so the graph of $f$ is contained in $\Delta$; see Figures \ref{fig:5} and \ref{fig:7}. Also, for each $0\neq x\in \Omega$, the set of points $y\in Y$ with $(x,y)\in \Delta$ has two connected components corresponding to the sign of the inner product $\langle y,\uu_0(x)\rangle$. 

In order to define an auxiliary function $\psi\colon(\Omega\setminus\{0\})\times Y\to[0,1]\in C^\infty$ that will be of great utility, pick a function $\rho\in C^\infty(\R;[0,1])$ such that $\rho(r)=1$ for all $r\geq 1$ and $\rho(-r)=1-\rho(r)$, and let
\[\psi(x,y)=\rho\left(\frac{10\langle y,\uu_0(x)\rangle}{\|x\|^6}\right),\qquad (x,y)\in\Omega\times Y,\;x\neq 0.\]
Then
\begin{itemize}
\item $\psi(x,y)=1$ for $(x,y)\in\Delta$ with $\langle y,\uu_0(x)\rangle>0$, and %
 \item $\psi(x,y)=0$ for $(x,y)\in\Delta$ with  $\langle y,\uu_0(x)\rangle<0$.%
\end{itemize}

For later use we record the following properties of $\psi$ (see Figure \ref{fig:5}):
\begin{lemma}\label{lem:psi}
\begin{enumerate}[label=\roman*.,ref=(\roman*)]
    \item \label{it:psibound}$ |\psi(x,y)|\leq 1$.
    \item \label{it:Usmooth} the function
\[U(x,y)=\begin{cases}\psi(x,y)\uu_0(x)+(1-\psi(x,y))\uu_1(x),&x\neq 0,\\
0,&x=0,
\end{cases}\]
is smooth on $(\Omega\setminus\{0\})\times Y$, and can be alternatively written as
    \[U(x,y)=\begin{cases} (2\psi(x,y)-1)\uu_0(x), & x\neq 0,\\
    0,&x=0,
    \end{cases}\]
    because $\uu_0=-\uu_1$, and verifies
    \[\|U(x,y)\|=O(\|x\|^3)\]
    as $x\to 0$.
    \item \label{it:Utrivial} On $\Delta$, the function $U(x,y)$ coincides either with $\uu_0(x)$ or with $\uu_1(x)$, whichever is closest to $y$. 
    \item \label{it:Vsmooth} For $i=0,1$, let $D\uu_i$ be the $2\times 2$ matrix
\[D\uu_i=\left(\frac{\partial \uu_i}{\partial x_1} , \frac{\partial \uu_i}{\partial x_2}\right),\]
except at the points of the form $(a,0)$, $a\geq0$, where this is not defined; we define $D\uu_i$ there by extending it continuously from above, namely, 
\begin{equation*}\label{eq:defDu}D\uu_i(a,0)\coloneqq(-1)^ia^2\begin{pmatrix}
 3& 0\\ 0&1/2
\end{pmatrix},\quad a\geq0.
\end{equation*}
The function \[
    V(x,y)=\begin{cases}
      \psi(x,y)D\uu_0(x)+(1-\psi(x,y))D\uu_1(x),&x\neq 0,\\
      0,&x=0,
    \end{cases}
    \]
    is smooth on $(\Omega\setminus\{0\})\times Y$, and 
    \[\|V(x,y)\|=O(\|x\|^2)\]
    as $x\to 0$.
    \item \label{it:Vtrivial} On $\Delta$, the function $V(x,y)$ coincides either with $D\uu_0(x)$ or with $D\uu_1(x)$, according to whether $\uu_0(x)$ or $\uu_1(x)$ is closest to $y$, respectively.
\end{enumerate}
\end{lemma}
\begin{proof}
 Using Lemma \ref{lem:lemregularity} below with $u=\uu_0$ and then again with $u=D\uu_0$, we see that $U(x,y)=(2\psi(x,y)-1)\uu_0(x)$ and $V(x,y)=(2\psi(x,y)-1)D\uu_0(x)$ are smooth on $(\Omega\setminus\{0\})\times Y$. The rest of the lemma is clear from the definitions.
\end{proof}

\begin{lemma}\label{lem:lemregularity}
 Let $k>0$ and let $u\colon\Omega\setminus \{(a,0):a\geq 0\}\to\R^k$ be a smooth function such that, for all derivatives $\partial^Iu$ of $u$, of any order including zero, we have that the following limits exist and satisfy
 \[\lim_{\substack{\bar a\to a\\b\searrow0}}\partial^Iu(\bar a,b)=-\lim_{\substack{\bar a\to a\\b\nearrow0}}\partial^Iu(\bar a,b),\quad a> 0.\]
 Assume additionally that \begin{equation}\label{eq:slitcontinuity} 
 u(a,0)=\lim_{b\searrow0}u(a,b),\quad a>0.
 \end{equation}
 Then $(2\psi(x,y)-1)u(x)$ is $C^{\infty}$ on $(\Omega\setminus\{0\})\times Y$.
\end{lemma}
\begin{proof}
 Fix $y\in Y$ and $a>0$. Take sequences $(a_i)\subset\R$, $(b_i)\subset\R_{>0}$, $(y_i)\subset\R^2$ such that $a_i\to a$, $b_i\searrow0$, $y_i\to y$. We have, using $\rho(r)=1-\rho(-r)$, for every multi-index $I$, and every $a>0$ and $y\in Y$,
\begin{align*}
  \lim_{\substack{\bar a\to a\\ b\searrow 0\\\bar y\to y}}\partial^I[(2\psi((\bar a,b),\bar y)-1)u(\bar a,b)]
  &=\lim_{i\to+\infty}\partial^I[(2\psi((a_i,b_i),y_i)-1)u(a_i,b_i)]\\
  &=\lim_{i\to+\infty}\partial^I[(2\rho\left(\frac{10\langle y_i,\uu_0(a_i,b_i)\rangle}{\|(a_i,b_i)\|^6}\right)-1)u(a_i,b_i)]\\
  &=\lim_{i\to+\infty}\partial^I[(2(1-\rho\left(-\frac{10\langle y_i,\uu_0(a_i,b_i)\rangle}{\|(a_i,b_i)\|^6}\right))-1)u(a_i,b_i)] \\
  &=\lim_{i\to+\infty}
  \partial^I[(2(1-\rho\left(\frac{10\langle y_i,-\uu_0(a_i,b_i)\rangle}{\|(a_i,b_i)\|^6}\right))-1)u(a_i,b_i)] \\
    &=\lim_{i\to+\infty}\partial^I[(2(1-\rho\left(\frac{10\langle y_i,\uu_0(a_i,-b_i)\rangle}{\|(a_i,b_i)\|^6}\right))-1)u(a_i,b_i)] \\
  &=\lim_{i\to+\infty}\partial^I[-(2\rho\left(\frac{10\langle y_i,\uu_0(a_i,-b_i)\rangle}{\|(a_i,b_i)\|^6}\right)-1)u(a_i,b_i)] \\
   &=\lim_{i\to+\infty}\partial^I[(2\rho\left(\frac{10\langle y_i,\uu_0(a_i,-b_i)\rangle}{\|(a_i,-b_i)\|^6}\right)-1)u(a_i,-b_i)] \\
    &=\lim_{i\to+\infty}\partial^I[(2\psi((a_i,-b_i),y_i)-1)u(a_i,-b_i)]\\
    &=\lim_{\substack{\bar a\to a\\ b\nearrow 0\\\bar y\to y}}\partial^I[(2\psi((\bar a,b),\bar y)-1)u(\bar a,b)]
  \end{align*} 
  This means that all derivatives of $(2\psi(x,y)-1)u(x)$ exist on $\{(a,0):a> 0\}$. 
  A similar calculation, together with \eqref{eq:slitcontinuity}, shows that $(2\psi(x,y)-1)u(x)$ is continuous. This shows that $(2\psi(x,y)-1)u(x)$ is $C^\infty$ on $(\Omega\setminus\{0\})\times Y$, as the continuity of the partial derivatives near a given point implies their existence at the point.
\end{proof}
    \begin{figure}
 \includegraphics[width=0.7\textwidth]{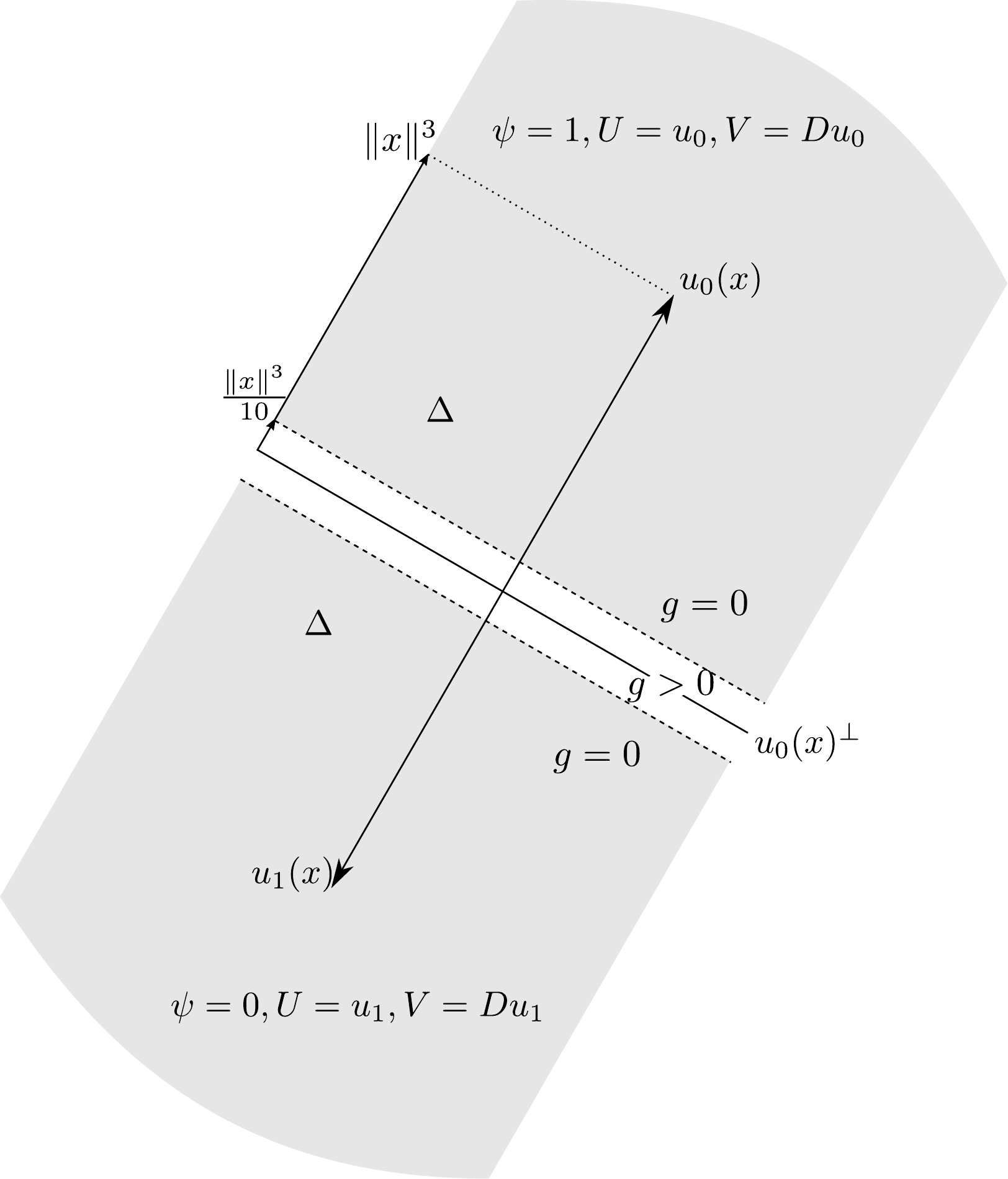}
 \centering
 \caption{For $x\in\Omega$, this is the plane $\{x\}\times Y$. We have shaded the region $\Delta$, and indicated the vectors $\uu_0(x)$ and $\uu_1(x)=-\uu_0(x)$, together with their length, $\|x\|^3$, and the distance from $\Delta$ to the origin, $\|x\|^3/10$. We have also indicated what the values of $\psi$, $U$, and $V$ are on each of the connected components of $\Delta\cap (\{x\}\times Y)$. We have also included a reminder that $g$ (defined just after Lemma \ref{lem:psi}) is positive only outside of $\Delta$. }
  \label{fig:5}
 \end{figure}

    \begin{figure}
 \includegraphics[width=0.35\textwidth]{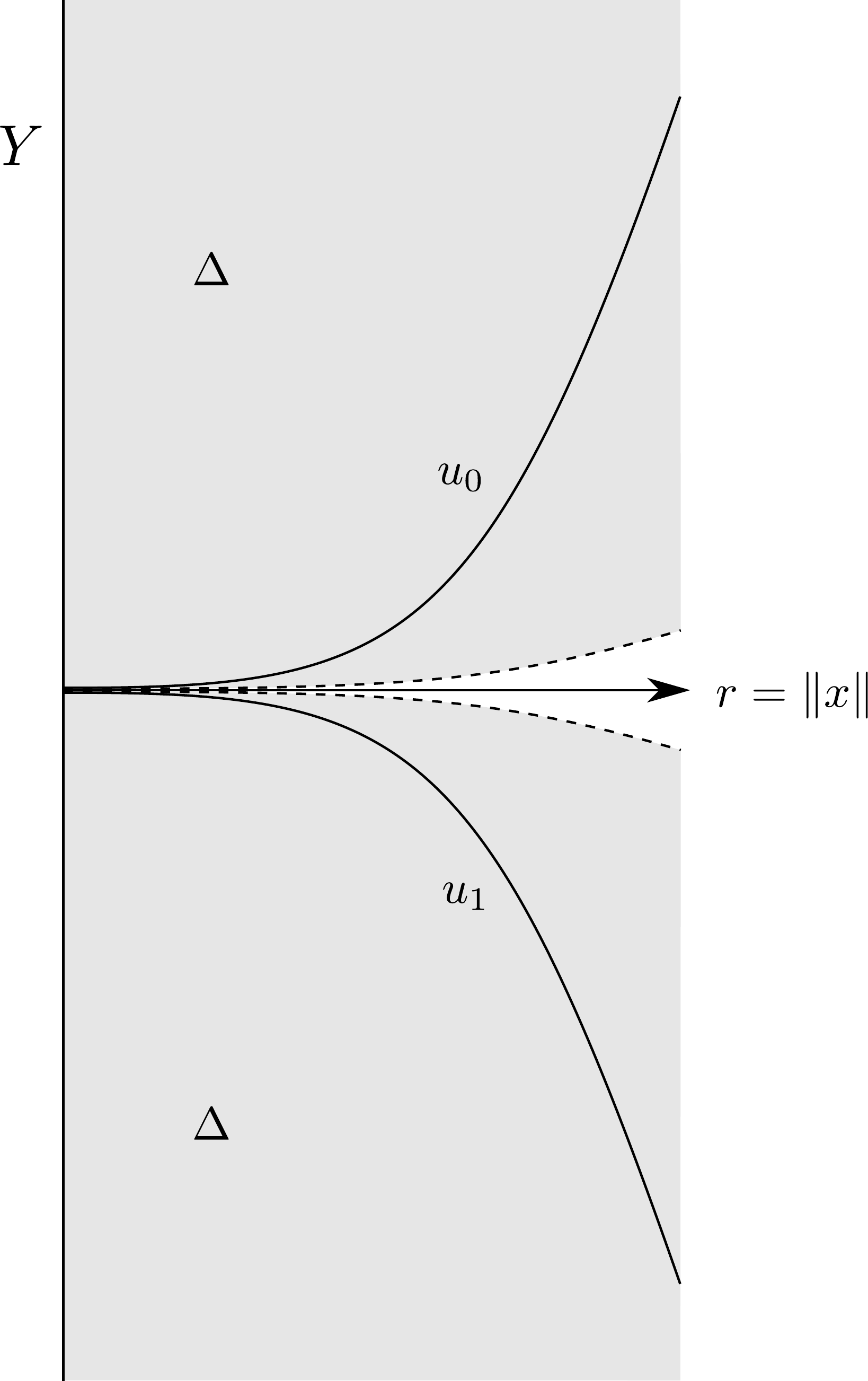}
 \centering
 \caption{Radial scheme of the graph of $f$, made up of those of $\uu_0$ and $\uu_1$, and of $\Delta$ (shaded, with dashed boundary).}
  \label{fig:7}
 \end{figure}

Take also a positive function $g\colon\Omega\times Y\to\R$ that will be auxiliary at helping us force minimizers of the proposed Lagrangian $L$ (to be defined below) to be supported in $\Delta$. We take $g$ such that
\begin{itemize}
\item $g\in C^\infty(\Omega\times Y)$,
\item $g(x,y)\geq0$,
\item $g(x,y)=0$ for all $(x,y)\in \Delta$, and
\item $g$ verifies
\begin{equation}\label{eq:grel}
 \|y-U(x,y)\|^2+g(x,y)\geq \min_{i\in\{0,1\}}\|y-\uu_i(x)\|^2
\end{equation}
if $(x,y)\notin \Delta$.  

Observe that, by Lemma \ref{lem:psi}\ref{it:Usmooth}--\ref{it:Utrivial}, the function 
\[S(x,y)\coloneqq\min_i\|y-\uu_i(x)\|^2-\|y-U(x,y)\|^2\] 
vanishes on $\Delta$  and is smooth everywhere except at the locus of  points $(x,y)$ with $\langle y,u_0(x)\rangle=0$, since it is there that $\|y-\uu_0(x)\|=\|y-\uu_1(x)\|$; indeed, this is a consequence of the calculation
\begin{align*}
    \|y-\uu_1(x)\|^2-\|y-\uu_0(x)\|^2=2\langle y,\uu_0(x)-\uu_1(x)\rangle+\|\uu_1(x)\|^2-\|\uu_0(x)\|^2=4\langle y,\uu_0(x)\rangle,
\end{align*}
which is true since $\uu_0=-\uu_1$.
Also, by Lemma \ref{lem:psi}\ref{it:Usmooth}, $S(x,y)=O(\|x\|^3)$ as $x\to 0$.  Thus in order to get a function $g$ that complies with inequality \eqref{eq:grel}, it suffices to take $g$ equal to $S$ %
in a small neighborhood of $\Delta$ while ensuring that it remains  $\geq S$ everywhere.
\end{itemize}
The function $g$ will force the minimizers to be supported within $\Delta$. Remark that $g(0,0)=0$ because $g$ is $C^\infty$, $g$ vanishes on $\Delta$, and $(0,0)\in\overline \Delta$.

Now we can define $L\colon \Omega\times Y\times Z\to\R$ to be given by
\begin{equation}\label{eq:longdefL}
    L(x,y,z)=\|y-U(x,y)\|^2+\|z-V(x,y)\|^2+g(x,y);
\end{equation}
in other words,
\begin{multline*}
 L(x,y,z)=\|y-\psi(x,y)\uu_0(x)-(1-\psi(x,y))\uu_1(x)\|^2\\
 +\|z-\psi(x,y)D\uu_0(x)-(1-\psi(x,y))D\uu_1(x)\|^2+g(x,y),
 \end{multline*}
for $x\neq 0$, $(x,y,z)\in \Omega\times Y\times Z$,
and 
\[L(0,y,z)=\|y\|^2+\|z\|^2,\quad (y,z)\in Y \times Z.\]
Observe that on $(x,y)\in\Delta$ the expression \eqref{eq:longdefL} simplifies to
\begin{equation}\label{eq:defL}
 L(x,y,z)=\|y-\uu_i(x)\|^2+\|z-D\uu_i(x)\|^2\quad \textrm{if}\quad i=\operatornamewithlimits{arg\,min}_{j\in\{0,1\}}\|y-\uu_j(x)\|^2
\end{equation}
because $g$ vanishes on $\Delta$ and because of Lemma \ref{lem:psi}.
\begin{lemma}\label{lem:regularityL}
 $L$ is of class $C^{1,1}_{\mathrm{loc}}$.
\end{lemma}

\begin{proof}[Proof of Lemma \ref{lem:regularityL}]
From Lemma \ref{lem:psi}, we know that $U$ and $V$ are $C^\infty$ on $(\Omega\setminus\{0\})\times Y$.
This, together with the expression \eqref{eq:longdefL} defining $L$ away from the origin, and the smoothness of $g$, we conclude that $L$ is $C^\infty$ on $(\Omega\setminus\{0\})\times Y$. 
For fixed $y'\in Y$ and $z'\in Z$, as $(x,y,z)\to(0,y',z')$, using the estimates from Lemma \ref{lem:psi} as well as the fact that 
\[g(x,y)=\begin{cases}O(\|x\|^2+\|y\|^2),& y'\neq 0,\\
0&y'=0,\end{cases}\] 
as $(x,y)\to(0,y')$ (which follows from $g$ being smooth, nonnegative, and vanishing at the origin, $g(0,0)=0$, because then necessarily $\nabla g(0,0)=0$; and from $g(x,y)=0$ on a neighborhood of every point $(0,y)$, $y\neq 0$, as this point belongs to the closure $ \overline\Delta$ and, in a small-enough neighborhood of $(0,y)$, $\Delta$ is dense), we have
\begin{align*}
  |L(x,y,z)&-L(0,y',z')-2\langle y',y-y' \rangle-2\langle z',z-z' \rangle|  \\
  &=|\|y-U(x,y)\|^2+\|z-V(x,y)\|^2+g(x,y)\\
  &\qquad -\|y'\|^2-\|z'\|^2-2\langle y',y-y' \rangle-2\langle z',z-z'\rangle|\\
  &= |\|y\|^2-\|y'\|^2-2\langle y',y-y' \rangle\\
  &\qquad+\|U(x,y)\|^2-2\langle y,U(x,y)\rangle\\
  &\qquad +\|V(x,y)\|^2-2\langle z, V(x,y)\rangle+g(x,y)\\
  &\qquad +\|z\|^2-\|z'\|^2-2\langle z',z-z'\rangle|\\
  &\leq \|y-y'\|^2%
  +O(\|x\|^6+\|y\|\|x\|^3\\
  &\qquad +\|x\|^4+\|z\|\|x\|^2+\chi_{y'\neq0}(\|x\|^2+ \|y\|^2))\\
  &\qquad +\|z-z'\|^2\\
  &\leq O(\|x\|^2+\|y-y'\|^2+\|z-z'\|^2)\\
  &= O(\|(x,y,z)-(0,y',z')\|^2).
\end{align*}
Here, $\chi_{y'\neq 0}\in\{0,1\}$ vanishes when $y'=0$ and is 1 otherwise.
Then \cite[Proposition 4.11.3]{fathi2008weak} implies that the derivative is locally Lipschitz continuous.
\end{proof}

\begin{proof}[Proof of the theorem]
We present the proof in several steps.

\noindent\textbf{Step 1.} $\displaystyle\min_{(\mu,\mu_\partial)\in \mathcal M}\int_{\Omega\times Y\times Z}L\,d\mu=0.$

The map $f$ can be encoded using the measure $\mu$ on $\Omega\times Y\times Z$ defined by  the pushforwards
\[\mu=\tfrac12\xi_{0\#}dx+\tfrac12\xi_{1\#}dx\]
where $dx$ is 
{the} Lebesgue measure on $\Omega$, and $\xi_i\colon\Omega\to\Omega\times Y\times Z$ is the map
\[\xi_i(x)=(x,\uu_i(x),D\uu_i(x)),\quad x\in \Omega.\]
The corresponding boundary measure $\mu_\partial$ is uniquely determined by $\mu$ and is given by $\mu_\partial=\frac12\zeta_{0\#}\sigma(x)+\frac12\zeta_{1\#}\sigma(x)$, where $\sigma$ is the uniform measure on the unit circle with mass $2\pi$ and $\zeta_i(x)=(x,u_i(x))$. The pair $(\mu,\mu_\partial)$ is in $\mathcal M$.

By the definition \eqref{eq:Delta} of $\Delta$, it holds that $(x,\uu_i(x)) \in \Delta$ for $i \in \{0,1\}$. Therefore the $(x,y)$-marginal of $\mu$ is supported in $\Delta$, where $L$ is given by \eqref{eq:defL} (see also Figure~\ref{fig:5}).
It follows that
\begin{align*}
    \int_{\Omega\times Y\times Z} L\,d\mu = \int_{\Delta \times Z} L\,d\mu 
    &=\frac12\sum_{i=0}^1\int_{\Omega}L(x,\uu_i(x),D\uu_i(x))\,dx\\
    &=\frac12\sum_{i=0}^1 \int_{\Omega}\|\uu_i(x)-\uu_i(x)\|^2+\|D\uu_i(x)-D\uu_i(x)\|^2\,dx=0,
\end{align*}
Since the integrand is nonnegative, this is  the minimum of the integral of $L$ over any measure $\mu$ with $(\mu,\mu_\partial)\in\mathcal M$.

\noindent\textbf{Step 2.} \emph{Reparameterization of $f$ using $\bar \uu_\alpha$ and choice of $\alpha_0$.}

For $\alpha \in \R$, let $\bar \uu_\alpha$ be the $\R^2$-valued function on $\Omega$ given by
\begin{equation}\label{eq:ui}
 \bar \uu_\alpha(r\cos (\theta+\alpha),r\sin(\theta+\alpha))=%
 r^3\left(\cos \frac{\theta+\alpha}2,\sin  \frac{\theta+\alpha}2\right),\quad r\in [0,1),\,\theta\in [0,2\pi), 
\end{equation}
so that $\bar \uu_\alpha=-\bar \uu_{\alpha+2\pi}$. 
Thus if $\alpha\in [0,2\pi)$ then
\begin{equation}\label{eq:ubarexplain}
 \bar \uu_\alpha(x)=\begin{cases}
 \uu_1(x),&\textrm{for $\theta(x)\in [\alpha,2\pi)$},\\
 \uu_0(x),&\textrm{for $\theta(x)\in [0,\alpha)$},
\end{cases}\quad  
\bar \uu_{\alpha+2\pi}(x)=\begin{cases}
 \uu_0(x),&\textrm{for $\theta(x)\in [\alpha,2\pi)$},\\
 \uu_1(x),&\textrm{for $\theta(x)\in [0,\alpha)$},
\end{cases}
\end{equation}
where $x\in\Omega$ and $\theta(x)\in[0,2\pi)$ is the polar angle of $x=r(\cos\theta(x),\sin\theta(x))$.  Therefore $\uu_0=\bar \uu_0$ and $\uu_1=\bar \uu_{2\pi}$. Just like $\uu_0$ and $\uu_1$ parameterize the image of $f$ and the jump between the two happens at angle 0 (see Figure \ref{fig:6}), for each $\alpha\in\R$ the functions $\bar \uu_{\alpha}$ and $\bar \uu_{\alpha+2\pi}=-\bar \uu_{\alpha}$ give another parametrization of the image of $f$, with the jump from one chart $\uu_1$ to the other $\uu_0$ at angle $\alpha$.

    \begin{figure}
 \includegraphics[width=0.7\textwidth]{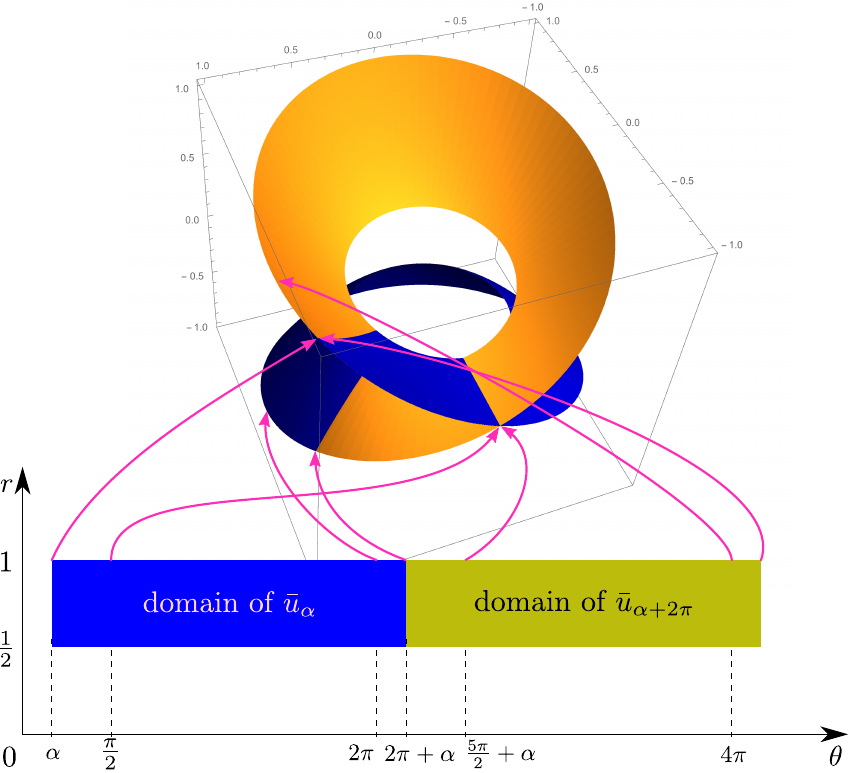}
 \centering
 \caption{This 3-dimensional projection of the graph of $f|_\Gamma$ under the map $(x_1,x_2,y_1,y_2)\mapsto(x_1,x_2,y_1)$ has been colored to distinguish the images of $\uu_\alpha$ and $\uu_{\alpha+2\pi}$. We have also represented the domains, in polar coordinates, of these functions, and indicated where some points are mapped. Note that the apparent self-intersection is an artifact of the projection that does not occur in reality.}
  \label{fig:9}
 \end{figure}

Let $\Gamma\subset \Omega$ be the corona consisting of points $x$ with radius $\frac12\leq |x|\leq 1$, whose area is $|\Gamma|=3\pi/4$.
Take 
\[E=\frac{1}{\denomE}.\]
Let $h\colon\Omega\to Y$ be any function of class $W^{1,2}(\Omega)$, a candidate  solution to the optimization problem on the left-hand side of \eqref{eq:gap}.

    \begin{figure}
 \includegraphics[width=0.65\textwidth]{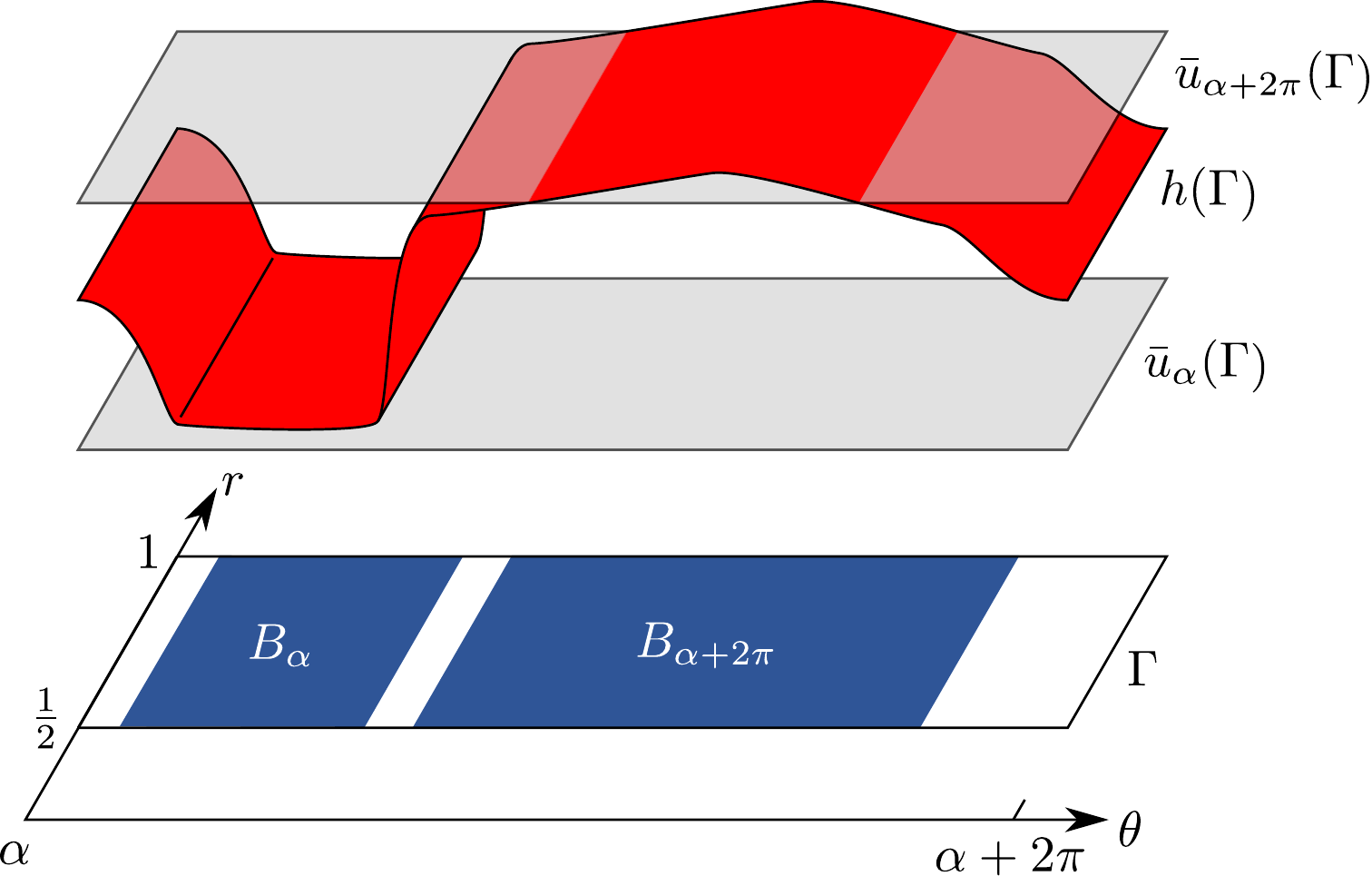}
 \centering
 \caption{In polar coordinates $(\theta,r)$, the corona $\Gamma$ can be parameterized by the rectangle $[\alpha,\alpha+2\pi]\times[\tfrac12,1]$, and its image under $\bar \uu_\alpha$ and $\bar \uu_{\alpha+2\pi}$ is a double-covering. In the picture, we illustrate the definition of the disjoint sets $B_\alpha$ and $B_{\alpha+2\pi}$ for a given function $h$; these sets are the subsets of $\Gamma$ in which $h$ is $E$-close to $\bar \uu_\alpha(\Gamma)$ and $\bar \uu_{\alpha+2\pi}(\Gamma)$, respectively. Changing $\alpha$ translates the picture in the $\theta$ direction. While $h$ is $2\pi$-periodic in $\theta$,  the overall picture is $4\pi$-periodic because the right-hand border $\bar \uu_{\alpha}(\{\alpha+2\pi\}\times[\tfrac12,1])$ coincides with the left-hand border $\bar \uu_{\alpha+2\pi}(\{\alpha\}\times[\tfrac12,1])$, and similarly for  $\bar \uu_{\alpha+2\pi}(\{\alpha+2\pi\}\times[\tfrac12,1])$ and $\bar \uu_{\alpha}(\{\alpha\}\times[\tfrac12,1])$. As we move $\alpha$, the union $B_\alpha\cup B_{\alpha+2\pi}$ does not change, but the contents of the sets $B_\alpha$ and $B_{\alpha+2\pi}$ get gradually interchanged. Although we have drawn $h$, $B_\alpha$ and $B_{\alpha+2\pi}$ as independent of $r$, this need not be the case.}
  \label{fig:8}
 \end{figure}

For $\alpha\in\R$, let $B_\alpha\subset \Omega$ be defined by
\[
B_\alpha = \{x\in\Gamma \,: \, \|h(x)-\bar \uu_\alpha(x)\|\leq E\},
\]
(see Figure \ref{fig:8}). Given $\alpha$, the union $B_\alpha\cup B_{\alpha+2\pi}$ is, because of \eqref{eq:ubarexplain}, the set of points $x\in \Gamma$ such that $h(x)$ is $E$-close to $f(x)$. As the angle $\alpha$  that determines which of those points are in $B_\alpha$ and which are in $B_{\alpha+2\pi}$ varies, the areas of these sets vary continuously; in other words, $\alpha\mapsto |B_\alpha|$ is continuous. Let 
\[\varphi(\alpha)=|B_\alpha|-|B_{\alpha+2\pi}|.\]
Then $\varphi$ is continuous, verifies $\varphi(\alpha)=-\varphi(\alpha+2\pi)$, and is $4\pi$-periodic in $\alpha$. By the intermediate value theorem, there is some $\alpha_0\in[0,2\pi)$ such that  $\varphi(\alpha_0)=0$. In particular, with our choice of $\alpha_0$ we have 
\[|B_{\alpha_0}|=|B_{{\alpha_0}+2\pi}|=\frac{|B_{\alpha_0}|+|B_{{\alpha_0}+2\pi}|}{2}.\]

\noindent\textbf{Step 3.}
\emph{Bound for $h$ when $|B_{\alpha_0}|=|B_{{\alpha_0}+2\pi}|< |\Gamma|/\denomA$.}

Let $j(x)\in \{0,1\}$ be given by%
\[j(x)=\operatornamewithlimits{arg\,min}_{j\in\{0,1\}}\|h(x)-\bar \uu_{{\alpha_0}+2\pi j}(x)\|,\quad x\in\Gamma. \]
We have, from \eqref{eq:grel} and \eqref{eq:ubarexplain},
\[\|h(x)-U(x,h(x))\|^2+g(x,h(x))\geq \|h(x)-\bar \uu_{{\alpha_0} + 2\pi j(x)}(x)\|^2.\]
We use this to get the uniform bound 
\begin{align*}
 \int_\Omega L(x,h(x),Dh(x))dx
 &=\int_\Omega \|h(x)-U(x,h(x))\|^2\\
 &\qquad +\|Dh(x)-V(x,h(x))\|^2+g(x,h(x))\,dx\\
 &\geq\int_\Gamma \|h(x)-U(x,h(x))\|^2+g(x,h(x))\,dx\\
 &\geq \int_\Gamma \|h(x)-\bar \uu_{{\alpha_0} + 2\pi j(x)}(x)\|^2dx\\
 &\geq \int_{\Gamma\setminus(B_{\alpha_0}\cup B_{\alpha_0+2\pi})}E^2\,dx\\
 &\geq  E^2(|\Gamma|-|B_{\alpha_0}|-|B_{{\alpha_0}+2\pi}|)\\
 &>  E^2\left(|\Gamma|-2\frac{|\Gamma|}4\right)=E^2\frac{|\Gamma|}{2}>0.
\end{align*}
We have additionally used the fact that, by our choice of the sets $B_{\alpha_0}$ and $B_{\alpha_0+2\pi}$, for $x$ outside these sets it holds that $\|h(x)-\bar \uu_{\alpha_0+2\pi j(x)}(x)\|\geq E$.
 
\noindent\textbf{Step 4.} \emph{Definition and properties of $\bar h_0$.}

Let,  for $x\in\Gamma$,
\[\bar h_0(x)=\min\left(\|h(x)-\bar \uu_{{\alpha_0}}(x)\|,\;\frac1{\denomxi}\right).
\]
The role of the function $\bar h_0$ is to give a sort of truncated version of the distance from $h$ to $\bar \uu_{{\alpha_0}}$ that will be useful in our estimation below.
Note that, by the definition of $B_{\alpha_0}$, $\bar h_0(x)\leq E$ on $B_{\alpha_0}$.

Observe that if we parameterize $\Gamma$ in polar coordinates with the rectangle $[\frac12, 1]\times [\alpha_0,\alpha_0+2\pi)$, then $\bar \uu_{\alpha_0}$ is smooth on that chart, and $\bar h_0$ is in $W^{1,2}$, as is $h$. In other words, although $\bar h_0$ is possibly discontinuous on the ray segment 
\[R_{\alpha_0}=\{x\in \Gamma:x=r(\cos{\alpha_0},\sin{\alpha_0}), \;r\in[\frac12,1]\}\] 
corresponding to angle $\alpha_0$, it is a Sobolev $W^{1,2}$ function on the rest of $\Gamma$.

\noindent\textbf{Claim.} We have, for almost every $x\in \Gamma\setminus R_{\alpha_0}$,
\begin{equation}\label{eq:estimatederivative}
 \|Dh(x)-V(x,h(x))\|\geq \|D\bar h_0(x)\|.
\end{equation}
\begin{proof}[Proof of the claim.] We have the following cases for $x\in\Gamma\setminus R_{\alpha_0}$:
\begin{itemize}
\item In the region where $\|h(x)-\bar \uu_{{\alpha_0}}(x)\|\geq 1/\denomxi$, the function $\bar h_0$ is constant so the right-hand side equals zero, and the inequality is verified trivially.
\item In the region where $0< \|h(x)-\bar \uu_{{\alpha_0}}(x)\|< 1/\denomxi$, we have $(x,h(x))\in \Delta$ because for $x\in\Gamma$ we have $\frac12\leq \|x\|\leq1$ and then, expanding the squared inequality 
\[0< \|h(x)-\bar \uu_{{\alpha_0}}(x)\|^2=\|h(x)\|^2-2\langle \bar \uu_{{\alpha_0}}(x),h(x)\rangle+\|\bar \uu_{{\alpha_0}}(x)\|^2< \tfrac{1}{\denomxi^2},\] 
we get
\begin{multline*}
\langle\bar \uu_{\alpha_0}(x),h(x)\rangle>\tfrac12(\|h(x)\|^2+\|\bar \uu_{\alpha_0}(x)\|^2-\tfrac1{10^2})\\
 \geq\tfrac12((\|\bar \uu_{\alpha_0}(x)\|-\tfrac1{10})^2+\|\bar \uu_{\alpha_0}(x)\|^2-\tfrac1{10^2})=\|\bar \uu_{\alpha_0}(x)\|^2-\tfrac1{10}\|\bar \uu_{\alpha_0}(x)\|\\
 =\|x\|^6-\tfrac1{10}\|x\|^3>\frac{\|x\|^6}{10}.
\end{multline*}
Since the left-hand side equals $|\langle \uu_i(x),h(x)\rangle|$ for some $i=0,1$ by \eqref{eq:ubarexplain}, this shows that $(x,h(x))\in \Delta$, as per its definition \eqref{eq:Delta}.
This, in turn, means (by Lemma \ref{lem:psi}\ref{it:Vsmooth}) that the left-hand side of \eqref{eq:estimatederivative} reduces to \[\|Dh(x)-D\bar \uu_{\alpha_0}(x)\|\]
by our choice of $\psi$ and \eqref{eq:ubarexplain}.
Inequality  \eqref{eq:estimatederivative} then  follows by taking $\phi(x)=h(x)-\bar \uu_{\alpha_0}(x)$ and observing that all weakly differentiable functions $\phi$ verify, almost everywhere within the set where $\|\phi\|\neq 0$,
\[ \Big\|D\|\phi\|\Big\|=\left\|\frac{\phi}{\|\phi\|}D\phi\right\|=\left\|(D\phi)^t\frac{\phi^t}{\|\phi\|}\right\|\leq\|(D\phi)^t\| \frac{\|\phi\|}{\|\phi\|}=\|D\phi\|,\]
where $(D\phi)^t$ is the transposed Jacobian matrix and $\|(D\phi)^t\|=\|D\phi\|$ is its operator norm.
\item In the region where $ \bar h_0(x)=\|h(x)-\bar \uu_{{\alpha_0}}(x)\|=0$, %
 the weak derivative of $\bar h_0$ vanishes wherever it is defined (because $\bar h_0$ is nonnegative), hence verifying \eqref{eq:estimatederivative}; the set where it is not defined has measure zero because $\bar h_0$ is weakly differentiable since $h$ is.
\end{itemize}
\end{proof}

\noindent\textbf{Step 5.} \emph{Bound for $h$ when} 
\begin{equation}\label{eq:largesets}
|B_{\alpha_0}|=|B_{{\alpha_0}+2\pi}|\geq \frac{|\Gamma|}{\denomA}.
\end{equation}
For $x\in B_{\alpha_0+2\pi}$, $\|h(x)-\bar \uu_{\alpha_0+2\pi}\|\leq E$, so 
\[
 \|\bar \uu_{\alpha_0}(x)-\bar \uu_{\alpha_0+2\pi}(x)\|
 \leq\|\bar \uu_{\alpha_0}(x)-h(x)\|+\|h(x)-\bar \uu_{\alpha_0+2\pi}(x)\|\leq \|\bar \uu_{\alpha_0}(x)-h(x)\|+E,
\]
and then
\[\|h(x)-\bar \uu_{\alpha_0}(x)\|>\|\bar\uu_{\alpha_0}(x)-\bar \uu_{\alpha_0+2\pi}(x)\|-E\geq \frac18-E>\frac1{10},\]
so $\bar h_0(x)=1/10$.
Hence, using \eqref{eq:largesets},
\[M\coloneqq |\Gamma|^{-1}\int_\Gamma \bar h_0(x)dx\geq|\Gamma|^{-1}\int_{B_{\alpha_0+2\pi}} \frac 1{\denomxi} dx\geq\frac{|B_{\alpha_0+2\pi}|}{\denomxi|\Gamma|}\geq\frac{1}{ \the\numexpr\denomxi*\denomA\relax}.%
\]
The domain consisting of the slit corona $\Gamma\setminus R_{\alpha_0}$ satisfies the so-called \emph{cone property} \cite[Definition 2.5.14]{gasipapa}.
By \eqref{eq:estimatederivative} together with the Poincar\'e-Wirtinger inequality \cite[Theorem 2.5.21]{gasipapa} for the domain $\Gamma\setminus R_{\alpha_0}$ with constant $C>0$,
\begin{align*}
    \int_\Omega L(x,h(x),Dh(x))dx&=\int_\Omega \|h(x)-U(x,h(x))\|^2 \\
    &\qquad +\|Dh(x)-V(x,h(x))\|^2+g(x,h(x))\,dx\\
    &\geq \int_\Gamma  \|Dh(x)-V(x,h(x))\|^2dx\\
    &\geq \int_{\Gamma\setminus R_{\alpha_0}}  \|D\bar h_0(x)\|^2dx\\
    &\geq C \int_{\Gamma\setminus R_{\alpha_0}}\left|\bar h_0(x)-M\right|^2dx.
\end{align*}
Now, since for $x\in B_{\alpha_0}$ we have $0\leq \bar h_0(x)\leq  E=1/\denomE<1/{ \the\numexpr\denomxi*\denomA\relax}\leq M$ there,  we have
\[|\bar h_0(x)-M|\geq F\coloneqq\frac{1}{ \the\numexpr\denomxi*\denomA\relax}-\frac{1}{ \denomE}>0,\]
so the above is
\begin{equation*}
    C \int_{\Gamma\setminus R_{\alpha_0}}\left|\bar h_0(x)-M\right|^2dx\geq C
    \int_{B_{\alpha_0}}   F^2dx= CF^2 |B_{\alpha_0}|\geq CF^2\frac{|\Gamma|}{\denomA}>0.
\end{equation*}
This is a uniform lower bound for all $h\in W^{1,2}(\Omega)$ satisfying the above constraints. 

Together, the bounds from Steps 3 and 5 prove the theorem.
\end{proof}

\section{Positive gap with integral constraints}\label{sec:integralconstraints} 

 The reader may be curious why we have not included, in the statements of Theorems \ref{thm:consolidated} and \ref{thm:nogap} and in the definitions \eqref{opt:classical} and \eqref{opt:relaxed} of $M_\mathrm{c}$ and $M_\mathrm{r}$, any integral constraints of the form
 \[\int_\Omega H(x,y(x),Dy(x))\,dx \le 0\quad\textrm{or}\quad  \int_\Omega H(x,y(x),Dy(x))\,dx = 0.\] 
 The reason is that  in the presence of these constraints, there may be a gap between the classical case and its relaxation. The following two sections give examples of such situations. The idea for each of these examples works in any dimensions $n,m>0$, and we show them in the special case $n=m=1$ for simplicity.
 
 We use the same notations as in the definitions \ref{opt:classical} and \ref{opt:relaxed} of $M_\mathrm{c}$ and $M_\mathrm{r}$.
 
 \subsection{Inequality integral constraints}
 
 Let $\Omega=(0,1)\subset\R$, $Y=\R=Z$, $L(x,y,z)=y$, $F(x,y,z)=y(1-y)$, $F_\partial=G=G_\partial=0$. Note that the only Lipschitz curves $y\colon\Omega\to Y$  such that $F(x,y(x),Dy(x))=0$ are $y_0(x)=0$ and $y_1(x)=1$, $x\in\Omega$, and these satisfy
 \[\int_{\Omega}L(x,y_0(x),Dy_0(x))dx=\int_0^10\,dx=0\quad\textrm{and}\quad\int_{\Omega}L(x,y_1(x),Dy_1(x))dx=\int_0^11\,dx=1.\]
 Let $H(x,y,z)=1-10y$. Consider the problem of computing $M_\mathrm{c}$ and $M_\mathrm{r}$ as in \ref{opt:classical} and \ref{opt:relaxed}, above, with the additional integral constraints
 \[\int_\Omega H(x,y(x),Dy(x))\,dx\leq 0\quad\textrm{and}\quad \int_{\Omega\times Y\times Z}H(x,y,z)\,d\mu(x,y,z)\leq 0,\]
 to be satisfied by the respective competitors $y$ and $(\mu,\mu_\partial)$.
 In other words, we have
\begin{alignat*}{2}
 M_\mathrm{c}= &\!\inf\limits_{y \in W^{1,\infty}((0,1);\R)}  &\quad &\displaystyle\int_0^1 y(x)\, dx \\ 
 &\textrm{subject to}& & y(x)\in\{ 0,1\},\quad \text{a.e. }x\in \Omega,\nonumber \\
 &&& \textstyle\int_0^1 (1-10\,y(x))\,dx\leq 0,
\end{alignat*}
 and
 \begin{alignat*}{2}
 M_\mathrm{r}=&\!\inf\limits_{(\mu,\mu_\partial)\in \mathcal M}  &\quad&\displaystyle\int_{\Omega\times Y\times Z} y\, d\mu(x,y,z) \\ 
 &\textrm{subject to} %
&& \operatorname{supp}\mu\subset (0,1)\times \{0,1\}\times \R,\nonumber\\ 
&&&\textstyle\int_{\Omega\times Y\times Z}(1-10\,y)\,d\mu(x,y,z)\leq 0.
\end{alignat*}
 We will show that in this case $M_\mathrm{c}>M_\mathrm{r}$. 
 
 For the classical case, we have
 \[\int_{\Omega}H(x,y_0,Dy_0)\,dx=\int_0^11-0\,dx=1\stackrel{!}>0\quad\textrm{and}\quad\int_{\Omega}H(x,y_1,Dy_1)\,dx=\int_0^11-10\,dx=-9\leq 0,\]
 so in the calculation of $M_\mathrm{c}$ the only competitor is $y_1$, because $y_0$ does not satisfy the integral constraint. We conclude that $M_\mathrm{c}=1$.
 
 For the relaxed case, consider the measure  $\mu=\frac9{10}\mu_0+\frac1{10}\mu_1$, where $\mu_i$ is the measure induced by $y_i$, $i=0,1$. Then 
 \[\int_{\Omega\times Y\times Z}H\,d\mu=\tfrac9{10}\int_{\Omega\times Y\times Z}H\,d\mu_0+\tfrac1{10}\int_{\Omega\times Y\times Z}H\,d\mu_1=\tfrac{9}{10}-9\tfrac1{10}=0,\]
 so  $\mu$ satisfies the constraint. We also have
 \[\int_{\Omega\times Y\times Z}L\,d\mu=\tfrac9{10}\int_{\Omega\times Y\times Z}L\,d\mu_0+\tfrac1{10}\int_{\Omega\times Y\times Z}L\,d\mu_1=\tfrac1{10}.\]
 Thus $M_\mathrm{r}\leq \frac1{10}<1=M_\mathrm{c}$.

\subsection{Equality integral constraints}

Let $\Omega=(0,1)\subset\R$, $Y=\R=Z$, $L(x,y,z)=y$, $F(x,y,z)=y(y-1)(y-2)$, $F_\partial=G=G_\partial=0$. Note that the only Lipschitz curves $y\colon\Omega\to Y$  such that $F(x,y(x),Dy(x))=0$ are $y_0(x)=0$, $y_1(x)=1$, and $y_2(x)=2$, $x\in\Omega$, and these satisfy
 \[\int_{\Omega}L(x,y_i(x),Dy_i(x))dx=\int_0^1i\,dx=i,\quad i=0,1,2.\]
 Let $H(x,y,z)=\frac74y-\frac34y^2$. Consider the problem of computing $M_\mathrm{c}$ and $M_\mathrm{r}$ as in \eqref{opt:classical} and \eqref{opt:relaxed}, above, with the additional integral constraints
 \[\int_\Omega H(x,y(x),Dy(x))\,dx= \tfrac12\quad\textrm{and}\quad \int_{\Omega\times Y\times Z}H(x,y,z)\,d\mu(x,y,z)=\tfrac12,\]
 to be satisfied by the respective competitors $y$ and $(\mu,\mu_\partial)$. In other words, we have
\begin{alignat*}{2}
 M_\mathrm{c}= &\!\inf\limits_{y \in W^{1,\infty}((0,1);\R)}  &\quad &\displaystyle\int_0^1 y(x)\, dx \\ 
 &\textrm{subject to}& & y(x)\in\{ 0,1,2\},\quad \text{a.e. }x\in \Omega,\nonumber \\
 &&& \textstyle\int_\Omega (\tfrac74y(x)-\tfrac34y(x)^2)\,dx=\tfrac12,
\end{alignat*}
 and
 \begin{alignat*}{2}
 M_\mathrm{r}=&\!\inf\limits_{(\mu,\mu_\partial)\in \mathcal M}  &\quad&\displaystyle\int_{\Omega\times Y\times Z} y\, d\mu(x,y,z) \\ 
 &\textrm{subject to} %
&& \operatorname{supp}\mu\subset \Omega\times \{0,1,2\}\times Z,\nonumber\\ 
&&&\textstyle\int_{\Omega\times Y\times Z}(\tfrac74y-\tfrac34y^2)\,d\mu(x,y,z)=\tfrac12.
\end{alignat*}
We will show that in this case $M_\mathrm{c}>M_\mathrm{r}$ too.
 
 For the classical case, we have
 \[\int_{\Omega}H(x,y_0,Dy_0)\,dx=\int_0^10\,dx=0\stackrel{!}\neq\tfrac12 \quad\textrm{and}\quad\int_{\Omega}H(x,y_1,Dy_1)\,dx=\int_0^1\tfrac74-\tfrac34\,dx=1\stackrel!\neq \tfrac12,\]
 and
 \[\int_{\Omega}H(x,y_2,Dy_2)\,dx=\int_0^12\tfrac74-4\tfrac34dx=\tfrac12,\]
 so in the calculation of $M_\mathrm{c}$ the set of competitors contains only $y_2$. We conclude that \[M_\mathrm{c}=\int_0^1L(x,y_2(x),Dy_2(x))\,dx=\int_0^12\,dx=2.\]
 For the relaxed case, consider the measure  $\mu=\frac12\mu_0+\frac12\mu_1$, where $\mu_i$ is the measure induced by $y_i$, $i=0,1$. Then 
 \[\int_{\Omega\times Y\times Z}H\,d\mu=\tfrac12\int_{\Omega\times Y\times Z}H\,d\mu_0+\tfrac12\int_{\Omega\times Y\times Z}H\,d\mu_1=0\tfrac{1}{2}+1\tfrac12=\tfrac12,\]
 so $\mu$ satisfies the constraint. We also have
 \[\int_{\Omega\times Y\times Z}L\,d\mu=\tfrac12\int_{\Omega\times Y\times Z}L\,d\mu_0+\tfrac12\int_{\Omega\times Y\times Z}L\,d\mu_1=\tfrac12.\]
 Thus $M_\mathrm{r}\leq \frac12<2=M_\mathrm{c}$.

\section{Acknowledgements}
This work was co-funded by the European Union under the ROBOPROX project (reg.~no. CZ.02.01.01/00/22 008/0004590) and by the Prime Minister’s Office, Singapore under its Campus for Research Excellence and Technological Enterprise (CREATE) programme.
This work was partially supported by UiT Aurora Center for Mathematical
Structures in Computations MASCOT. It was also partially supported by the project Pure Mathematics in Norway funded by Trond Mohn Foundation and Tromsø Research Foundation. It was also partially supported by ANR--3IA Artificial and Natural Intelligence Toulouse Institute [ANR--19--PI3A--0004].

The authors would like to thank Martin Kru\v z\'ik, Jared Miller, Corbinian Schlosser and Ian Tobasco for constructive feedback on this work.

The authors would like to acknowledge the valuable effort of Alessandro Vici, whose masters thesis \cite{vici} focused on this work before its final publication, and included comments that helpfully pointed out minor errors in Lemmas \ref{lem:dividentity}, \ref{lem:function}, and \ref{lem:weakderivative}.

\paragraph{Data availability statement.} This manuscript has no associated data.
\paragraph{Competing interests.} The authors have no competing interests.

\bibliography{bibl}{}
\bibliographystyle{plain}

\end{document}